%% file: main.tex
\documentclass[a4paper, 11pt]{amsart}
\usepackage[utf8]{inputenc}

\pdfoutput=1

\usepackage{amsfonts}
\usepackage{amssymb}
\usepackage{amsmath}
\usepackage{amsthm} 
\usepackage{mathtools} 
\usepackage{etoolbox}

\usepackage[nopatch=footnote]{microtype}
\usepackage{enumitem} 
\numberwithin{equation}{section} 

\usepackage{tikz}
\usetikzlibrary{arrows.meta, positioning}
\usepackage[noadjust,space,nosort]{cite}
\usepackage{hyperref} 

\newtheorem{theorem}{Theorem}[section]
\newtheorem{proposition}[theorem]{Proposition}
\newtheorem{lemma}[theorem]{Lemma}
\newtheorem{corollary}[theorem]{Corollary}
\theoremstyle{definition}

\newtheorem{example}{Example}[section]
\newtheorem{remark}{Remark}[section]

\newtheorem{assumptiondecay}{Condition}
\newtheorem{assumptionmeasure}{Condition}
\newtheorem{assumptionsequence}{Condition}
\newtheorem{assumptionshortreturns}{Condition}
\newenvironment{assumptiond}[1]{
  \renewcommand\theassumptiondecay{#1}
  \assumptiondecay
}{\endassumptiondecay}
\newenvironment{assumptionm}[1]{
  \renewcommand\theassumptionmeasure{#1}
  \assumptionmeasure
}{\endassumptionmeasure}
\newenvironment{assumptions}[1]{
  \renewcommand\theassumptionsequence{#1}
  \assumptionsequence
}{\endassumptionsequence}
\newenvironment{assumptionsr}[1]{
  \renewcommand\theassumptionshortreturns{#1}
  \assumptionshortreturns
}{\endassumptionshortreturns}

\newcounter{proofcase}
\AtBeginEnvironment{proof}{\setcounter{proofcase}{0}}
\newcommand{\case}[1][]{%
  \par\addvspace{\medskipamount}%
  \refstepcounter{proofcase}%
  \noindent
  \ifblank{#1}
    {\textit{Case \theproofcase.}}
    {\textit{Case \theproofcase: #1.}}%
  \enspace}

\newcounter{proofcount}
\newcounter{proofstep}
\AtBeginEnvironment{proof}{%
  \stepcounter{proofcount}{0}%
  \setcounter{proofstep}{0}%
}

\newcommand{\step}[1][]{%
  \par\addvspace{\medskipamount}%
  \refstepcounter{proofstep}%
  \noindent
  \ifblank{#1}
    {\textit{Step \theproofstep.}}
    {\textit{Step \theproofstep: #1.}}%
  \enspace}

\newcommand{\conclusion}{%
  \par\addvspace{\medskipamount}%
  \noindent\textit{Conclusion.}\enspace}

\newcommand{\naturals}{\mathbb{N}}
\newcommand{\reals}{\mathbb{R}}
\newcommand{\complexes}{\mathbb{C}}
\newcommand{\proj}{\mathbb{P}}
\newcommand{\torus}{\mathbb{T}}
\newcommand{\bcal}{\mathcal{B}}
\newcommand{\dcal}{\mathcal{D}}
\newcommand{\ecal}{\mathcal{E}}
\newcommand{\fcal}{\mathcal{F}}
\newcommand{\gcal}{\mathcal{G}}
\newcommand{\hcal}{\mathcal{H}}
\newcommand{\mcal}{\mathcal{M}}

\newcommand{\pcal}{\mathcal{P}}
\newcommand{\qcal}{\mathcal{Q}}
\newcommand{\scal}{\mathcal{S}}
\newcommand{\diff}{\mathop{}\!d}
\newcommand{\leqs}{\lesssim}
\newcommand{\geqs}{\gtrsim}
\newcommand{\eqs}{\asymp}
\newcommand{\coleq}{\coloneqq}
\newcommand{\norm}[1]{\lVert #1 \rVert}
\newcommand{\bignorm}[1]{\biggl\lVert #1 \biggr\rVert}
\newcommand{\normsp}[2]{\norm{#1}_{#2}}
\newcommand{\lipnorm}[1]{\norm{#1}_{\textup{Lip}}}

\newcommand{\supnorm}[1]{\norm{#1}_{\infty}}
\newcommand{\lnorm}[2]{\norm{#1}_{L^{#2}}}
\newcommand{\biglnorm}[2]{\bignorm{#1}_{L^{#2}}}
\newcommand{\holnorm}[2]{\norm{#1}_{#2}}
\newcommand{\holconst}[2]{\lvert#1\rvert_{#2}}
\newcommand{\lipconst}[1]{\lvert#1\rvert_{\textup{Lip}}}
\newcommand{\EE}{\mathbf{E}}
\newcommand{\CE}[2]{\EE[#1 | #2]}
\newcommand{\charfun}{\mathbf{1}}
\newcommand{\prob}{\mathbf{P}}
\DeclareMathOperator{\var}{Var}
\DeclareMathOperator{\covar}{Cov}
\DeclareMathOperator{\supp}{supp}
\DeclareMathOperator{\dist}{dist}
\newcommand{\bigo}{\mathcal{O}}
\newcommand{\smallo}{o}
\newcommand{\borel}{\bcal}
\newcommand{\bv}{\mathrm{BV}}
\newcommand{\card}{\#}
\newcommand{\todistr}{\xrightarrow[]{\dcal}}

\newcommand{\tomeas}[1]{\xrightarrow[]{#1}}

\newcommand{\partition}{\qcal}
\newcommand{\diam}[1]{\lvert #1 \rvert}
\newcommand{\hol}[1]{\hcal_{#1}}
\newcommand{\haus}{H}
\newcommand{\Ell}[1]{L^{#1}}
\newcommand{\hlmax}{\mathfrak{M}}
\newcommand{\leb}{m}

\DeclareMathOperator{\totalvar}{var}
\newcommand{\sft}{\Sigma}
\newcommand{\shiftop}{\sigma}
\newcommand{\sftpos}{\Sigma^+}
\newcommand{\sigmapos}{\sigma}
\newcommand{\sftfun}{\fcal}
\newcommand{\sftnorm}[2]{\holnorm{#1}{#2}}

\newcommand{\spgap}{\tau_0}
\newcommand{\doc}{\spgap}
\newcommand{\mdoc}{\tau_r}
\newcommand{\frost}{s_0}
\newcommand{\thinan}{\alpha_0}
\newcommand{\thinanr}{\rho_0}
\newcommand{\srelog}{\upsilon_0}
\newcommand{\srein}{\delta_0}
\newcommand{\sreout}{\eta_0}
\newcommand{\sremin}{\delta}
\newcommand{\srilog}{\srelog}
\newcommand{\sriin}{\srein}
\newcommand{\sriout}{\sreout}
\newcommand{\srimin}{\sremin}
\newcommand{\seqlow}{\gamma} 
\newcommand{\seqhigh}{\upsilon}
\newcommand{\frostan}{\omega_0}
\newcommand{\lexmeas}{\eta_1}
\newcommand{\lexcor}{\eta_2}
\newcommand{\lexcorcoeff}{q_0}
\newcommand{\lexcorsum}{p_0}
\newcommand{\limmeas}{\lexmeas}
\newcommand{\limcor}{\lexcor}
\newcommand{\limcorcoeff}{\lexcorcoeff}
\newcommand{\limcorsum}{\lexcorsum}
\newcommand{\lsre}{\varepsilon}
\newcommand{\lsri}{\varepsilon}

\newcommand{\exE}{E}

\newcommand{\exAv}{\Theta}
\newcommand{\exsig}{\sigma}
\newcommand{\exs}{s}
\newcommand{\exesum}{\ecal}
\newcommand{\exmsum}{\asipmsum}
\newcommand{\imr}{\hat{r}}
\newcommand{\imE}{\hat{E}}
\newcommand{\imM}{M}
\newcommand{\imsig}{\hat{\sigma}}
\newcommand{\ims}{\hat{s}}
\newcommand{\imesum}{\hat{\ecal}}
\newcommand{\immsum}{\mcal}
\newcommand{\abr}{\chi}
\newcommand{\asipvar}{\zeta} 
\newcommand{\asipvarA}{\hat{\zeta}} 

\newcommand{\asipvarlem}{\asipvar}
\newcommand{\abstractvar}{\xi} 
\newcommand{\asipmeas}{b}
\newcommand{\asipmsum}{\mcal}
\newcommand{\sftvar}{\zeta}
\newcommand{\sftvarpos}{\hat{\zeta}}

\newcommand\ptmark[1]{\draw (#1,-0.1) -- (#1,0.1)}

\title[Limit Laws for Recurrence and Shrinking Targets]{Limit Laws for
  Poincar\'e Recurrence and the Shrinking Target Problem}
\date{June 14, 2026}
\author{Alejandro Rodriguez Sponheimer}
\address{Centre for Mathematical Sciences,
Lund University, Box~118, 221~00 Lund, Sweden}
\email{alejandro.rodriguez\_sponheimer@math.lth.se}



\begin{document}

\begin{abstract}
  We establish distributional laws for Poincar\'e recurrence in
  measure-preserving systems $(X,T,\mu)$ satisfying an exponential
  multiple decorrelation condition and a short returns condition.
  When the measure is absolutely continuous, the sum
  $\sum_{k=1}^{n} \charfun_{B(x,r_k)}(T^{k}x) - \mu(B(x,r_k))$
  \emph{does not} in general obey a CLT; instead, it converges to a
  non-standard distribution that is an \emph{average} of Gaussian laws
  weighted by the density of $\mu$.
  By considering a version of the sum where we appropriately rescale
  the radii of the balls, we recover the CLT.
  A key assumption in our recurrence theorems is that the corresponding
  hitting sums satisfy the CLT.
  We verify this assumption for Axiom~A systems by establishing the
  stronger ASIP for the shrinking target problem,
  extending Haydn, Nicol, T{\"o}r{\"o}k and Vaienti
  [Trans.\ Amer.\ Math.\ Soc.\ 2017]
  and related results.
  Systems for which our results apply include piecewise expanding
  systems on the interval, and Axiom A systems.
  The results highlight the difference between recurrence and
  hitting behaviour.
\end{abstract}

\maketitle

\tableofcontents

\section{Introduction}

\subsection{Background and motivation}
A fundamental result in dynamical systems is the
Poincar\'e Recurrence Theorem, which gives conditions for when a given
point returns infinitely often to a shrinking neighbourhood of its
initial position.
More precisely, if $(X,T,\mu,d)$ is a separable metric
measure-preserving system with a finite Borel measure $\mu$,
then $\mu$-a.e.\ point $x \in X$ is
\emph{recurrent}, meaning that
$\liminf_{n \to \infty} d(x,T^{n}x) = 0$.
While this theorem offers a general qualitative conclusion under
minimal assumptions, it leaves open the question of \emph{quantitative
recurrence}.
For example, can one~-- perhaps under additional conditions~--
establish the rate at which $T^{n}x$ returns to $x$?
This question has motivated a significant body of recent work in
quantitative recurrence.

The first result establishing a rate of convergence of
$\liminf_{n \to \infty}d(x,T^{n}x) = 0$ is due to Boshernitzan
\cite{boshernitzan-1993-quantitative} in \textup{1993} and, similar to
the Poincar\'e Recurrence theorem, has minimal assumptions. Namely,
under the additional assumption that there exists $\alpha > 0$ such
that $X$ is $\sigma$-finite with respect to the Hausdorff measure
$\haus_\alpha$, one has that
\begin{equation}
\label{eq:boshernitzan}
  \liminf_{n \to \infty} n^{-1/\alpha} d(x,T^{n}x) < \infty
  \quad \text{$\mu$-a.e.\ $x$}.
\end{equation}
Following Boshernitzan's work, Barreira and Saussol
\cite{barreira-2001-hausdorff} and Urb\'anski
\cite{urbanski-2007-recurrence} obtain a connection between recurrence
rates and the pointwise dimension of the measure for various systems.

Another way of viewing \eqref{eq:boshernitzan} is the following: for
$\mu$-a.e.\ $x$, there exists $C(x) > 0$ such that
\[
  \sum_{k=1}^{\infty} \charfun_{B(x,C(x) k^{-1/\alpha})}(T^{k}x)
  = \infty.
\]
More generally, we may consider the sum
\begin{equation}
\label{eq:boshernitzan:sum}
  \sum_{k=1}^{n} \charfun_{B(x, \imr_k(x))}(T^{k}x),
\end{equation}
for more general sequences of radii $\imr_n(x)$. Recently,
the study of the asymptotic behaviour of $\eqref{eq:boshernitzan:sum}$
has received substantial attention. This has resulted in
Borel--Cantelli lemmas
\cite{
  dolgopyat-2022-multiple,
  he-2023-quantitative,
  hussain-2022-dynamical,
  kirsebom-2023-shrinking,
  kleinbock-2023-dynamical
},
their strong versions
\cite{
  he-2024-quantitative,
  huang-2025-exponential,
  levesley-2024-shrinking,
  persson-2023-strong,
  persson-2026-strong,
  ars-2025-recurrence
},
as well as a Poisson-like law
\cite{holland-2025-distributional}.

We note the similarity with the more familiar \emph{shrinking target
problem}, in which one studies the hitting version of
\eqref{eq:boshernitzan:sum}, namely
\begin{equation}
\label{eq:hitting:sum}
  \sum_{k=1}^{n} \charfun_{B(y_k, r_k)}(T^{k}x),
\end{equation}
where $B(y_k,r_k)$ is a fixed sequence of balls.
The first result establishing a strong Borel--Cantelli lemma for
\eqref{eq:hitting:sum} dates back to Philipp~\cite{philipp-1967-some},
who establishes the result for $\times n$-maps,
$\beta$-transformations, and the Gauss map.
A particularly relevant result for hitting is that of Chernov and
Kleinbock~\cite{chernov-2001-dynamical}, who prove a strong
Borel--Cantelli lemma for Axiom~A diffeomorphisms. Recently, the
central limit theorem and almost sure invariance principle have been
considered
\cite{haydn-2013-central,haydn-2017-almost,chen-2018-non-stationary}.
Due to the similarity of the expressions \eqref{eq:boshernitzan:sum}
and \eqref{eq:hitting:sum}, one may expect that recurrence and hitting
satisfy similar (if not the same) limit laws.

In the present article, we focus on distributional laws for recurrence
and their relationship to the corresponding laws for hitting.
We give a brief overview of our results, delaying the precise
statements for later.
\begin{enumerate}[label=\Roman*.,ref=\Roman*]
  \item
    \label{enum:1}
    \textbf{Averaged limit theorem for recurrence.}
    We prove that, whenever the measure $\mu$ is absolutely continuous
    and satisfies suitable mixing conditions,
    \begin{equation}
    \label{eq:intro:result:recurrence:1}
      \sum_{k=1}^{n} \charfun_{B(x,r_k)}(T^{k}x) - \mu(B(x,r_k))
    \end{equation}
    \emph{does not} obey a CLT. Instead, it obeys a distributional law
    that is in some sense an average of Gaussian ones.
  \item
    \label{enum:2}
    \textbf{Central limit theorem for recurrence.}
    By instead considering balls $B(x,\imr_k(x))$ whose measures do not
    depend on $x$, we recover the CLT. That is, we prove that, under
    suitable mixing and regularity conditions,
    \begin{equation}
    \label{eq:intro:result:recurrence:2}
      \sum_{k=1}^{n} \charfun_{B(x,\imr_k(x))}(T^{k}x)
      - \mu\bigl(B(x,\imr_k(x))\bigr)
    \end{equation}
    \emph{does} obey a CLT.
  \item
    \label{enum:3}
    \textbf{Almost sure invariance principle for hitting.}
    We prove that, for Axiom~A systems equipped with Gibbs measures, 
    \begin{equation}
    \label{eq:intro:result:asip}
      \sum_{k=1}^{n} \charfun_{B(y_k,r_k)}(T^{k}x) - \mu(B(y_k,r_k))
    \end{equation}
    obeys an ASIP (and thus a CLT).
\end{enumerate}

Result~\ref{enum:1} is the first to establish a distributional result
for the sum \eqref{eq:intro:result:recurrence:1}.
It highlights the difference between the recurrence and hitting
problems, similar to other recent limit laws for recurrence 
\cite{
  he-2024-quantitative,
  holland-2025-distributional,
  huang-2025-exponential
}.
Result~\ref{enum:2} recovers the `expected' limit law by considering a
variation of the first problem, similar to what is done in
\cite{
  persson-2023-strong,
  ars-2025-recurrence,
  persson-2026-strong,
  huang-2025-exponential
}.
This version also allows us to consider more general measures that need
not be absolutely continuous, thus allowing us to apply our results to
Axiom~A systems for instance.
Finally, Result~\ref{enum:3} establishes the CLT and ASIP for the
shrinking target problem in Axiom~A systems with a general class of
Gibbs measures and arbitrary shrinking sequences of target balls. It
improves upon previous works
\cite{haydn-2013-central,haydn-2017-almost}, which limit the systems to
expanding ones, and \cite{chen-2018-non-stationary}, which restricts
the measure and targets considered.
Beyond this independent contribution, Result~\ref{enum:3} is used to
verify assumptions in Results~\ref{enum:1} and \ref{enum:2}, which
explicitly assume the CLT for the corresponding hitting sums.
Results~\ref{enum:1} and \ref{enum:3} make the recurrence-vs-hitting
contrast precise since the recurrence sum
\eqref{eq:intro:result:recurrence:1} does not obey a Gaussian limit,
while the hitting sum \eqref{eq:intro:result:asip} does.

\subsection{Recurrence laws}
We now describe the recent literature on limit laws for the recurrence
problem, with particular attention to the recurrence-vs-hitting
contrast.
Borel--Cantelli lemmas (or Zero--one laws) have been considered in 
\cite{
  dolgopyat-2022-multiple,
  he-2023-quantitative,
  hussain-2022-dynamical,
  kirsebom-2023-shrinking,
  kleinbock-2023-dynamical
}.
Instead of discussion all of the results, we highlight Dolgopyat, Fayad
and Liu \cite{dolgopyat-2022-multiple}, who prove a \emph{multiple}
Borel--Cantelli lemma for exponential multiple mixing systems provided
that the targets satisfy some abstract hypotheses.
They can apply their abstract result to both recurrence and hitting.

The first results we discuss in detail are `Strong Borel--Cantelli
lemmas for recurrence', which have only been established in the past
few years. Such results strengthen the $0$--$1$ law by giving a
quantitative relationship between the rate of returns and the measure
of the neighbourhoods.
The first results we describe, in order of appearance, are
\cite{
  persson-2023-strong,
  ars-2025-recurrence,
  persson-2026-strong,
  huang-2025-exponential
}.
When the measure of $B(x,\imr_k(x))$ does not depend on $x$, and under
exponential decay of $3$-fold correlations,
Persson and the author, individually
\cite{persson-2023-strong,ars-2025-recurrence} and jointly
\cite{persson-2026-strong}, prove
\begin{equation}
\label{eq:persson-ars}
  \lim_{n \to \infty} \frac{\sum_{k=1}^{n} \charfun_{B(x,\imr_k(x))}
    (T^{k}x)}{\sum_{k=1}^{n} \mu\bigl(B(x,\imr_k(x))\bigr)} = 1
\end{equation}
under some decay conditions on the sequence
$M_n = \mu(B(x,\imr_n(x)))$.
These results can be applied to a class of Gibbs measures for piecewise
expanding maps of the interval and Axiom A systems.
Moreover in \cite{persson-2026-strong}, the decay conditions on $(M_n)$
are significantly weakened and the optimal error term is obtained under
a short returns estimate.
Huang, Li and Velani~\cite[Theorem~1.6]{huang-2025-exponential} obtain
\eqref{eq:persson-ars} with the optimal error term for self-conformal
systems equipped with Gibbs measures, assuming only that $M_n \to 0$.

He~\cite{he-2024-quantitative}, Levesley, Li, Simmons and
Velani~\cite{levesley-2024-shrinking}, and Huang, Li and
Velani~\cite{huang-2025-exponential} instead consider targets of the
form $B(x,r_k)$ whose radii do not depend on $x$ (in fact
\cite{he-2024-quantitative,levesley-2024-shrinking} consider
rectangles). In \cite{he-2024-quantitative}, He considers piecewise
expanding systems of $[0,1]^{d}$ and measures $\mu$ that are
absolutely continuous with respect to the Lebesgue measure. He
establishes
\begin{equation}
\label{eq:he}
  \lim_{n \to \infty} \frac{\sum_{k=1}^{n}
  \charfun_{B(x,r_k)}(T^{k}x)}{\sum_{k=1}^{n} \leb(B(x,r_k))} = h(x)
  \quad \text{$\mu$-a.e.\ $x$},
\end{equation}
whenever $\sum_{k=1}^{\infty}r_k^{d} = \infty$ and $r_k \to 0$, and
where $h$ is the density of $\mu$.
In \cite{levesley-2024-shrinking}, the authors restrict the measure
$\mu$ to be the Lebesgue measure and are able to remove the condition
$r_k \to 0$, while also obtaining an error term.
When $\mu$ is a Gibbs measure equivalent to the $\alpha$-Hausdorff
measure, \cite[Theorem~1.7]{huang-2025-exponential} establishes
\[
  \lim_{n \to \infty} \frac{\sum_{k=1}^{n}
  \charfun_{B(x,r_k)}(T^{k}x)}{\sum_{k=1}^{n} \mu(B(x,r_k))} = 1
  \quad \text{$\mu$-a.e.\ $x$}.
\]

In another direction, Holland and
Todd~\cite{holland-2025-distributional} prove that, whenever $\mu$ is
absolutely continuous with respect to the Lebesgue measure such that
correlations decay exponentially for $\bv$ versus $\Ell{1}(\mu)$
observables, $\sum_{k=1}^{n} \charfun_{B(x,r_n)}(T^{k}x)$ converges in
distribution to an \emph{averaged} Poisson distribution, where
$r_n = \tau/2n$ for some parameter $\tau > 0$. More precisely, for all
$k \in \naturals$,
\begin{equation}
\label{eq:poisson}
  \lim_{n \to \infty} \mu \biggl\{x \in X : \sum_{j=1}^{n}
  \charfun_{B(x,r_n)}(T^{j}x) = k\biggr\}
  = \int_X \frac{\tau^{k} h(x)^{k+1}}{k!} e^{-h(x)\tau} \diff x,
\end{equation}
where $h$ is the density of $\mu$.

We see that the limit laws in \eqref{eq:he} and \eqref{eq:poisson}
depend explicitly on the density of the measure.
Notably, this is in stark contrast to the shrinking target setting,
whose corresponding limit laws are the expected `pure' ones.
This phenomenon ultimately arises from two facts: the targets
$B(x,r_k)$ depend on the initial point $x$, and the measures are
generally not translation-invariant. That is, $\mu(B(x,r_k))$ varies
with $x$.

We see from \eqref{eq:persson-ars} that adjusting for the lack of
translation-invariance by implicitly defining the radius $r_k(x)$ to
satisfy $M_k = \mu(B(x,r_k(x)))$, then one recovers the `pure' limit
law.
The survey \cite{he-2025-shrinking} by He, Li and Velani compares the
recurrence and shrinking target problem in various settings.
Our results extend this recurrence-vs-hitting picture to the
distributional setting.

\subsection{The central limit theorem}
We say that a sequence of real-valued centred observables $(\varPhi_k)$
on the m.p.s.\ $(X,T,\mu)$ satisfies the (self-norming) central
limit theorem (CLT) if
$\sigma_n^{-1} \sum_{k=1}^{n} \varPhi_k$,
where
$\sigma_n^2 = \var( \sum_{k=1}^{n} \varPhi_k )$,
converges in distribution to a standard Gaussian variable. In other
words, if
\begin{gather}
\label{clt}
\tag{CLT}
  \lim_{n \to \infty} \mu \biggl\{ x : \frac{1}{\sigma_n}
    \sum_{k=1}^{n} \varPhi_k(x) < t \biggr\}
  = \frac{1}{\sqrt{2\pi}}\int_{-\infty}^{t} e^{-\frac{s^2}{2}} \diff s
    \quad \text{for all $t \in \reals$}.
\end{gather}
Alternatively, L\'evy's continuity theorem
(see for instance \cite[Theorem~5.3]{kallenberg-2002-foundations})
implies that \eqref{clt} is equivalent to
\[
  \lim_{n \to \infty} \mu \biggl( \exp\biggl( \frac{it}{\sigma_n}
  \sum_{k=1}^{n} \varPhi_k \biggr) \biggr)
  = e^{-\frac{t^2}{2}}
  \quad \text{for all $t \in \reals$}.
\]
Usually one considers $\varPhi_k = \varphi \circ T^{k}$ for a fixed
observable $\varphi$, or more
generally $\varPhi_k = \varphi_k \circ T^{k}$ for a sequence of
observables $(\varphi_k)$. As mentioned previously, when
$\varPhi_k(x) = \charfun_{B(x,r_k)}(T^{k}x) - \mu(B(x,r_k))$, we do
not obtain \eqref{clt}, but rather a non-standard law
\begin{gather}
\label{alt}
\tag{ALT}
  \lim_{n \to \infty} \mu \biggl( \exp\biggl( \frac{it}{\sigma_n}
  \sum_{k=1}^{n} \varPhi_k \biggr) \biggr)
  = \int_X \exp\Bigl( -\frac{h(x)}{2\mu(h)} t^2 \Bigr) \diff\mu(x)
  \quad \text{for all $t \in \reals$},
\end{gather}
where $h$ is the density of the measure $\mu$.
We call this the \emph{averaged limit theorem} (ALT).
Note that \eqref{clt} and \eqref{alt} coincide when $h = 1$, i.e., when
$\mu$ is the Lebesgue measure.

The literature on the CLT for dynamical systems is vast, see for
instance
\cite{
  bjorklund-2020-quantitative,
  bowen-2008-equilibrium,
  chernov-1995-limit,
  dolgopyat-2004-limit,
  guivarch-1988-theoremes
}.
We direct the interested reader to the survey of
Le~Borgne~\cite{leborgne-2014-martingales} for the martingale approach,
and to the survey of Gou{\"e}zel~\cite{gouezel-2015-limit} for the
spectral approach.
We will direct our focus on non-autonomous CLTs, since we are
interested in \eqref{clt} for observables of the form
$\varPhi_k(x) = \charfun_{B_k(x)}(T^{k}x) - \mu(B_k(x))$.
Most of the recent literature
\cite{
  conze-2011-limit,
  demers-2025-central,
  denker-2017-lindeberg,
  dolgopyat-2025-berry,
  haydn-2017-almost,
  haydn-2013-central
}
is focused on sequential CLTs, where the dynamical system depends on
$k$. In these works, the observables either do not depend on
$k$, or the conditions on the sequence of observables are too
restrictive to allow application to the recurrence or shrinking target
problems in higher dimensional hyperbolic settings.

Conze and Raugi~\cite{conze-2007-limit} consider the CLT for piecewise
expanding sequential dynamical systems on $[0,1]$.
Later, Haydn, Nicol, Vaienti and Zhang~\cite{haydn-2013-central} adapt
the martingale method to prove the CLT for the shrinking target problem
in the setting of piecewise expanding interval maps with a spectral
gap. Their method involves a short returns estimate,
which was later removed by Haydn, Nicol, T{\"o}r{\"o}k and
Vaienti~\cite{haydn-2017-almost}. In both papers, the variance
$\sigma_n^2$ is assumed to satisfy either a polynomial decay or growth
condition, respectively.
Recently, Dolgopyat and Hafouta~\cite{dolgopyat-2025-berry} prove
Berry--Esseen theorems for sequential dynamical systems and remove the
growth condition on $\sigma_n^2$ completely, assuming only that
$\sigma_n^2 \to \infty$. Their results are stated for abstract systems
equipped with some abstract `variation' norm. Taking $T_j = T$ in
\cite[Theorem~3.1]{dolgopyat-2025-berry} and $f_j$ to be indicator
functions of shrinking targets, one obtains Berry--Esseen theorems for
the shrinking target problem in the setting of
piecewise expanding interval maps.
We note that in
\cite{haydn-2013-central, haydn-2017-almost, dolgopyat-2025-berry}
the observables $(f_j)$ are required to be uniformly bounded in the
relevant (semi-)norm.
In a recent preprint, Demers and Liverani~\cite{demers-2025-central}
establish a CLT for sequential dispersing billiards in which the
observables are allowed to be unbounded, but with uniformly bounded
H{\"o}lder coefficients.
As a result, it is not possible to apply aforementioned results to the
shrinking target problem in the higher dimensional hyperbolic setting.
A related result by Denker, Senti and
Zhang~\cite{denker-2017-lindeberg} establishes a Lindeberg-type CLT for
dynamical arrays in which the observables depend on the variable $n$ in
\eqref{clt} rather than $k$.

Although not applicable to the non-autonomous setting, we mention the
result of Bj{\"o}rklund and Gorodnik~\cite{bjorklund-2020-central} as
it features a similar assumption.
Namely, they establish an autonomous CLT for group actions that satisfy
exponential multiple mixing as formulated in Bj{\"o}rklund, Einsiedler
and Gorodnik~\cite{bjorklund-2020-quantitative}.
We emphasize that the multiple mixing condition we assume in this paper
is different to the one established in
\cite{bjorklund-2020-quantitative} (seemingly neither property implies
the other).

Finally, our result is motivated by the paper \cite{conze-2011-limit}
of Conze and Le~Borgne. In \cite[Theorem~3.3]{conze-2011-limit}, they
show that the central limit theorem generally does not hold for the
sums of the form
$n^{-1/2} \sum_{k=1}^{n}
\varphi(T^{k}x - x) - \mu (\varphi(\cdot - x) )$, when $X$ is a compact
group.
Instead, when $\varphi$ is H{\"o}lder continuous, and $(X,T,\mu)$
satisfies a multiple decorrelation property, the characteristic
function of the sum converges to the function
\begin{equation}
\label{eq:conze}
  t \mapsto \int_X
    \exp \Bigl(- \frac{1}{2} \sigma_x^2 t^2 \Bigr) \diff\mu(x),
\end{equation}
where
$\sigma_x^2 = \lim_{n \to \infty} n^{-1} \var( \sum_{k=1}^{n}
\varphi(T^{k}\cdot -\ x) - \mu(\varphi(\cdot - x)))$.
Although their result is autonomous,
it suggests that the sum in \eqref{eq:intro:result:recurrence:1}
satisfies a non-standard limit theorem.

Our main results on distributional laws are
Theorems~\ref{thm:alt} and Theorems~\ref{thm:clt}.
In Theorem~\ref{thm:alt}, we show that the sum in
\eqref{eq:intro:result:recurrence:1} satisfies \eqref{alt}
under a short returns estimate,
a multiple decorrelation condition,
and the assumption that the CLT holds for the shrinking target problem.
In Theorem~\ref{thm:clt} and under similar assumptions, we show that
one can consider a version of the sum in
\eqref{eq:intro:result:recurrence:1}, namely the sum in
\eqref{eq:intro:result:recurrence:2}, and recover the CLT.
An added benefit of considering the version in
\eqref{eq:intro:result:recurrence:2} is that one can consider systems
for which the measure is not absolutely continuous.

We take a characteristic function approach inspired by
\cite{conze-2011-limit}, which is where the multiple decorrelation
condition is important. The approach involves showing that the
characteristic function of the suitably normalized recurrence sum
\eqref{eq:intro:result:recurrence:1}
can be approximated by an average over the characteristic functions of
the corresponding suitably normalized hitting sums 
\[
  \sum_{k=1}^{n} \charfun_{B(y,r_k)}(T^{k}x) - \mu( B(y,r_k) ).
\]
By the assumption that the suitably normalized hitting sums converge in
distribution to a standard Gaussian variable for $\mu$-a.e.\ $y$, one
then obtains \eqref{alt}. Since the normalising factors of the hitting
sums depend on $y$, this makes it clear why we obtain a non-standard
limit law rather than \eqref{clt}.
If we instead consider the sum in \eqref{eq:intro:result:recurrence:2},
we see that the normalising factor of the corresponding hitting sums
\[
  \sum_{k=1}^{n} \charfun_{B(y,\imr_k(y))}(T^{k}x)
  - \mu\bigl( B(y,\imr_k(y)) \bigr)
\]
do not depend on $y$. Thus, it becomes clear why we recover \eqref{clt}
in this setting.
It is therefore important to show that the corresponding sums of the
shrinking target problem obey a CLT. As such, we establish the stronger
almost sure invariance principle, from which the CLT follows.

\subsection{The almost sure invariance principle}
A stronger result than the CLT is the almost sure invariance principle
(ASIP), which can be stated simply as an almost-sure version of the
CLT. The particular version of the ASIP we are interested in is the
following. We say that a sequence of real valued centred observables
$(\varPhi_k)$ on the m.p.s.\ $(X,T,\mu)$ satisfies the ASIP with
rate $\beta \in (0,1)$ if there exists a sequence of standard Gaussian
random variables $(Z_n)$ (by enlarging the underlying probability space
if necessary) such that
\begin{gather}
\label{asip}
\tag{ASIP}
  \sup_{1 \leq k \leq n} \biggl| \sum_{i=1}^{k}
  \varPhi_i - \sum_{i=1}^{k} Z_i \biggr|
  = \smallo( \sigma_n^{1-\beta} )
  \quad \text{a.s.}
\end{gather}
and $\sum_{k=1}^{n} \EE[Z_k^2] \sim \sigma_n^2$,
where $\sim$ denotes asymptotic equivalence. Again, usually one
considers $\varPhi_k = \varphi \circ T^{k}$, or
$\varPhi_k = \varphi_k \circ T^{k}$.
We are interested in the case when
$\varPhi_k = \charfun_{B(y_k,r_k)} \circ T^{k}$.

The ASIP is a strong asymptotic estimate, and basically says that the
sum $\sum_{k=1}^{n} \varPhi_k$ can be approximated sufficiently well
by Brownian motion. As such, one can translate properties of Brownian
motion to the sum of observables. In particular, the ASIP implies other
limit laws, such as the functional CLT and the law of the iterated
logarithm (LIL). For a demonstration of these implications, and for the
first ASIP for dynamical systems, see Philipp and
Stout~\cite{philipp-1975-almost}.

In the stationary setting, the ASIP has been established for wide class
of systems, see for instance
\cite{
  denker-1984-approximation,
  dolgopyat-2004-limit,
  gouezel-2010-almost,
  melbourne-2009-vector,
  merlevede-2012-strong
}
and references within.
As with the CLT, we restrict our focus to the non-stationary ASIP,
which has recently gained significant interest 
\cite{
  bayraktar-2018-almost,
  castro-2019-stability,
  chen-2018-non-stationary,
  haydn-2017-almost,
  liu-2024-wasserstein,
  su-2019-almost,
  su-2022-vector
}.
(We make a note that random systems have also been considered, see
\cite{liu-2024-wasserstein,su-2019-almost,su-2022-vector}
and references within.)
Similarly to the non-stationary CLT literature, it is usually assumed
that the sequence of observables is constant or uniformly bounded in
the relevant (semi-)norm, thus preventing the application to the
shrinking target problem in higher dimensional hyperbolic systems.
The exception to this is \cite{chen-2018-non-stationary}, which allows
for a specific form of targets for $2$-dimensional non-uniformly
hyperbolic systems.

Haydn, Nicol, T{\"o}r{\"o}k and Vaienti~\cite{haydn-2017-almost}
establish the ASIP for sequential expanding maps of the interval, and
uniformly hyperbolic maps. For expanding maps, the authors apply their
results to the shrinking target problem under assumptions on the growth
speed of the variance.
Chen, Yang and Zhang~\cite{chen-2018-non-stationary} extend the
shrinking target result of \cite{haydn-2017-almost} to $2$-dimensional
non-uniformly systems with singularities, and for which the measure
$\mu$ is the SRB measure. However, the targets need to be of a certain
form with respect to the singularities. They also improve the growth
condition of the variance compared to \cite{haydn-2017-almost}.
Bayraktar~\cite{bayraktar-2018-almost} considers sequences of
holomorphic maps of the same degree on the complex projective space
$\proj^{k}$ equipped with a natural measure associated to the sequence
of maps. The observables are assumed to be either H{\"o}lder continuous
or `DHS'.
Castro, Rodrigues and Varandas~\cite{castro-2019-stability} consider
convergent sequences of Anosov diffeomorphisms and a fixed H{\"o}lder
continuous observable. For any equilibrium state of the limit, they
obtain the ASIP with respect to an associated measure.
Su~\cite{su-2019-almost} considers sequences of Pomeau--Manneville maps
equipped with the Lebesgue measure, and a fixed Lipschitz continuous
observable.
Subsequently~\cite{su-2022-vector} extended the results to
vector-valued sequence of Lipschitz observables.
Finally, Liu and Wang~\cite{liu-2024-wasserstein} obtain convergence
rates with respect to the Wasserstein distance for
$\beta$-transformations, piecewise expanding maps on the interval, and
a class of covering maps. They consider uniformly bounded sequences of
H{\"o}lder continuous observables.

Motivated by the shrinking target problem, we modify the approach of
\cite{haydn-2017-almost}, which in turn is based on an abstract
martingale framework due to Cuny and
Merlev\`ede~\cite[Theorem~2.3]{cuny-2015-strong}
(see Theorem~\ref{thm:asip:abstract} below).
Our main ASIP contribution result is Theorem~\ref{thm:asip:hyperbolic},
which establishes the ASIP for Axiom~A systems with Gibbs measures, and
for which the sequence of observables need not be uniformly bounded.
More precisely, if $(\varphi_n)$ is a sequence of $\alpha$-H{\"o}lder
continuous observables, then we allow the H{\"o}lder coefficient
$\holconst{\varphi_n}{\alpha}$ to grow polynomially in $n$, i.e.,
$\holconst{\varphi_n}{\alpha} = \bigo(n^{q})$.
This improves upon \cite[Theorem~6.1]{haydn-2017-almost}, which assumes
$\sup_{n \geq 1} \holnorm{\varphi_n}{\alpha} < \infty$.
The polynomial growth of H{\"o}lder norms is what allows us to obtain
the ASIP for the shrinking target problem
(Theorem~\ref{thm:asip:hyperbolic:targets}) as a consequence since it
gives us enough freedom to approximate indicator functions of balls by
H{\"o}lder continuous observables.
We also improve the rate on the lower polynomial growth of the variance
$\sigma_n^2 \geq n^{\delta}$ from
$\delta > \frac{\sqrt{17}-1}{4} \approx 0.781$ in
\cite{haydn-2017-almost} to $\delta > 1/2$.
Compared to \cite[Theorem~18]{chen-2018-non-stationary}, we are able to
consider more general Gibbs measures than the SRB measure, and we are
able to consider any shrinking sequence of target balls. We note that 
\cite{chen-2018-non-stationary} does allow for polynomial growth of
`dynamical H{\"o}lder norms'.

As in \cite{haydn-2017-almost}, we establish the ASIP for Axiom~A
systems by exploiting the Markov partition, and use a modification of
the usual `Sinai trick' to reduce the problem to that of establishing
the ASIP for a one-sided sub-shift of finite type.
A modified Sinai trick is already necessary in
\cite{haydn-2017-almost} in order to account for the non-autonomous
nature of the problem. However, since we impose weaker conditions on
the H{\"o}lder norms of the observables, we need further modifications.
In particular, our modification results in the corresponding
sequence of observables on the shift-space no longer being uniformly
bounded in the supremum norm. In particular, we have
$\supnorm{\varphi_n} = \bigo (\log n)$. As a result, we also need to
further weaken the assumptions of the ASIP for non-invertible systems
present in \cite{haydn-2017-almost}.
This is done in Theorem~\ref{thm:asip:non-invertible:sup-growth}.

\subsection{Applications and examples}
We now list some representative examples of dynamical systems for which
we obtain the non-standard limit law \eqref{alt} for
$\varPhi_k(x) = \charfun_{B(x,r_k)}(T^{k}x) - \mu(B(x,r_k))$, i.e.,
\begin{multline}
\label{eq:ex:alt}
  \lim_{n \to \infty} \mu \biggl( \exp\biggl( \frac{it}{\sigma_n}
  \sum_{k=1}^{n} \charfun_{B(x,r_k)}(T^{k}x) - \mu(B(x,r_k)) \biggr)
  \biggr) \\
  = \int_X \exp\Bigl( -\frac{h(x)}{2\mu(h)} t^2 \Bigr) \diff\mu(x)
  \quad \text{for all $t \in \reals$},
\end{multline}
and for which we obtain \eqref{clt} for
$\varPhi_k(x) = \charfun_{B(x,r_k(x))}(T^{k}x) - \mu(B(x,r_k(x)))$,
i.e.,
\begin{equation}
\label{eq:ex:clt}
  \frac{1}{\sigma_n} \sum_{k=1}^{n}
  \charfun_{B(x,r_k(x))}(T^{k}x) - \mu\bigl(B(x,r_k(x))\bigr)
  \todistr N(0,1).
\end{equation}

\begin{example}[Gauss map]
\label{ex:gauss}
  Let $G \colon [0,1] \to [0,1]$ be the Gauss map, i.e.,
  \[
    G(x) =
    \begin{cases}
      0                   & x=0, \\
      \frac{1}{x} \pmod 1 & x \in (0,1],
    \end{cases}
  \]
  and let $\mu$ be the unique absolutely continuous measure with
  density
  \[
    h(x) = \frac{1}{(1+x)\log 2}.
  \]
  Then \eqref{eq:ex:alt} and \eqref{eq:ex:clt} both hold, but the
  limiting distributions are distinct.
  This is an explicit example of the CLT failing if one does not
  rescale the radii.
\end{example}

\begin{example}[Cat map]
\label{ex:cat}
  Let $T \colon [0,1]^2 \to [0,1]^2$ be given by $Tx = Ax \pmod 1$,
  where
  \[
    A =
    \begin{pmatrix}
      2 & 1 \\
      1 & 1
    \end{pmatrix},
  \]
  and let $\mu$ be the Lebesgue measure.
  Then \eqref{eq:ex:alt} and \eqref{eq:ex:clt} hold, and the limiting
  distributions are both standard Gaussian ones.
\end{example}

If we consider measures that are not absolutely continuous with respect
to the Lebesgue measure, we can consider more general systems. However,
the new difficulty lies in verifying the short returns condition for
these systems. In particular, we cannot verify this condition for
Example~\ref{ex:cat} if we replace the Lebesgue measure with a more
general Gibbs measure.
The choice of $3$ in the following example reflects the minimal
expansion needed to verify our short returns condition.
The minimum expansion fails for the doubling map, and as such we cannot
verify the short returns conditions when $\mu$ is not the Lebesgue
measure.

\begin{example}[$\times 3$-map]
\label{ex:tripling}
  Let $T \colon [0,1] \to [0,1]$ be given by $Tx = 3x \pmod 1$, and let
  $\mu$ be a non-absolutely continuous Gibbs measure corresponding to a
  H{\"o}lder continuous potential.
  Then \eqref{eq:ex:clt} holds, while \eqref{eq:ex:alt} is not
  well-defined.
\end{example}

In all of the examples above, we can show that the corresponding sums
for the shrinking target problem satisfy the ASIP.

\subsection{Structure of the paper}
The remainder of the paper is organised as follows.
Section~\ref{sec:prelims} establishes notation and assumptions.
Sections~\ref{sec:clt} and \ref{sec:asip} state the main distributional
and ASIP results, respectively.
Section~\ref{sec:variance} collects variance estimates, and
Section~\ref{sec:supporting} two technical lemmas.
Sections~\ref{sec:proof:clt}--\ref{sec:proof:supporting}
give the proofs, roughly in the order in which the results appear.
Section~\ref{sec:sr:proof} verifies the short returns assumption for
expanding maps on the interval equipped with Gibbs measures, and
hyperbolic toral automorphisms equipped with the Lebesgue measure.

\section{Preliminaries}
\label{sec:prelims}

\subsection{Setting and notation}
Throughout the paper, we assume that $(X,d)$ is a smooth compact
Riemannian manifold, and that
$(X,T,\mu)$ is a Borel measure-preserving dynamical system (m.p.s.)\
with $\mu(X) = 1$.
That is, $(X,\mu)$ is a Borel probability
space and $T \colon X \to X$ is a measurable transformation preserving
$\mu$. We assume also that $\mu$ is non-atomic.
We denote the normalised transfer operator corresponding to the
system by $P$, i.e., $P \colon \Ell{2}(\mu) \to \Ell{2}(\mu)$ is the
operator defined by
$\int_X \varphi (\psi \circ T) \diff\mu
= \int_X (P\varphi) \psi \diff\mu$
for all $\varphi, \psi \in \Ell{2}(\mu)$.

For an observable $f \colon X \to \reals$, we write
$\mu(f) = \int_X f \diff\mu$ and we denote by
$\tilde{f}$ the centred observable, i.e., $\tilde{f} = f - \mu(f)$.
If $g \colon X \to \reals$ is another observable, we define
$\covar(f,g) = \mu(fg) - \mu(f)\mu(g)$ and $\var(f) = \covar(f,f)$.
By the characteristic function of $f$, we
mean the function $\psi_f \colon \reals \to \complexes$ given by
$\psi_f(t) = \int_X \exp(itf(x)) \diff\mu(x)$.
For $\alpha \in (0,1]$, we denote by $\hol{\alpha} =
\hol{\alpha}(\reals)$ (resp.\ $\hol{\alpha}(\complexes)$) the space of
real-valued (resp.\ complex-valued) $\alpha$\nobreakdash-H{\"o}lder
continuous functions on $X$, and equip it with the norm
$\holnorm{\cdot}{\alpha}$ defined by
$\holnorm{f}{\alpha} = \supnorm{f} + \holconst{f}{\alpha}$ for $f \in
\hol{\alpha}$, where $\supnorm{f} \coleq \sup_{x \in X} |f(x)|$ and
\[
  \holconst{f}{\alpha} \coleq \sup_{x \neq y}
  \frac{|f(x) - f(y)|}{d(x,y)^{\alpha}}.
\]
When $\alpha = 1$, we write
$\lipconst{\cdot} = \holconst{\cdot}{1}$ and
$\lipnorm{\cdot} = \holnorm{\cdot}{1}$.

For a set $A$, denote by $\card{A}$ the cardinality of $A$ and,
for a set $Q \subset X$, let
$\diam{Q} \coleq \sup \{d(x,y) : x,y \in Q\}$ be the diameter of $Q$.
If $\partition$ is a collection of subsets of $X$,
we define $\diam{\partition} \coleq \sup_{Q \in \partition} \diam{Q}$.
We write $a \leqs_x b$ when there exists a constant $C > 0$ depending
on $x$ such that $a \leq Cb$. We write $a \leqs b$ if $C > 0$ is a
universal constant, i.e., it depends only on the system $(X,T,\mu,d)$.
We define $\geqs$ analogously, and $a \eqs b$ if $a \leqs b$ and $a
\geqs b$. For sequences $a_n$ and $b_n$, we write $a_n \sim b_n$ to
mean that $a_n$ and $b_n$ are asymptotically equivalent, i.e.,
$a_n/b_n \to 1$.

\subsubsection{Recurrence sets and radius functions}
In Result~\ref{enum:1}, it will be useful to consider the sets
$\exE_n = \{x \in X : T^{n}x \in B(x,r_n)\}$, where $(r_n)$ is a
sequence of scalars and $B(x,r)$ denotes the open ball with centre $x$
and radius $r$. In order to recover the CLT in Result~\ref{enum:2}, we
consider a reference sequence $(M_n)$ in $[0,1]$ and the corresponding
sets $\imE_n = \{x \in X : T^{n}x \in B(x,\imr_n(x))\}$, where the
functions $\imr_n \colon X \to [0,\infty)$ satisfy
$\mu(B(x,\imr_n(x))) = M_n$ $\mu$-a.e.\ $x$. We delay the question
of existence and briefly comment on the notation. The above notation is
used in
\cite{
  he-2024-quantitative,
  holland-2025-distributional,
  levesley-2024-shrinking
}, while in
\cite{
  persson-2023-strong,
  persson-2026-strong,
  ars-2025-recurrence
}, only the sets $\imE_n$ are studied and are therefore denoted by
$E_n$ within. Since we will be studying both sets, we opt for the above
notation for consistency and clarity.

\subsection{Main assumptions}
We now state the main assumptions, which will be used at different
times throughout the article. Briefly stated, we restrict ourselves to
rapidly mixing systems for which the measure satisfies some continuity
properties. For some results, we also assume a short returns estimate.

\begin{assumptiond}{D.1}[Spectral Gap]
\label{cond:transfer-operator}
  For every $\alpha \in (0,1]$, the transfer operator $P$ restricts to
  an operator $P \colon \hol{\alpha} \to \hol{\alpha}$ such that there
  exists $\spgap > 0$ for which
  \[
    \holnorm{P^{n}\varphi}{\alpha}
    \leqs_\alpha \holnorm{\varphi}{\alpha} e^{-\spgap n}
  \]
  for all $n \in \naturals$ and all $\varphi \in \hol{\alpha}$ with
  $\mu(\varphi) = 0$.
\end{assumptiond}

\begin{assumptiond}{D.2}[Decay of Correlations]
\label{cond:doc}
  For every $\alpha \in (0,1]$, there exists $\doc > 0$ such that,
  for all $n \in \naturals$ and all
  $\varphi,\psi \in \hol{\alpha}$,
  \[
    | \covar( \varphi, \psi \circ T^{n} ) | \leqs_\alpha
    \holnorm{\varphi}{\alpha} \holnorm{\psi}{\alpha} e^{-\doc n}.
  \]
\end{assumptiond}

\begin{assumptiond}{D.3}[Multiple Decorrelation]
\label{cond:multiple-doc}
  For every $\alpha \in (0,1]$, there exists $r \geq 1$, a constant
  $\mdoc > 0$, and a polynomial $P_r$ with real non-negative
  coefficients such that the following holds. For all integers
  $m,m',N \geq 0$, all integer sequences
  $0 \leq k_1 \leq \ldots \leq k_m$ and
  $0 \leq l_1 \ldots \leq l_{m'}$, and all functions
  $\varphi_1, \ldots, \varphi_{m+m'} \in \hol{\alpha}(\complexes)$,
  \begin{multline*}
    \covar \biggl( \prod_{i=1}^{m} \varphi_i \circ T^{k_i},
    \prod_{j=1}^{m'} \varphi_{m+j} \circ T^{N + l_j} \biggr) \\
    \leq \biggl( \prod_{i=1}^{m+m'} \supnorm{\varphi_i}
    + \sum_{i=1}^{m+m'} \holconst{\varphi_i}{\alpha} \prod_{j \neq i}
    \supnorm{\varphi_j} \biggr) P_r(l_{m'}) e^{-\mdoc (N - rk_m)}.
  \end{multline*}
\end{assumptiond}

\begin{assumptionm}{M.1}[Frostman property]
\label{cond:frostman}
  There exists $\frost > 0$ such that, for all $x \in X$ and $r > 0$,
  \[
    \mu\bigl( B(x, r) \bigr) \leqs r^{\frost}.
  \]
\end{assumptionm}

\begin{assumptionm}{M.2}[Thin annuli]
\label{cond:thin-annuli}
  There exists $\thinan, \thinanr > 0$ such that, for all $x \in X$ and
  $\varepsilon \leq r \leq \thinanr$,
  \[
    \mu\bigl( B(x, r + \varepsilon) \setminus B(x, r) \bigr)
    \leqs \varepsilon^{\thinan}.
  \]
\end{assumptionm}

\begin{assumptionsr}{SR.1}[Short returns for $(r_k)$]
\label{cond:sr:explicit}
  There exists $\srelog > 1$ and $\srein, \sreout > 0$ such that,
  for all $k \geq 1$ and all $l = 1,\ldots, (\log k)^{\srelog}$,
  \[
    \mu \bigl\{ x \in X : \mu\bigl( B(x,r_k) \cap T^{-l} B(x,r_k)
    \bigr) > \exAv_k^{1+\srein} \bigr\} \leqs \exAv_k^{\sreout},
  \]
  where $\exAv_k = \int_X \mu(B(x,r_k)) \diff\mu(x)$.
\end{assumptionsr}

\begin{assumptionsr}{SR.2}[Short returns for $(\imr_k)$]
\label{cond:sr:implicit}
  There exists $\srilog > 1$ and $\sriin, \sriout > 0$ such that,
  for all $k \geq 1$ and all $l = 1,\ldots, (\log k)^{\srilog}$,
  \[
    \mu \bigl\{ x \in X : \mu\bigl( B(x,\imr_k(x)) \cap T^{-l}
    B(x,\imr_k(x)) \bigr)
    > M_k^{1+\sriin} \bigr\} \leqs M_k^{\sriout}.
  \]
\end{assumptionsr}

We will frequently, though not always, rely on the following condition,
where $\seqlow, \seqhigh \geq 0$ are constants to be specified,
and $(a_n)$ will be one of the sequences $(r_n)$ or $(M_n)$.

\begin{assumptions}{$\scal(\seqlow,\seqhigh)$}
\label{cond:sequence}
  The sequence $(a_n)$ is monotone decreasing and, for all
  $n \in \naturals$,
  \[
    \frac{1}{n^{\seqlow}} \leqs a_n \leqs \frac{1}{(\log n)^{\seqhigh}}.
  \]
\end{assumptions}

\subsection{Discussion of the assumptions}
We now discuss the assumptions, giving examples of systems satisfying
them, and comparing with previous works.

\subsubsection{Recurrence sets and radius functions}
\label{sec:discussion:radius}
In establishing statistical properties of the sum
$S_n(x) = \sum_{k=1}^{n}\charfun_{B(x,r_k)}(T^{k}x)$,
it becomes important to obtain good estimates of the measure and
correlations of the sets $(\exE_n)$.
In general, this is difficult for two reasons: we cannot
directly exploit the $T$-invariance of $\mu$ since the balls $B(x,r_n)$
(clearly) depend on $x$; and the measure of $B(x,r_n)$ may vary wildly
as $x$ changes. For the first issue, it is possible to use decay of
correlations to show that
$\mu(\exE_n) \approx \int \mu(B(x,r_n)) \diff\mu(x)$. However, the
second issue makes correlation estimates difficult to obtain in
general. If we assume that $\mu$ is absolutely continuous with respect
to $\leb$, then it is possible to obtain correlation estimates under
some conditions on the density.

In order to consider more general measures, we study the sets
$(\imE_n)$ defined in terms of the reference sequence $(M_n)$ and
rescaled balls $B(x,\imr_n(x))$. This alleviates issues arising from
$\mu$ not being translation-invariant at the cost of having to work
with implicitly defined radii $\imr_n(x)$, which a priori may not even
exist for some $x$. However, if we define
$\imr_n(x) = \inf \{r \geq 0 : \mu(B(x,r)) \geq M_n\}$, then it is
possible to show that $\imr_n$ is $1$-Lipschitz for all $n \geq 1$.
Furthermore, if $M_n \to 0$, then $\imr_n \to 0$ uniformly on
$\supp \mu$.
Under Condition~\ref{cond:thin-annuli}, we have $\imr_n \leq \thinanr$ on
$\supp \mu$ for all sufficiently large $n$. Since the boundary of any
ball of radius less than $\thinanr$ has $0$ measure, we have
$\mu(B(x,\imr_n(x))) = M_n$ for all $x \in \supp \mu$
\cite[Lemma~3.1]{ars-2025-recurrence}.
More generally, even without \ref{cond:thin-annuli}, if $\mu$ is
non-atomic, then $\mu(B(x,\imr_n(x))) = M_n$ $\mu$-a.e.\ $x$.

\begin{remark}
The use of $\imr_n$ originated in Kirsebom, Kunde and
Persson~\cite{kirsebom-2023-shrinking} in which they are used to
obtain Borel--Cantelli lemmas for measures singular with respect to the
Lebesgue measure.
Later, the recurrence sets $(\imE_n)$ formed the main objects of study
in \cite{persson-2023-strong,persson-2026-strong,ars-2025-recurrence}.
Meanwhile, the functions $\imr_k$ have been used in
\cite{he-2023-quantitative,he-2024-quantitative,huang-2025-exponential}
to obtain results for the explicit radius problem.
In this paper, we state and prove distributional results for both $r_k$
and $\imr_k$, but neither is used to derive the other.
\end{remark}

\subsubsection{Decay of correlations}
Condition~\ref{cond:transfer-operator} holds when the transfer
operator has a spectral gap, and implies decay of correlations for
H{\"o}lder versus $\Ell{1}$ observables, i.e., for all
$\varphi \in \hol{\alpha}$ and $\psi \in \Ell{1}$,
\begin{equation}
\label{eq:doc:transfer-operator}
  | \covar ( \varphi, \psi \circ T^{n} ) |
  \leqs_\alpha \holnorm{\varphi}{\alpha} \lnorm{\psi}{1} e^{-\spgap n}.
\end{equation}
Systems for which the transfer operator has a spectral gap include
sub-shifts of finite type
(see e.g.\ Bowen~\cite{bowen-2008-equilibrium}), piecewise
uniformly expanding maps on $[0,1]$ equipped with an absolutely
continuous invariant measure
(see \cite{lasota-1973-existence, rychlik-1983-bounded}),
and multi-dimensional expanding maps
(see \cite{saussol-2000-absolutely}).

For invertible systems (such as Axiom~A systems), the transfer operator
is given by $Pf = f \circ T^{-1}$ and as such one cannot hope for
Condition~\ref{cond:transfer-operator} to hold.
In such cases, we will assume \ref{cond:doc}~-- a weaker condition~--
which, in addition to the aforementioned systems, is known to hold for
for piecewise uniformly expanding maps on $[0,1]$ equipped with Gibbs
measures (see \cite{liverani-1998-conformal}), and Axiom~A
diffeomorphisms equipped with equilibrium states corresponding to a
H{\"o}lder continuous potential (see \cite{bowen-2008-equilibrium}). It
is a standard assumption when considering statistical properties of
dynamical systems.

To prove the distributional laws and to establish estimates of
$\mu(\exE_k \cap \exE_j)$ and $\mu(\imE_k \cap \imE_j)$, we will assume
higher-order mixing in the form of Condition~\ref{cond:multiple-doc}.
While it is not a pure exponential decay in $N$ due to the polynomial
$P_r$ and the $rk_m$ term in the exponential, the power of the
statement lies in the fact that $P_r$ and $\mdoc$ do not depend on the
number, $m+m'$, of observables. When proving the
distributional laws, we take full advantage of this fact, while when
estimating $\mu(\exE_k \cap \exE_j)$ and $\mu(\imE_k \cap \imE_j)$, we
only use the condition restricted to $m+m'=3$ observables.

The form of \ref{cond:multiple-doc} we use is due to
P\`ene~\cite[Property {$(\pcal_r)$}]{pene-2004-multiple}.
We stress that the constant $\mdoc$ and polynomial $P_r$ are uniform
in the number of observables $m+m'$, which will be essential in our
proofs of the distributional results.
We note that Conze and Le~Borgne~\cite[Property~3.1]{conze-2011-limit},
on which our method is based, assume \ref{cond:multiple-doc} with
$r = 1$ and $P_r = C$.
In \cite{pene-2000-thesis}, P\`ene establishes \ref{cond:multiple-doc}
for sub-shifts of finite type equipped with Gibbs measures (in fact
$P_r$ is constant in this case). Using a standard argument
\cite{bowen-2008-equilibrium}, one can show that the condition extends
to Axiom~A diffeomorphisms equipped with Gibbs measures corresponding
to H{\"o}lder continuous potential.
For dynamical systems to which Young's method of
\cite{young-1998-statistical} can be applied, it can be shown that
\ref{cond:multiple-doc} is satisfied for any $r > 1$
\cite[Corollary~B.2]{pene-2002-rates}. This version holds for
billiards with bounded horizon and no corners
\cite[Theorem~7.41]{chernov-2006-chaotic} with `dynamically H{\"o}lder
continuous observables'.

By contrast, in other (more familiar) versions of
\emph{multiple decorrelation}
\cite{
  bjorklund-2020-quantitative,
  dolgopyat-2004-limit,
  kotani-2001-pressure
}
the multiplicative constant $P_r = C$ depends on $m+m'$,
which is insufficient for our approach. More precisely,
Kotani and Sunada~\cite{kotani-2001-pressure} prove the non-uniform
multiple decorrelation property for Axiom~A
diffeomorphisms equipped with Gibbs measures corresponding to a
H{\"o}lder continuous observable;
Dolgopyat~\cite{dolgopyat-2004-limit} proves it for hyperbolic systems
for partially hyperbolic systems that are `strongly $u$-transitive with
exponential rate';
and Bj{\"o}rklund, Einsiedler and
Gorodnik~\cite{bjorklund-2020-quantitative} prove it for semi-simple
Lie groups, semi-simple $S$-algebraic groups, and semi-simple adele
groups.

\subsubsection{Continuity of measure}
In addition to the good mixing properties, we require
the measure to be sufficiently smooth.
Conditions~\ref{cond:frostman} and \ref{cond:thin-annuli} are used to
obtain good H{\"o}lder continuous approximations of indicator functions
of balls. For example, for
$\charfun_{E_n}(x) = \charfun_{B(x,\imr_k(x))}(T^{k}x)$,
we will obtain approximations
$g_k(x,y) \approx \charfun_{B(x,\imr_k(x))}(y)$
such that
$\iint g_k \diff\mu\diff\mu \approx M_n$
and
$\int g_k(x,T^{k}x) \diff\mu(x) \approx \mu(E_k)$.

Conditions~\ref{cond:frostman} and \ref{cond:thin-annuli} are standard
assumptions; they are assumed, either implicitly or explicitly, in the
aforementioned works establishing RSBCs
\cite{
  he-2024-quantitative,
  holland-2025-distributional,
  levesley-2024-shrinking,
  persson-2023-strong,
  persson-2026-strong,
  ars-2025-recurrence
}.
If $\mu$ is absolutely continuous with density $h \in L^{q}(\leb)$ for
some $q > 1$, then both conditions hold. Furthermore, if $X$ is a
Riemannian manifold with dimension $\dim X$ and $\mu$ satisfies
\ref{cond:frostman} for $\frost > \dim X - 1$, then
\ref{cond:thin-annuli} holds for $\thinan = \dim X + 1 - s$ (see the
proof of Theorem~1.1 in \cite{ars-2025-recurrence}). For instance, this
applies to Gibbs measures for finitely-branched uniformly expanding
systems on the interval.
For self-conformal systems on $\reals^{n}$, \ref{cond:frostman} and
\ref{cond:thin-annuli} are satisfied
\cite[Theorem~5.1]{huang-2025-exponential}.

For hyperbolic toral automorphisms, we can perturb the potentials
of Gibbs measures satisfying Conditions~\ref{cond:frostman} and
\ref{cond:thin-annuli} and obtain other Gibbs measures satisfying the
conditions (with different constants). Indeed, suppose the Gibbs
measure $\mu_0$ corresponding to a H{\"o}lder continuous potential
$\varphi_0$ satisfies \ref{cond:frostman} and \ref{cond:thin-annuli}
(e.g., take $\mu_0 = \leb$). Let $\psi \in C^{\infty}$ and write
$\varphi_\varepsilon = \varphi_0 + \varepsilon \psi$, and
$\mu_\varepsilon$ for the corresponding Gibbs measure.
Using that balls are sufficiently well-approximated by cylinders for
hyperbolic toral automorphisms, the Gibbs property implies that
exists $\varepsilon_0 > 0$ such that, for all
$\varepsilon < \varepsilon_0$, the measure $\mu_\varepsilon$
satisfies \ref{cond:frostman} and \ref{cond:thin-annuli}.

\subsubsection{Short returns}
Conditions~\ref{cond:sr:explicit} and \ref{cond:sr:implicit} 
are forms of short returns estimates. Informally, the conditions assert
that rapidly recurrent neighbourhoods are relatively rare.
These conditions are used in obtaining strong variance
estimates, which are then used in verifying the assumptions in the
distributional results.
A version of Condition~\ref{cond:sr:explicit} first appears in
Collet~\cite{collet-2001-statistics}, who establishes it using a
maximal function technique. Since then, several versions for various
systems have been considered, e.g., see
\cite{
  gupta-2011-extreme,
  holland-2024-dichotomy, 
  holland-2012-extreme, 
  holland-2016-quantitative 
}.
We verify these conditions for Gibbs measures
corresponding to expanding interval systems, and for hyperbolic toral
automorphisms in Section~\ref{sec:sr:proof}.

We will show that (see Lemma~\ref{lem:r-explicit:sr}), for absolutely
continuous Gibbs measures, \ref{cond:sr:explicit} implies
\[
  \mu(\hat{E}_k \cap \hat{E}_{k+l}) \leqs \mu(\hat{E}_k)^{1+\sremin}
  \quad k \geq 1,\ l=1,\ldots,(\log k)^{\srelog}
\]
for some $\sremin > 0$.
When working with the rescaled recurrence sets, \ref{cond:sr:implicit}
is the more natural condition. Similarly, \ref{cond:sr:implicit}
implies, for more general Gibbs measures,
\[
  \mu(E_k \cap E_{k+l}) \leqs \mu(E_k)^{1+\srimin}
  \quad k \geq 1,\ l=1,\ldots,(\log k)^{\srilog}
\]
for some $\srimin > 0$. In \cite{persson-2026-strong}, a similar
estimate is shown for expanding systems on the interval equipped with
Gibbs measures (see Lemma~6 within), and hyperbolic automorphisms of
$\torus^2$ equipped with the Lebesgue measure (see Lemma~8 within).

\subsubsection{Sequence assumption}
Throughout the present paper, we will impose conditions on our targets.
When considering the distributional law for $(r_n)$, we will impose
conditions on $(r_n)$ since the measure is assumed to be absolutely
continuous. When considering $(\imr_n)$, we will instead impose
conditions on the reference sequence $(M_n)$. For the shrinking target
problem, we impose conditions on the measure $\mu(B_n)$ of the targets.
Usually, we assume that the targets satisfy \ref{cond:sequence}
for some $\seqlow$ and $\seqhigh > 1$. In particular, $\seqlow$ will be
chosen such that
\[
  \sum_{k=1}^{n} r_k^{\dim X} = \infty,
  \quad
  \sum_{k=1}^{\infty} M_k = \infty
  \quad \text{or} \quad
  \sum_{k=1}^{\infty} \mu(B_k) = \infty.
\]
The most restrictive part of Condition~\ref{cond:sequence} is the lower
decay rate $n^{-\seqlow}$, which is used to ensure a strong enough lower
bound on the variance growth. As such, we cannot relax this condition
in the present paper to just the divergence of the above sums.
The upper decay rate $(\log n)^{-\seqhigh}$ is used to control
short-term correlations, where the mixing conditions of \ref{cond:doc}
or \ref{cond:multiple-doc} are too weak to give a useful bound.

\section{Distributional laws for Poincar\'e recurrence}
\label{sec:clt}
We devote this section to distributional laws for recurrence.
We begin with the non-standard ALT, and then proceed to the
CLT for the appropriately altered sum.

\subsection{The ALT for recurrence}
\label{sec:clt:alt}
We begin by stating the ALT result.
Recall the auxiliary sets
$\exE_n = \{x \in X : T^{n}x \in B(x,r_n)\}$, and let
\begin{align*}
  \exsig_n^2 &\coleq \var\biggl( x \mapsto \sum_{k=1}^{n}
    \charfun_{B(x,r_k)}(T^{k}x) - \mu(B(x,r_k)) \biggr), \\
  \exs_n^2(y) &\coleq \var\biggl( x \mapsto \sum_{k=1}^{n}
    \charfun_{B(y,r_k)} (T^{k}x) - \mu(B(y,r_k)) \biggr).
\end{align*}

\begin{theorem}[ALT for recurrence]
\label{thm:alt}
  Assume that $(X,T,\mu)$ is an m.p.s.\ satisfying
  Condition~\ref{cond:multiple-doc} and such that the following holds:
  \begin{enumerate}[label=(\roman*)]
    \item
      $d\mu = h \diff\leb$ for a density $h \in \Ell{q}(\leb)$
      satisfying $h \geq c > 0$ for some $q \geq 2$ and $c > 0$;
    \item \label{item:alt:r}
      $(r_n)$ satisfies Condition~\ref{cond:sequence} for some
      $\seqlow \in (0,1/\dim X)$ and $\seqhigh > p/\dim X$, where $p$ is
      the H{\"o}lder conjugate of $q$;
    \item \label{item:alt:clt}
      The corresponding shrinking target problem satisfies the CLT,
      i.e.,
      \[
        \frac{1}{\exs_n(y)} \sum_{k=1}^{n} \bigl( \charfun_{B(y,r_k)}
        \circ T^{k} - \mu(B(y,r_k)) \bigr) \todistr N(0,1)
        \quad \text{$\mu$-a.e.\ $y \in X$};
      \]
    \item \label{item:alt:var}
      The variances satisfy the limit
      \[
        \lim_{n \to \infty}
        \int_{X} \biggl| \frac{\exs_n(y)}{\exsig_n}
        - \frac{\sqrt{h(y)}}{\sqrt{\mu(h)}} \biggr| \diff\mu(y) = 0.
      \]
  \end{enumerate}
  Then the sum
  \begin{equation}
  \label{eq:alt:sum}
    \frac{1}{\exsig_n} \sum_{k=1}^{n}
    \charfun_{B(x,r_k)}(T^{k}x) - \mu(B(x,r_k))
  \end{equation}
  satisfies \eqref{alt}, i.e., the characteristic function of the
  sum converges to
  \[
    \int_{X} \exp \Bigl( - \frac{h(x)}{2\mu(h)} t^2 \Bigr) \diff\mu(x)
    \quad \text{for all $t \in \reals$}.
  \]
  This limit is not Gaussian unless $h = 1$.
\end{theorem}

Since $h$ is bounded away from $0$, the limiting distribution has
probability density function $f \colon \reals \to (0,\infty)$ given by
\[
  f(t) = \frac{1}{\sqrt{2\pi}} \int_X
  \frac{\sqrt{\mu(h)}}{\sqrt{h(y)}}
  \exp\Bigl(-\frac{\mu(h)}{2h(y)}t^2\Bigr) \diff\mu(y)
\]
by the Fourier inversion formula.

Note that $q \geq 2$ is necessary for the limit law in the theorem to
be well-defined since $\mu(h) = \leb(h^2)$. Under
Condition~\ref{cond:sr:explicit}, we prove
hypothesis~\ref{item:alt:var} when
$q > 2$ in Proposition~\ref{prop:r-explicit:L1}.
By further restricting the upper decay rate of the sequence $(r_n)$,
we verify \ref{cond:sr:explicit} in
\text{\S}\ref{sec:srt:proof:explicit} for selected systems.

\begin{remark}[Simplification for expanding maps]
\label{rem:alt:bv}
  In the case of expanding systems on the interval,
  Condition~\ref{cond:multiple-doc} can be relaxed to a version
  involving $\bv$ observables in place of H{\"o}lder continuous ones.
  As a result, the proof of Theorem~\ref{thm:alt} becomes simpler.
\end{remark}

\begin{remark}[Non-standard centring]
  We note that we use a non-standard centring in \eqref{eq:alt:sum},
  i.e., we centre by subtracting $\mu(B(x,r_k))$ which depends on $x$,
  rather than
  $\mu(\exE_k) = \int \charfun_{B(x,r_k)}(T^{k}x) \diff\mu(x)$.
  The fact that $\mu$ is not translation-invariant means that the
  former is the correct centring of the observables.
  We also note that He~\cite{he-2024-quantitative} uses a non-standard
  centring when proving
  \[
    \lim_{n \to \infty} \frac{\sum_{k=1}^{n}
    \charfun_{B(x,r_k)}(T^{k}x)}{\sum_{k=1}^{n} \mu(B(x,r_k))}
    = 1 \quad \text{$\mu$-a.e.\ $x$}.
  \]
\end{remark}

\begin{remark}
  We compare Theorem~\ref{thm:alt} with \eqref{eq:poisson} due to
  Holland and Todd~\cite{holland-2025-distributional}
  since both results exhibit the same averaging phenomenon. That is,
  both limiting distributions are averages of the corresponding
  pointwise laws weighted by the density of the measure.
  In \cite{holland-2025-distributional}, the
  pointwise law is Poisson with variance $\tau h(x)$,
  while in our case, the pointwise law is Gaussian with variance
  $h(x)/\mu(h)$. In both cases, the averaging occurs since $\mu$ is not
  translation-invariant while the targets depend on $x$.

  Although not directly applicable to our non-autonomous setting, Conze
  and Le~Borgne~\cite{conze-2011-limit} also exhibits an averaging
  phenomenon in \eqref{eq:conze}, which is also an average of Gaussian
  distributions.
\end{remark}

\subsection{The CLT for recurrence}
\label{sec:clt:clt}
We now state the CLT for the sums of featuring the appropriately scaled
radii. The rescaling allows us to both recover the CLT and obtain
results for measures that do not need to be absolutely continuous.
Recall that the auxiliary sets $(\imE_n)$ are defined with respect to a
fixed reference sequence $(M_n)$. That is,
\[
  \imE_n = \{x \in X : T^{n}x \in B(x,\imr_n(x))\},
  \quad \text{where } \mu\bigl( B(x,\imr_n(x)) \bigr) = M_n.
\]
Furthermore, let
\[
  \imsig_n^2 \coleq \var \biggl( \sum_{k=1}^{n} \charfun_{\imE_k}
    \biggr)
  \quad \text{and} \quad
  \ims_n^2(y) \coleq \var \biggl( \sum_{k=1}^{n}
    \charfun_{B(y,\imr_k(y))} \circ T^{k} \biggr).
\]

\begin{theorem}[CLT for recurrence]
\label{thm:clt}
  Assume that $(X,T,\mu)$ is an m.p.s.\ satisfying
  Conditions~\ref{cond:multiple-doc}
  such that the following holds:
  \begin{enumerate}[label=(\roman*)]
    \item \label{item:clt:measure}
      $\mu$ satisfies Conditions~\ref{cond:frostman} and
      \ref{cond:thin-annuli};
    \item
      $(M_n)$ satisfies Condition~\ref{cond:sequence} for some
      $\seqlow \in (0,1)$ and $\seqhigh > 1$;
    \item \label{item:clt:clt}
      The corresponding shrinking target problem satisfies the CLT,
      i.e.,
      \[
        \frac{1}{\ims_n(y)} \sum_{k=1}^{n} \bigl(
          \charfun_{B(y,\imr_k(y))} \circ T^{k} - M_k \bigr) \todistr
          N(0,1)
        \quad \text{$\mu$-a.e.\ $y \in X$};
      \]
    \item \label{item:clt:var}
      The variances satisfy the distributional limit
      \[
        \frac{\ims_n^2}{\imsig_n^2} \todistr 1,
      \]
      where we consider $\ims_n^2 \colon y \mapsto \ims_n^2(y)$ as a
      random variable on $(X,\mu)$.
  \end{enumerate}
  Then
  \[
    \frac{1}{\imsig_n} \sum_{k=1}^{n}
    \charfun_{B(x,\imr_k(x))}(T^{k}x) - \mu\bigl( B(x,\imr_n(x)) \bigr)
  \]
  converges in distribution to a standard Gaussian variable.
\end{theorem}

Under Condition~\ref{cond:sr:implicit}, we can verify
hypothesis~\ref{item:clt:var} while slightly restricting the upper
decay rate of $(M_n)$.
For selected systems, we verify Condition~\ref{cond:sr:implicit} in
\text{\S}\ref{sec:srt:proof:implicit} with no further restrictions on
$(M_n)$.
See Remark~\ref{rem:alt:bv} above for the $\bv$ simplification, which
also applies here.

\begin{remark}[Difference between ALT and CLT]
  The main difference in the hypotheses of Theorems~\ref{thm:alt} and
  \ref{thm:clt} is in hypothesis \ref{item:clt:measure}, which allows
  one to consider more general measures.
  It is precisely this relaxation of absolute continuity of
  $\mu$ that leads us to consider the rescaled balls $B(x,\imr_n(x))$. 
  In addition to being a more tractable set for non-translation
  invariant measures, the rescaled balls allow us to relax
  $\Ell{1}(\mu)$ convergence to convergence in distribution in
  hypothesis~\ref{item:clt:var}. (Since the limiting distribution is
  constant, convergence in distribution is equivalent to convergence in
  measure.)
  By considering the rescaled balls, we also recover the `expected'
  result.
  Thus, the rescaled balls $B(x,\imr_n(x))$ are the `correct' targets
  to consider if one wants to recover the `expected results'.
  They also have the added benefit of relaxing the hypotheses.
\end{remark}

\begin{remark}
  We note that Theorem~\ref{thm:clt} is the first CLT for recurrence in
  the setting of the targets $B(x,\imr_n(x))$.
  The only previous limit laws for $(\imE_n)$ are strong
  Borel--Cantelli lemmas (SLLNs) as established in
  \cite{
    huang-2025-exponential,
    persson-2023-strong,
    persson-2026-strong,
    ars-2025-recurrence
  }.
\end{remark}

\section{The ASIP for the shrinking target problem}
\label{sec:asip}
We devote this section to the ASIP results. Our main contribution is
Theorem~\ref{thm:asip:hyperbolic:targets} -- the ASIP for the shrinking
target problem in Axiom~A systems. In the proof, we reduce the problem
to a more general ASIP for non-autonomous H{\"o}lder continuous
observables (Theorem~\ref{thm:asip:hyperbolic}), which in turn is
reduced to a corresponding for non-invertible systems
(Theorem~\ref{thm:asip:non-invertible:sup-growth}).
We also obtain the ASIP for the shrinking target problem for systems
for which the transfer operator is quasi-compact on the space of
H{\"o}lder continuous functions
(Theorem~\ref{thm:asip:non-invertible:targets}).

\subsection{The ASIP for Axiom~A systems}
We begin with the Axiom~A results, stating the general result first.

\begin{theorem}
\label{thm:asip:hyperbolic}
  Suppose that $T \colon X \to X$ is a topologically mixing Axiom~A
  diffeomorphism, and $\mu$ is a Gibbs measure corresponding to a
  H{\"o}lder continuous potential.
  Assume $(\varphi_n)$ is a sequence of $\alpha$\nobreakdash-H{\"o}lder
  continuous functions such that
  \begin{enumerate}[label=(\roman*)]
    \item
      there exists $p > 0$ such that
      \[
        \sup_{n \geq 1} \supnorm{\varphi_n} < \infty
        \quad \text{and} \quad
        \holconst{\varphi_n}{\alpha} \leqs n^{p};
      \]
    \item
      there exists $\delta > 1/2$ such that
      \[
        \asipvar_n^2 \coleq \var\biggl( \sum_{k=1}^{n} \varphi_k \circ
        T^{k}\biggr) \geqs n^{\delta}.
      \]
  \end{enumerate}
  Then $(\tilde{\varphi}_n \circ T^{n})$ satisfies \eqref{asip} for any
  $\beta < \min \{\frac{1}{2}, 1 - \frac{1}{2\delta}\}$ with
  $\sum_{k=1}^{n} \EE[Z_k^2] = \asipvar_n^2 + \bigo(\asipvar_n (\log
  \asipvar_n)^2)$.
\end{theorem}

\begin{remark}
  An analysis of the proof of Theorem~\ref{thm:asip:hyperbolic} gives
  the following improvement.
  If we replace $\sup_{n \geq 1} \supnorm{\varphi_n} < \infty$
  with $\supnorm{\varphi_n} \leqs (\log n)^{q}$ for some $q > 0$,
  then the conclusion holds with error
  $\bigo(\asipvar_n (\log \asipvar_n)^{2+q})$.
\end{remark}

\begin{remark}
  We compare with the existing works
  \cite{chen-2018-non-stationary,haydn-2017-almost}.
  Theorem~\ref{thm:asip:hyperbolic} improves
  \cite[Corollary~6.2]{haydn-2017-almost} in two ways.
  Firstly, we allow the sequence $(\holconst{\varphi_n}{\alpha})$
  of H{\"o}lder constants to grow at a polynomial rate.
  Secondly, we relax the condition on $\delta$ from
  $\delta > \frac{\sqrt{17}-1}{4} \approx 0.781$ to $\delta > 1/2$.
  In \cite[Theorem~1]{chen-2018-non-stationary},
  $2$-dimensional non-uniformly hyperbolic systems with
  singularities are considered, but with $\mu$ restricted to the
  SRB measure. It is also required that $\delta > 1/2$.
\end{remark}

Theorem~\ref{thm:asip:hyperbolic} can now be applied to the shrinking
target problem for targets that can be sufficiently well-approximated
by H{\"o}lder continuous functions.

\begin{theorem}
\label{thm:asip:hyperbolic:targets}
  Suppose that $T \colon X \to X$ is a topologically mixing Axiom~A
  diffeomorphism, and $\mu$ is a Gibbs measure corresponding to a
  H{\"o}lder continuous potential.
  Let $(B_n)$ be a sequence of balls with centres in $\supp\mu$, and
  write $\asipmeas_n = \mu(B_n)$ and
  $\asipvar_n^2 = \var( \sum_{k=1}^{n} \charfun_{B_k} \circ T^{k})$.
  Assume that
  \begin{enumerate}[label=(\roman*)]
    \item
      $\mu$ satisfies Condition~\ref{cond:frostman} and
      \ref{cond:thin-annuli};
    \item \label{item:asip:hyperbolic:targets:var}
      $(\asipmeas_n)$ satisfies Condition~\ref{cond:sequence} for some
      $\seqlow < 1/2$ and $\seqhigh > 1$.
  \end{enumerate}
  Then $(\tilde{\charfun}_{B_n} \circ T^{n})$ satisfies \eqref{asip}
  for any $\beta < 1 - \frac{1}{2(1-\seqlow)}$ with
  $\sum_{k=1}^{n} \EE[Z_k^2] = \asipvar_n^2 + \bigo(\asipvar_n (\log
  \asipvar_n)^2)$.
\end{theorem}

\begin{remark}
  Theorem~\ref{thm:asip:hyperbolic:targets} establishes the ASIP
  for the shrinking target problem for Axiom~A systems equipped with
  Gibbs measures, which is, to the author's knowledge, the first such
  result.
  The analogous result for systems whose transfer operator is
  quasi-compact on $\bv$ is established in
  \cite[Theorem~5.1]{haydn-2017-almost}. For $2$-dimensional
  non-uniformly hyperbolic systems with singularities,
  \cite[Theorem~18]{chen-2018-non-stationary} proves the corresponding
  result for the SRB measure and targets whose boundaries are contained
  in the singular set.
\end{remark}

\subsection{The ASIP for non-invertible systems}
As mentioned previously, the strategy for proving the ASIP for Axiom~A
systems relies on reducing the problem to that of non-invertible
systems via the Markov partition provided by the Axiom~A structure.
This is done using a modified `Sinai trick', which introduces
a logarithmic growth of $\supnorm{\varphi_n^{*}}$ of the corresponding
functions $\varphi_n^{*}$ on the shift space due to the polynomial
growth of $\holconst{\varphi_n}{\alpha}$.
We account for this in the following result.

\begin{theorem}
\label{thm:asip:non-invertible:sup-growth}
  Suppose that $(X,T,\mu)$ is an m.p.s.\ satisfying
  Condition~\ref{cond:transfer-operator}.
  Assume $(\varphi_n)$ is a sequence of $\alpha$\nobreakdash-H{\"o}lder
  continuous observables such that
  \begin{enumerate}[label=(\roman*)]
    \item
      there exists $p > 0$ such that
      \[
        \supnorm{\varphi_n} \leqs \log n
        \quad \text{and} \quad
        \holconst{\varphi_n}{\alpha} \leqs n^{p};
      \]
    \item
      there exists $\delta > 1/2$ such that
      \[
        \asipvar_n^2 \coleq \var\biggl( \sum_{k=1}^{n} \varphi_k \circ
        T^{k}\biggr) \geqs n^{\delta}.
      \]
  \end{enumerate}
  Then $(\tilde{\varphi}_n \circ T^{n})$ satisfies \eqref{asip} for any
  $\beta < \min \{\frac{1}{2}, 1 - \frac{1}{2\delta}\}$ with 
  $\sum_{k=1}^{n} \EE[Z_k^2] = \asipvar_n^2 + \bigo(\asipvar_n (\log
  \asipvar_n)^2)$.
\end{theorem}

\begin{remark}
  Theorem~\ref{thm:asip:non-invertible:sup-growth} relaxes the uniform
  bounds in $\supnorm{\cdot}$ and $\holconst{\cdot}{\alpha}$ present in 
  \cite[Theorem~6.1]{haydn-2017-almost}. By doing this, we can then
  use our results for the Axiom~A setting. At the same time, we
  relax the condition
  $\delta > \frac{\sqrt{17}-1}{4} \approx 0.781$ to $\delta > 1/2$.
\end{remark}

The logarithmic growth of $\supnorm{\varphi_n}$ in
Theorem~\ref{thm:asip:non-invertible:sup-growth} matches what arises
from the Sinai-trick reduction in the proof of
Theorem~\ref{thm:asip:hyperbolic}.
When the observables take the form of indicator functions of balls, we
can relax the lower bound $\delta > 1/2$ to $\delta > 0$.

\begin{theorem}
\label{thm:asip:non-invertible:targets}
  Suppose that $(X,T,\mu)$ is an m.p.s.\ satisfying
  Condition~\ref{cond:transfer-operator}.
  Let $(B_n)$ be a sequence of balls with centres in $\supp\mu$, and
  write $\asipmeas_n = \mu(B_n)$ and
  $\asipvar_n^2 = \var( \sum_{k=1}^{n} \charfun_{B_k} \circ T^{k} )$.
  Assume that
  \begin{enumerate}[label=(\roman*)]
    \item
      $\mu$ satisfies Condition~\ref{cond:frostman} and
      \ref{cond:thin-annuli};
    \item
      $(\asipmeas_n)$ satisfies Condition~\ref{cond:sequence} for some
      $\seqlow \in (0,1)$ and $\seqhigh > 1$.
  \end{enumerate}
  Then $(\tilde{\charfun}_{B_k} \circ T^{k})$ satisfies \eqref{asip}
  for any $\beta < 1/2$ with
  $\sum_{k=1}^{n} \EE[Z_k^2] = \asipvar_n^2 + \bigo(\asipvar_n \log
  \asipvar_n)$.
\end{theorem} 

\begin{remark}
  Under the additional assumption that Conditions~\ref{cond:frostman}
  and \ref{cond:thin-annuli} hold,
  Theorem~\ref{thm:asip:non-invertible:targets} extends
  \cite[Theorem~5.1]{haydn-2017-almost} to systems where the transfer
  operator is quasi-compact on the space of H{\"o}lder continuous
  functions rather than the space of $\bv$ functions.
\end{remark}

\section{Variance estimates}
\label{sec:variance}
The results of Section~\ref{sec:clt} feature the variance estimates in
the hypotheses.
We devote this section to estimates of the variances, culminating in
results verifying the hypotheses.
We begin with measure, correlation and variance estimates for
\text{\S}\ref{sec:clt:alt}, that is, estimates involving the sets
\[
  \exE_n \coleq \{x \in X : T^{n}x \in B(x,r_n)\}.
\]
We see that correlation -- and as a result variance -- estimates for
the sets $\exE_n$ are difficult to work with, even when the measure is
absolutely continuous but not equal to the Lebesgue measure.
For example, it is generally not true that
$\mu(\exE_k \cap \exE_j) \approx \mu(\exE_k) \mu(\exE_j)$
when $k$ and $j-k$ are large.

As we will see, the sets
\[
  \imE_n \coleq \{x \in X : T^{n}x \in B(x,\imr_n(x))\}
\]
with the rescaled radii satisfy stronger correlation estimates, in
particular, \emph{it is true} that
$\mu(\imE_k \cap \imE_j) \approx \mu(\imE_k) \mu(\imE_j)$
when $k$ and $j-k$ are large. We can therefore obtain stronger variance
estimates as well.

\subsection{Variance estimates for Section~\ref{sec:clt:alt}}
\label{sec:var:explicit}
We state estimates for the recurrence problem in
\text{\S}\ref{sec:clt:alt}, that is when $(r_k)$ is a fixed sequence in
$[0,\infty)$.
For the shrinking target problem, the hitting sets are of the form
$A_n = \{x \in X : T^{n}x \in B(y_n,r_n)\} T^{-n}B(y_n,r_n)$. Since $T$
preserves $\mu$, we obtain $\mu(A_n) = \mu(B(y_n,r_n))$ directly.
In the recurrence setting, the targets depend on the initial point and
we can therefore not directly use the measure-preserving property.
Nevertheless, for sufficiently mixing systems, one can obtain estimates
for $\mu(\exE_n)$ in terms of $\int \mu(B(x,r_n)) \diff\mu(x)$.
The following lemma can be found in the work of
Kirsebom, Kunde and Persson~\cite{kirsebom-2023-shrinking} for
systems satisfying decay of correlations for $\bv$ vs.\ $\Ell{1}$
observables (see the proof of Theorem~\textup{C} within).
A simple rewriting of \cite[Proposition~4.1]{ars-2025-recurrence}
establishes the lemma for for systems satisfying
decay of correlations for H{\"o}lder continuous observables (that is,
Condition~\ref{cond:doc}).

\begin{lemma}[\cite{kirsebom-2023-shrinking}]
\label{lem:r-explicit:meas}
  Suppose $(X,T,\mu)$ is an m.p.s.\ and assume that
  \begin{enumerate}[label=(\roman*)]
    \item
      Condition~\ref{cond:doc} holds;
    \item
      $d\mu = h \diff\leb$ with density $h \in \Ell{q}(\leb)$ such
      that $h \geq c > 0$ for some $q > 1$ and $c > 0$;
    \item
      $r_k \to 0$.
  \end{enumerate}
  Then there exists a constant $\lexmeas > 0$ such that
  \[
    \Bigl| \mu(\exE_n) - \int_X \mu(B(x,r_n)) \diff\mu(x) \Bigr|
    \leqs e^{-\lexmeas n}
    \quad \text{for all $n \in \naturals$}.
  \]
\end{lemma}

Using higher-order mixing, we obtain the following estimate. It is
similar to an application of \cite[Lemma~3.1]{kirsebom-2023-shrinking},
but we use a weaker mixing condition. That is, we use
Condition~\ref{cond:multiple-doc} restricted to three observables
(i.e., for $m+m' = 3$) instead of $3$-fold mixing, which results in
some polynomial terms that do not appear in
\cite[Lemma~3.1]{kirsebom-2023-shrinking}.

\begin{lemma}
\label{lem:r-explicit:correlations}
  Suppose $(X,T,\mu)$ is an m.p.s.\ and assume that
  \begin{enumerate}[label=(\roman*)]
    \item
      Condition~\ref{cond:multiple-doc} holds when restricted to
      three observables;
    \item \label{item:r-explicit:correlations:2}
      $d\mu = h \diff\leb$ with density $h \in \Ell{q}(\leb)$ such
      that $h \geq c > 0$ for some $q > 1$ and $c > 0$;
    \item \label{item:r-explicit:correlations:3}
      $r_k \to 0$.
  \end{enumerate}
  Then there exists $\lexcorsum > 2$ and $\lexcorcoeff, \lexcor > 0$
  such that, for all $k < j \leq n$,
  \begin{multline*}
    \Bigl| \mu(\exE_k \cap \exE_j) - \int_{X} \mu(B(x,r_k))
      \mu(B(x,r_j)) \diff\mu(x) \Bigr| \\
    \leqs n^{-\lexcorsum}
      + n^{\lexcorcoeff} \bigl( e^{- \lexcor k} + e^{- \lexcor (j-k)}
      \bigr).
  \end{multline*}
\end{lemma}
The existence of $\lexcorsum > 2$ is important for estimates involving double
sums of correlations since $n^{2-\lexcorsum} \to 0$ as $n \to \infty$.

\begin{remark}
Note that the conclusion is not strictly a correlation estimate, as
\[
  \int \mu(B(x,r_k)) \mu(B(x,r_k)) \diff\mu(x)
  \not\approx
  \int \mu(B(x,r_k)) \diff\mu(x) \int \mu(B(x,r_k)) \diff\mu(x)
\]
in general. In fact, this is a large reason for why the sets $\imE_k$
with rescaled radii are more tractable since they \emph{do} satisfy a
correlation estimate (see Lemma~\ref{lem:r-implicit:correlations}
below).
\end{remark}

\begin{remark}
  It is possible to relax the absolute continuity of $\mu$ in
  Lemmas~\ref{lem:r-explicit:meas} and
  \ref{lem:r-explicit:correlations}
  to Conditions~\ref{cond:frostman} and \ref{cond:thin-annuli} instead.
  Replacing the radius condition
  \[
    r_k \to 0
  \]
  with an average measure condition
  \[
    \lim_{k \to \infty} \int_X \mu(B(x,r_k)) \diff\mu(x) = 0,
  \]
  we obtain the same conclusion.
  We do not do so in the present paper since we always assume that
  $\mu$ is absolutely continuous in Section~\ref{sec:clt:alt}.
\end{remark}

We now proceed to the variance estimates. Recall that in
Section~\ref{sec:clt:alt}, the sum we consider is
\[
  \sum_{k=1}^{n} \bigl(\charfun_{B(x,r_k)}(T^{k}x)
  - \mu(B(x,r_k))\bigr).
\]
Recall that
\begin{align*}
  \exsig_n^2 &\coleq \var\biggl( x \mapsto \sum_{k=1}^{n}
    \charfun_{B(x,r_k)}(T^{k}x) - \mu(B(x,r_k)) \biggr), \\
  \exs_n^2(y) &\coleq \var\biggl( x \mapsto \sum_{k=1}^{n}
    \charfun_{B(y,r_k)} (T^{k}x) - \mu(B(y,r_k)) \biggr).
\end{align*}
For convenience, we also let
\[
  \exesum_n = \sum_{k=1}^{n} \mu(\exE_k).
\]
The first result is a lower growth bound for $\exsig_n^2$ in terms of
$\exesum_n$.

\begin{lemma}
\label{lem:r-explicit:liminf}
  Suppose $(X,T,\mu)$ is an m.p.s.\ and assume that
  \begin{enumerate}[label=(\roman*)]
    \item
      Condition~\ref{cond:multiple-doc} holds when restricted to three
      observables;
    \item \label{item:r-explicit:liminf:2}
      $d\mu = h \diff\leb$ with density $h \in \Ell{q}(\leb)$ such
      that $h \geq c > 0$ for some $q > 1$ and $c > 0$;
    \item \label{item:r-explicit:liminf:3}
      $(r_n)$ satisfies Condition~\ref{cond:sequence} for some
      $\seqlow \in (0,1/\dim X)$ and $\seqhigh > p/\dim X$,
      where $p$ is the H{\"o}lder conjugate of $q$.
  \end{enumerate}
  Then
  \[
    \liminf_{n \to \infty} \frac{\exsig_n^2}{\exesum_n}
    \geq 1.
  \]
  In particular, $\exsig_n^2 \geqs n^{1-\seqlow \dim X}$.
\end{lemma}

Under the additional assumption of Condition~\ref{cond:sr:explicit} and
increased regularity of the density, we obtain $\Ell{1}$-convergence of
the quantity $\exs_n^2/\exsig_n^2$.

\begin{proposition}
\label{prop:r-explicit:L1}
  Suppose $(X,T,\mu)$ is an m.p.s.\ and assume that
  \begin{enumerate}[label=(\roman*)]
    \item
      Condition~\ref{cond:multiple-doc} holds when restricted to three
      observables;
    \item \label{item:r-explicit:L1:2}
      $d\mu = h \diff\leb$ with density $h \in \Ell{q}(\leb)$ such
      that $h \geq c > 0$ for some $q > 2$ and $c > 0$;
    \item
      Condition~\ref{cond:sr:explicit} holds for some $\srelog > 1$ and
      $\srein, \sreout > 0$;
    \item \label{item:r-explicit:L1:4}
      $(r_n)$ satisfies Condition~\ref{cond:sequence} for some
      $\seqlow \in (0,1/\dim X)$ and
      $\seqhigh > p /(\sremin \dim X)$, where
      $\sremin = \min \{\srein, \sreout\}$ and $p$ is the H{\"o}lder
      conjugate of $q$.
  \end{enumerate}
  Then
  \[
    \lim_{n \to \infty} \int_X \biggl| \frac{\exs_n(y)}{\exsig_n}
    - \frac{\sqrt{h(y)}}{\sqrt{\mu(h)}} \biggr| \diff\mu(y) = 0.
  \]
\end{proposition}

We verify \ref{cond:sr:explicit} in
\text{\S}\ref{sec:srt:proof:explicit} for selected systems.

\subsection{Variance estimates for Section~\ref{sec:clt:clt}}
\label{sec:var:implicit}
We proceed with the estimates for Section~\ref{sec:clt:clt}, which
mirror the previous estimates in \text{\S}\ref{sec:var:explicit}.
Throughout this subsection, we weaken absolute continuity of the
measure to Condition~\ref{cond:frostman} and \ref{cond:thin-annuli}.

The following lemma is established by Kirsebom, Kunde and
Persson~\cite[Lemma~4.1]{kirsebom-2023-shrinking}
for systems satisfying decay of correlations for $\bv$ vs.\ $\Ell{1}$.
Note that \cite[Lemma~4.1]{kirsebom-2023-shrinking} is slightly
stronger in their specific setting.
In \cite[Proposition~4.1]{ars-2025-recurrence}, the author establishes
the lemma for systems satisfying decay of correlations for H{\"o}lder
continuous observables (that is, Condition~\ref{cond:doc}).

\begin{lemma}[{\cite[Lemma~4.1]{kirsebom-2023-shrinking},%
\cite[Proposition~4.1]{ars-2025-recurrence}}]
\label{lem:r-implicit:meas}
  Suppose $(X,T,\mu)$ is an m.p.s.\ and assume that
  \begin{enumerate}[label=(\roman*)]
    \item
      Condition~\ref{cond:doc} holds;
    \item
      $\mu$ satisfies
      Conditions~\ref{cond:frostman} and \ref{cond:thin-annuli};
    \item
      $\imM_n \to 0$.
  \end{enumerate}
  Then there exists a constant $\limmeas > 0$ such that
  \[
    | \mu(\imE_n) - \imM_n | \leqs e^{-\limmeas n}
    \quad \text{for all $n \in \naturals$}.
  \]
\end{lemma}

Using higher-order mixing, one can obtain correlation estimates.
The following lemma is similar to
\cite[Lemma~4.2]{kirsebom-2023-shrinking},
\cite[Lemma~3]{persson-2026-strong} and
\cite[Proposition~4.2]{ars-2025-recurrence}.
Compared to those works, we relax the $3$-fold decorrelation conditions
and instead use Condition~\ref{cond:multiple-doc} restricted to three
observables.
As a result, our estimates feature polynomially terms that do not
appear in
\cite{kirsebom-2023-shrinking,persson-2026-strong,ars-2025-recurrence}.

\begin{lemma}
\label{lem:r-implicit:correlations}
  Suppose $(X,T,\mu)$ is an m.p.s.\ and assume that
  \begin{enumerate}[label=(\roman*)]
    \item
      Condition~\ref{cond:multiple-doc} holds when restricted to three
      observables;
    \item
      $\mu$ satisfies
      Conditions~\ref{cond:frostman} and \ref{cond:thin-annuli};
    \item \label{item:r-implicit:correlations:3}
      $\imM_n \to 0$.
  \end{enumerate}
  then there exists $\limcorsum > 2$ and $\limcorcoeff, \limcor > 0$
  such that, for all $k < j \leq n$,
  \begin{equation*}
    | \mu(\imE_k \cap \imE_j) - \mu(\imE_k) \mu(\imE_j) |
    \leqs n^{-\limcorsum} + n^{\limcorcoeff}
      \bigl( e^{- \limcor k} + e^{- \limcor (j-k)} \bigr).
  \end{equation*}
\end{lemma}

Compared to Lemma~\ref{lem:r-explicit:correlations}, we see that
$\imE_k \cap \imE_j$ decorrelates when $k$ and $j-k$ are large.
As with Lemma~\ref{lem:r-explicit:correlations}, the existence of
$\limcorsum > 2$ is important for estimating double sums.

We now proceed to the variance estimates. Recall that in
Section~\ref{sec:clt:clt}, the sum we consider is
\[
  \sum_{k=1}^{n} \bigl(\charfun_{B(x,\imr_k(x))}(T^{k}x)
  - \mu(B(x,\imr_k(x)))\bigr).
\]
Recall that
\[
  \imsig_n^2 = \var\biggl( \sum_{k=1}^{n} \charfun_{\imE_k} \biggr)
  \quad \text{and} \quad
  \ims_n^2(y) = \var\biggl( \sum_{k=1}^{n} \charfun_{B(y,\imr_k(y))}
    \circ T^{k} \biggr).
\]
For convenience, we also let
\[
  \imesum_n = \sum_{k=1}^{n} \mu(\imE_k).
  \quad \text{and} \quad
  \immsum_n = \sum_{k=1}^{n} \mu(\imM_k).
\]
The first result is a lower growth bound for $\imsig_n^2$ in terms of
$\imesum_n$.

\begin{lemma}
\label{lem:r-implicit:liminf:sigma}
  Suppose $(X,T,\mu)$ is an m.p.s.\ and assume that
  \begin{enumerate}[label=(\roman*)]
    \item
      Condition~\ref{cond:multiple-doc} holds when restricted to three
      observables;
    \item
      $\mu$ satisfies
      Conditions~\ref{cond:frostman} and \ref{cond:thin-annuli};
    \item \label{item:r-implicit:liminf:sigma:3}
      $(\imM_n)$ satisfies Condition~\ref{cond:sequence} for some
      $\seqlow \in (0,1)$ and $\seqhigh > 1$.
  \end{enumerate}
  Then
  \[
    \liminf_{n \to \infty} \frac{\imsig_n^2}{\imesum_n} \geq 1.
  \]
  In particular, $\imsig_n^2 \geqs n^{1-\seqlow}$.
\end{lemma}

A similar bound is satisfied for $\ims_n^2$ and $\immsum_n$. Note that
we only use Condition~\ref{cond:doc} and not \ref{cond:multiple-doc} in
the following lemma.

\begin{lemma}
\label{lem:r-implicit:liminf:s}
  Suppose $(X,T,\mu)$ is an m.p.s.\ and assume that
  \begin{enumerate}[label=(\roman*)]
    \item
      Condition~\ref{cond:doc} holds;
    \item
      $\mu$ satisfies
      Conditions~\ref{cond:frostman} and \ref{cond:thin-annuli};
    \item
      $(\imM_n)$ satisfies Condition~\ref{cond:sequence} for some
      $\seqlow \in (0,1)$ and $\seqhigh > 1$.
  \end{enumerate}
  Then
  \[
    \liminf_{n \to \infty} \frac{\ims_n^2(y)}{\immsum_n} \geq 1
    \quad \text{uniformly in } y \in \supp \mu.
  \]
  In particular, $\ims_n^2(y) \geqs n^{1-\seqlow}$ uniformly in
  $y \in \supp\mu$.
\end{lemma}

We also obtain convergence of the average of $\ims_n^2/\imsig_n^2$
without any further assumptions.

\begin{proposition}
\label{prop:r-implicit:mean}
  Suppose $(X,T,\mu)$ is an m.p.s.\ and assume that
  \begin{enumerate}[label=(\roman*)]
    \item
      Condition~\ref{cond:multiple-doc} holds when restricted to three
      observables;
    \item
      $\mu$ satisfies
      Conditions~\ref{cond:frostman} and \ref{cond:thin-annuli};
    \item
      $(\imM_n)$ satisfies Condition~\ref{cond:sequence} for some
      $\seqlow \in (0,1)$ and $\seqhigh > 1$.
  \end{enumerate}
  Then
  \[
    \lim_{n \to \infty} \int_X \frac{\ims_n^2(y)}{\imsig_n^2}
    \diff\mu(y) = 1.
  \]
\end{proposition}

Finally, under the additional assumption of
Condition~\ref{cond:sr:implicit}, we obtain convergence in measure of
$\ims_n^2/\imsig_n^2$.

\begin{proposition}
\label{prop:r-implicit:meas}
  Suppose $(X,T,\mu)$ is an m.p.s.\ and assume that
  \begin{enumerate}[label=(\roman*)]
    \item
      Condition~\ref{cond:multiple-doc} holds when restricted to three
      observables;
    \item
      $\mu$ satisfies
      Conditions~\ref{cond:frostman} and \ref{cond:thin-annuli};
    \item
      Condition~\ref{cond:sr:implicit} holds with
      $\sriin,\sriout > 0$;
    \item \label{item:r-implicit:meas:4}
      $(\imM_n)$ satisfies Condition~\ref{cond:sequence} for some
      $\seqlow \in (0,1)$ and $\seqhigh > 2/\srimin$, where
      $\srimin = \min \{\sriin, \sriout\}$.
  \end{enumerate}
  Then
  \[
    \lim_{n \to \infty} \mu \biggl\{y \in X : \Bigl|
    \frac{\ims_n^2(y)}{\imsig_n^2} - 1  \Bigr| > \varepsilon \biggr\}
    = 0 \quad \text{for all } \varepsilon > 0
  \]
\end{proposition}

For selected systems, we verify Condition~\ref{cond:sr:implicit} in
\text{\S}\ref{sec:srt:proof:implicit} with no further restrictions on
$(\imM_n)$.

\section{H{\"o}lder approximations and correlations lemmas}
\label{sec:supporting}

We state a lemma regarding a H{\"o}lder continuous
approximation of indicator functions of balls as it will be used in
many of the proofs. The proof is similar to that of the approximation
argument in \cite[Lemma~3.1]{kirsebom-2023-shrinking}.
Recall the constants $\frost, \thinan$ and $\thinanr$ appearing in
Conditions~\ref{cond:frostman} and \ref{cond:thin-annuli}.

\begin{lemma}
\label{lem:construction}
  Assume that $\mu$ satisfies Conditions~\ref{cond:frostman} and
  \ref{cond:thin-annuli}, and fix $\alpha \in (0,1]$ and
  $\varepsilon > 0$.
  Suppose that $\abr \colon X \to [0,\infty)$ is Lipschitz continuous
  such that $\lipconst{\abr} \leq 1$ and
  $\abr \leq \thinanr$ on $\supp\mu$.
  Then there exists an $\alpha$\nobreakdash-H{\"o}lder continuous
  function $g \colon X \times X \to [0,1]$ such that
  $\holnorm{g}{\alpha} \leqs \varepsilon^{-1}$,
  \begin{equation}
  \label{eq:construction:upper}
    \charfun_{B(x,\abr(x))}(y) \leq g(x,y)
    \leq \charfun_{B(x,\abr(x)+\varepsilon)}(y)
    \quad \forall x,y \in X
  \end{equation}
  and
  \begin{equation}
  \label{eq:construction:measure}
    \int_{X} \bigl| g(x,y) - \charfun_{B(x,\abr(x))}(y) \bigr|
      \diff\mu(y)
    \leqs \varepsilon^{\frostan}
    \quad \forall x \in \supp\mu,
  \end{equation}
  where $\frostan = \min \{\thinan, \frost\}$ and the implicit constants
  do not depend on $g$ or $\abr$.
  Alternatively, one may replace \eqref{eq:construction:upper} with
  \begin{equation*}
    \charfun_{B(x,\abr(x)-\varepsilon)}(y) \leq g(x,y)
    \leq \charfun_{B(x,\abr(x))}(y)
    \quad \forall x,y \in X.
  \end{equation*}
\end{lemma}

\begin{proof}
  Let $Z = \{(x,y) \in X \times X : d(x,y) < \abr(x) + \varepsilon\}$
  and notice that $Z$ is open by the Lipschitz continuity of $\abr$.
  Define
  \[
    g(x,y) =
    \begin{cases}
      0 & (x,y) \notin Z, \\
      \min \bigl\{1, \varepsilon^{-1} \dist\bigl( (x,y),
      (X \times X) \setminus Z \bigr)\bigr\} & (x,y) \in Z.
    \end{cases}
  \]
  Then $g$ is a bounded Lipschitz continuous function on $X \times X$
  (and therefore $\alpha$\nobreakdash-H{\"o}lder continuous) with
  $\holnorm{g}{\alpha} \leq 2\varepsilon^{-1}$. Furthermore, it is
  clear that \eqref{eq:construction:upper} holds.

  For \eqref{eq:construction:measure}, notice that
  $|g(x,y) - \charfun_{B(x,\chi(x))}(y)|
  \leq \charfun_{A_\varepsilon}$,
  where $A_\varepsilon(x)$ is the annulus
  $\{y : \chi(x) \leq d(x,y) < \chi(x) + \varepsilon\}$.
  Let $x \in \supp\mu$.
  If $\varepsilon \leq \chi(x)$, then
  $\mu(A_\varepsilon(x)) \leqs \varepsilon^{\thinan}$
  by Condition~\ref{cond:thin-annuli}.
  If $\varepsilon > \chi(x)$, then
  \[
    \mu(A_\varepsilon(x))
    \leq \mu\bigl(B(x,\chi(x) + \varepsilon)\bigr)
    \leqs (\chi(x) + \varepsilon)^{\frost}
    \leq 2^{\frost}\varepsilon^{\frost}
  \]
  by Condition~\ref{cond:frostman}.
  The implicit constants in the last two estimates does not depend on
  $g$, $\chi$, or $x$, and we have therefore shown
  \eqref{eq:construction:measure}.
\end{proof}

In order to apply Lemma~\ref{lem:construction} to $\abr = r_k$ or
$\abr = \imr_k$, we must have $r_k, \imr_k \leq \thinanr$.
However, in Sections~\ref{sec:clt}--\ref{sec:variance}, we only assume
$r_k \to 0$ or $\imr_k \to 0$.
The following lemma ensures that, for the purposes of the present
paper, we may always without loss of generality assume
$r_k, \imr_k \leq \thinanr$ for all $k$.

\begin{lemma}
\label{lem:wlog}
  Let $\abr_k, \abr_k' \colon \supp\mu \to [0,\infty)$ be sequences of
  functions, and let
  \begin{align*}
    S_n(x) &= \sum_{k=1}^{n} \charfun_{B(x,\abr_k(x))}(T^{k}x)
    - \mu(B(x,\abr_k(x)), \\
    S_n'(x) &= \sum_{k=1}^{n} \charfun_{B(x,\abr_k'(x))}(T^{k}x)
    - \mu(B(x,\abr_k'(x)).
  \end{align*}
  Assume that there exists $N > 0$ such that $\abr_k = \abr_k'$ for all
  $k \geq N$. Then
  \[
    S_n \sim S_n'
    \quad \text{uniformly on } \supp\mu.
  \]
  Moreover, if either
  $\var(S_n) \to \infty$ or $\var(S_n') \to \infty$, then
  \[
    \var(S_n) \sim \var(S_n')
  \]
  and
  \[
    \lim_{n \to \infty} \biglnorm{\frac{1}{\sqrt{\var(S_n)}} S_n
    - \frac{1}{\sqrt{\var(S_n')}} S_n'}{1}
    = 0.
  \]
\end{lemma}

In our setting, we will apply the lemma as follows.
Let $(\abr_k)$ denote either $(r_k)$ or $(\imr_k)$.
If $\abr_k \to 0$ uniformly, then there exists $N > 0$ such that
$\abr_k \leq \thinanr$ for all $k \geq N$.
Define $\abr_k' = \thinanr$ for
$k < N$ and $\abr_k' = \abr_k$ for $k \geq N$.
Lemma~\ref{lem:wlog} allows us to consider the new sequence $(\abr_k')$
instead of $(\abr_k)$ while preserving the relevant asymptotic
behaviours of the sums and variances.
The benefit of considering $(\abr_k')$ is that
Lemma~\ref{lem:construction} can be applied.

\begin{proof}[Proof of Lemma~\ref{lem:wlog}]
  It follows directly from $\abr_k = \abr_k'$ for all $k \geq N$, that
  \[
    | S_n - S_n' | \leq 2N,
  \]
  which establishes the first part.
  This in turn implies
  \[
    \lvert S_n^2 - (S_n')^2 \rvert
    \leqs \min \{\lvert S_n \rvert, \lvert S_n' \rvert\} + 1
  \]
  and
  \[
    |\var(S_n) - \var(S_n')| \leqs \xi_n + 1,
  \]
  where $\xi_n^2 = \min \{ \var(S_n), \var(S_n')\}$,
  which imply the second part.
\end{proof}

Finally, we state the following lemma, which will be employed several
times throughout the proofs. It uses Lemma~\ref{lem:construction} and
is proven in Section~\ref{sec:proof:supporting}.

\begin{lemma}
\label{lem:correlation}
  Assume that $(X,T,\mu)$ satisfies Condition~\ref{cond:doc}, and fix
  $p > 0$ and $\alpha \in (0,1]$. Then the following holds.
  \begin{enumerate}[label=(\roman*)]
    \item \label{item:lem:correlation:1}
      There exist constants $C_1,C_2 > 0$ such that, for any
      $\alpha$\nobreakdash-H{\"o}lder continuous function
      $g \colon X \times X \to [0,1]$ with
      $\holnorm{g}{\alpha} \leq C_1 n^{p}$,
      \begin{multline*}
        \Bigl| \int_X g(x,T^{k}x) \diff\mu(x) - \iint
          g(x,y) \diff\mu(x)\diff\mu(y) \Bigr| \\
        \leq C_2\bigl( n^{-p} + n^{4(\dim X+3)p/\alpha} e^{-\doc k}
          \bigr)
        \quad \text{for all $0 \leq k \leq n$}.
      \end{multline*}
  \end{enumerate}
  Moreover, if we assume Condition~\ref{cond:multiple-doc}
  restricted to the case of $3$ observables, then we have the
  following.
  \begin{enumerate}[resume*]
    \item \label{item:lem:correlation:2}
      For any $\alpha$-H{\"o}lder continuous functions
      $g,h \colon X \times X \to [0,1]$ with
      $\holnorm{g}{\alpha}, \holnorm{h}{\alpha} \leq C_1 n^{p}$,
      \begin{multline*}
        \Bigl| \int_X g(x,T^{k}x) h(x,T^{j}x) \diff\mu(x)
          - \iint g(y,T^{k}x) h(y,T^{j}x)
          \diff\mu(x)\diff\mu(y) \Bigr| \\
        \leq C_2\bigl( n^{-p} + n^{4(\dim X+3)p/\alpha} e^{-\mdoc k}
          \bigr)
        \quad \text{for all $k \leq n$}.
      \end{multline*}
  \end{enumerate}
\end{lemma}

\section{Proof of distributional laws for recurrence}
\label{sec:proof:clt}
We devote this section to the proof of Theorems~\ref{thm:alt} and
\ref{thm:clt}.
Recall that the hypotheses of the theorems include
Conditions~\ref{cond:multiple-doc}, \ref{cond:frostman} and
\ref{cond:thin-annuli}. In Theorem~\ref{thm:alt}, the measure is
assumed to be absolutely continuous and the assumptions on the density
imply Conditions~\ref{cond:frostman} and \ref{cond:thin-annuli}.
The proofs both use characteristic functions through L\'evy's
continuity theorem. In particular, if $(f_n)$ and $(g_n)$ are sequences
of observables such that $g_n \to g$ in distribution and
$| \psi_{f_n}(t) - \psi_{g_n}(t) | \to 0$ for all $t \in \reals$,
then $f_n \to g$ in distribution.
Another argument that we repeatedly use is the following.
Let $f, g$ be two observables. Then, for all $t \in \reals$,
\begin{equation}
\label{eq:charfun:L1}
  | \psi_f(t) - \psi_g(t) |
  \leq \int_X \bigl| e^{itf(x)} - e^{itg(x)} \bigr| \diff\mu(x)
  \leq |t| \cdot \lnorm{f-g}{1}.
\end{equation}
We begin with a lemma for averaging characteristic functions.

\subsection{A lemma for averaging characteristic functions}
A key step in proving Theorems~\ref{thm:alt} and \ref{thm:clt} is to
write the characteristic functions of the recurrence sums as an average
over characteristic functions of the corresponding hitting sums.
We therefore establish the following lemma for a general sequence of
radius functions $(\abr_n)$, which we will later apply to scalars
$\abr_n = r_n \in \reals$, or implicitly defined radii
$\abr_n = \imr_n$, given by $\mu(B(x,\imr_n(x))) = M_n$.

\begin{lemma}
\label{lem:clt:general-rk}
  Suppose that $(X,T,\mu)$ satisfies
  Conditions~\ref{cond:multiple-doc}, \ref{cond:frostman}
  and \ref{cond:thin-annuli}.
  Suppose that $(\abr_n)$ is a sequence of $\mu$-a.e.\ Lipschitz
  continuous functions $\abr_n \colon X \to [0,\infty)$ such that
  $\lipconst{\abr_n} \leq 1$ and
  $\abr_n \to 0$ uniformly on $\supp \mu$.
  Let $(a_n)$ be a sequence in $(0,\infty)$ such that $a_n \leqs n$.
  Then, for all $t \in \reals$ and all $n \in \naturals$,
  \begin{multline*}
    \biggl| \int_X \exp \biggl( \frac{it}{a_n} \sum_{k=1}^{n}
    \charfun_{B(x,\abr_k(x))}(T^{k}x) - \mu\bigl( B(x,\abr_k(x)) \bigr)
    \biggr) \diff\mu(x) \\
    - \iint \exp \biggl( \frac{it}{a_n} \sum_{k=1}^{n}
    \charfun_{B(y,\abr_k(y))}(T^{k}x) - \mu\bigl( B(y,\abr_k(y)) \bigr)
    \biggr) \diff\mu(x) \diff\mu(y) \biggr| \\
    \leqs_t \frac{\log n}{a_n}.
  \end{multline*}
\end{lemma}

In the proof of Lemma~\ref{lem:clt:general-rk}, we utilise
Condition~\ref{cond:multiple-doc},
but we must make appropriate approximations of the functions
$x \mapsto \charfun_{B(x,\abr_k(x))}(T^{k}x)$.
We localise to remove the dependence on $x$ of the targets
$B(x,\abr_k(x))$, and then we approximate by H{\"o}lder continuous
observables.
The approximation by H{\"o}lder continuous observables is done by using
employing Lemma~\ref{lem:construction}, which relies on
Conditions~\ref{cond:frostman} and \ref{cond:thin-annuli}.
For the localisation, we would like to obtain a partition of $X$ for
which the indicator function of each partition element can be
well-approximated by H{\"o}lder continuous observables. This turns out
to be difficult in general; however, by removing a subset of $X$ with
small measure, Conze and Le~Borgne obtain a partition of the remaining
space with the desired properties. This lemma fixes a slight inaccuracy
in the proof of \cite[Theorem~3.3]{conze-2011-limit}.
The statement and proof were communicated by private correspondence,
and are available in \cite[Lemma~4]{persson-2026-strong}.

\begin{lemma}[Conze--Le~Borgne]
\label{lem:partition}
  Suppose that $X$ is a compact smooth manifold and let $\kappa_0 > 0$.
  Then there exists $C_0 > 0$ such that, for all $n \in \naturals$,
  there exist a subset $Y_n \subset X$ and a partition $\partition_n$
  of $X \setminus Y_n$ consisting entirely of sets with positive
  measure such that
  \begin{align*}
    \# \partition_n &\leq n^{\kappa_0 \dim X}, \\
    |\partition_n| &\leq  C_0n^{-\kappa_0}, \\
    \mu(Y_n) &\leq n^{-\kappa_0}.
  \end{align*}
  Moreover, there exist density functions
  $\rho_{Q,n} \colon X \to [0,\infty)$ such that,
  for all $Q \in \partition_n$,
  \begin{align*}
      \supnorm{\rho_{Q,n}} &\leq C_0 n^{(\dim X + 1)\kappa_0}, \\
      \lipconst{\rho_{Q,n}} &\leq C_0 n^{2(\dim X + 2)\kappa_0}, \\
      \lnorm{\rho_{Q,n} - \mu(Q)^{-1} \charfun_Q}{1}
      &\leq C_0 n^{-\kappa_0}.
  \end{align*}
\end{lemma}

We call $\partition_n$ a \emph{quasi-partition} of $X$ if it
satisfies the conclusion of Lemma~\ref{lem:partition}.
We can now prove Lemma~\ref{lem:clt:general-rk}.
The proof proceeds as follows. We approximate the indicator functions
$\charfun_{B(x,\abr_k(x))}(T^{k}x)$ by H{\"o}lder continuous functions.
We then localise the H{\"o}lder continuous functions using
Lemma~\ref{lem:partition}, and obtain the conclusion of the lemma for
the localisation.

\begin{proof}[Proof of Lemma~\ref{lem:clt:general-rk}]
  We first recall and set up notation.
  Recall the constants $\frost, \thinanr, \thinan, \mdoc > 0$ and
  polynomial $P_r$ from Conditions~\ref{cond:frostman},
  \ref{cond:thin-annuli} and \ref{cond:multiple-doc}.
  Write $\frostan = \min \{\frost, \thinan\}$ and $d_0 = \dim X$.
  By Lemma~\ref{lem:wlog}, we may without loss of generality assume
  that $\abr_n \leq \thinanr$ on $\supp\mu$ for all $n \in \naturals$.
  Denote the balls $B(x,\abr_k(x))$ by $B_k(x)$. For
  $\varepsilon \in \reals$, let
  $B_k^{\varepsilon}(x) = B(x,\abr_k(x) + \varepsilon)$.
  We will use the following notation:
  \begin{align*}
    \psi_n(t) &\text{ char.\ func.\ of }
    x \mapsto \frac{1}{a_n} \sum_{k=1}^{n} \charfun_{B_k(x)}(T^{k}x)
      - \mu(B_k(x)), \\
    \psi_n(t,y) &\text{ char.\ func.\ of }
    x \mapsto \frac{1}{a_n} \sum_{k=1}^{n} \charfun_{B_k(y)}(T^{k}x)
      - \mu(B_k(y)),
  \end{align*}
  where `char.\ func.'\ stands for `characteristic function'.
  The conclusion of the lemma can now be formulated as
  \begin{equation}
  \label{eq:clt:lem:goal}
    \Bigl| \psi_n(t) - \int_X \psi_n(t,y) \diff\mu(y) \Bigr|
    \leqs_t \frac{\log n}{a_n}.
  \end{equation}

  We now introduce the quasi-partition.
  Let $\kappa > 0$ be a constant to be specified later.
  Apply Lemma~\ref{lem:partition} with $\kappa_0 = 2\kappa$ to obtain
  the constant $C_0 > 0$, the set $Y_n$, the partition $\partition_n$ of
  $X \setminus Y_n$, the densities
  $\{\rho_{Q,n}\}_{Q \in \partition_n}$ and the corresponding
  estimates. For each $Q \in \partition_n$, we may select a point
  $y_Q \in Q \cap \supp\mu$ since $Q$ has positive measure.

  Let $\varepsilon_n = C_0 n^{-\kappa}$ and apply
  Lemma~\ref{lem:construction} with $\varepsilon = \varepsilon_n$.
  We obtain Lipschitz continuous functions
  $g_{n,k}, h_{n,k} \colon X \times X \to [0,1]$ such that we have
  the regularity properties
  \begin{equation}
  \label{eq:clt:lemma:lip}
    \lipnorm{g_{n,k}} \leqs n^{\kappa}
    \quad \text{and} \quad
    \lipnorm{h_{n,k}} \leqs n^{\kappa},
  \end{equation}
  the inequalities
  \begin{align}
  \label{eq:clt:lemma:g-approx}
    \charfun_{B_k(x)}(y)
    &\leq g_{n,k}(x,y) \leq \charfun_{B_k^{\varepsilon_n}(x)}(y)
    &&\quad (x,y \in X) \\
  \nonumber
    \charfun_{B_k^{-\varepsilon_n}(x)}(y)
    &\leq h_{n,k}(x,y) \leq \charfun_{B_k(x)}(y)
    &&\quad (x,y \in X)
  \end{align}
  and the estimates
  \begin{align}
    \label{eq:clt:lemma:g-measure}
    \Bigl| \int_{X} g_{n,k}(x,y) \diff\mu(y) - \mu(B_k(x)) \Bigr|
    &\leqs n^{-\frostan\kappa},
    && \text{uniformly in } x \in \supp\mu \\
    \label{eq:clt:lemma:h-measure}
    \biggl| \int_{X} h_{n,k}(x,y) \diff\mu(y) - \mu(B_k(x)) \biggr|
    &\leqs n^{-\frostan\kappa}
    && \text{uniformly in } x \in \supp\mu.
  \end{align}
  Mirroring the notation of $\psi_n(t)$ and $\psi_n(t,y)$, write
  \begin{align*}
    \hat{\psi}_n(t) &\text{ char.\ func.\ of }
      x \mapsto \frac{1}{a_n} \sum_{k=1}^{n} g_{n,k}(x,T^{k}x)
      - \mu\bigl(g_{n,k}(x,\cdot)\bigr), \\
    \hat{\psi}_n(t,y) &\text{ char.\ func.\ of }
      x \mapsto \frac{1}{a_n} \sum_{k=1}^{n} g_{n,k}(y,T^{k}x)
      - \mu\bigl(g_{n,k}(y,\cdot)\bigr).
  \end{align*}

  Let $D > 0$ be a constant to be specified later, and write
  $\beta_n = D \log n$. We shift time by $\beta_n$ and consider
  \[
    \psi_{n,\beta_n}(t) \text{ char.\ func.\ of }
    x \mapsto \frac{1}{a_n} \sum_{k=1+\beta_n}^{n+\beta_n}
    \charfun_{B_k(x)}(T^{k}x) - \mu\bigl( B_k(x) \bigr)
  \]
  and the analogously defined
  $\psi_{n,\beta_n}(t,y), \hat{\psi}_{n,\beta_n}(t)$ and
  $\hat{\psi}_{n,\beta_n}(t,y)$.

  By \eqref{eq:charfun:L1}, we have
  \begin{equation}
  \label{eq:clt:lem:psi:time-shift}
    | \psi_n(t) - \psi_{n,\beta_n}(t) | \leqs_t \frac{\log n}{a_n}.
  \end{equation}
  By \eqref{eq:clt:lemma:g-approx} and \eqref{eq:clt:lemma:g-measure},
  \begin{equation}
  \label{eq:clt:lemma:measure-L1}
    \int_X \biggl| \mu ( B_k(x) )
    - \int_X g_{n,k}(x,y) \diff\mu(y) \biggr| \diff\mu(x)
    \leqs n^{-\frostan\kappa}.
  \end{equation}
  On the other hand, by \eqref{eq:clt:lemma:g-approx},
  \[
    \lnorm{\charfun_{B_k(\cdot)}(T^{k}\cdot)
      - g_{n,k}(\cdot,T^{k}\cdot)}{1}
    \leq \int_X g_{n,k}(x,T^{k}x) - h_{n,k}(x,T^{k}x) \diff\mu(x).
  \]
  By Lemma~\ref{lem:correlation} (with $\alpha = 1$ and $p = \kappa$),
  this implies
  \begin{multline*}
    \biggl| \int_X g_{n,k}(x,T^{k}x) \diff\mu
    - \iint g_{n,k}(x,y) \diff\mu(x)\diff\mu(y) \biggr| \\
    \leqs n^{-\kappa} + n^{4(d_0+3)\kappa} e^{-\mdoc k}
  \end{multline*}
  and the corresponding estimate for $h_{n,k}$. Furthermore, by
  \eqref{eq:clt:lemma:g-measure} and \eqref{eq:clt:lemma:h-measure},
  \[
    \Bigl| \iint g_{n,k} \diff\mu\diff\mu
    - \iint h_{n,k} \diff\mu\diff\mu \Bigr|
    \leqs n^{-\frostan\kappa}
  \]
  Combining these three equations, we have
  \begin{equation}
  \label{eq:clt:lem:correlation}
    \lnorm{\charfun_{B_k(\cdot)}(T^{k}\cdot)
      - g_{n,k}(\cdot,T^{k}\cdot)}{1}
    \leqs n^{-\frostan\kappa} + n^{-\kappa}
      + n^{4(d_0+3)\kappa} e^{-\mdoc k}.
  \end{equation}
  Choosing $\kappa$ such that $n^{-\frostan\kappa} \leqs n^{-1}$ and
  $n^{-\kappa} \leqs n^{-1}$, the estimates
  \eqref{eq:clt:lemma:measure-L1} and \eqref{eq:clt:lem:correlation}
  imply
  \begin{equation}
  \label{eq:clt:lem:psi:1}
  \begin{split}
    | \psi_{n,\beta_n}(t) - \hat{\psi}_{n,\beta_n}(t) |
    &\leqs_t \frac{1}{a_n} \sum_{k=1+\beta_n}^{n+\beta_n}
      \bigl(n^{-1} + n^{4(d_0+3)\kappa} e^{-\mdoc k}\bigr) \\
    &\leqs_t \frac{1}{a_n} \bigl(1 + n^{4(d_0+3)\kappa}
      e^{-\mdoc \beta_n}\bigr) \\
    &\leqs_t \frac{1}{a_n},
  \end{split}
  \end{equation}
  where the last inequality follows from choosing $D$ sufficiently
  large in $\beta_n$.

  By \eqref{eq:clt:lemma:lip}, we have, for all $k \leq n$,
  \[
    \sum_{Q \in \partition_n} \charfun_Q(x) \bigl( g_{n,k}(x,T^{k}x)
      - g_{n,k}(y_Q,T^{k}x) \bigr)
    \leq \sum_{Q \in \partition_n} \charfun_Q(x) n^{\kappa}
      \diam{\partition_n}
    \leqs n^{-\kappa}
  \]
  since $\diam{\partition_n} \leqs n^{-\kappa}$.
  Similarly,
  \[
    \sum_{Q \in \partition_n} \charfun_Q(x)
    \bigl( \mu( g_{n,k}(x,\cdot) ) - \mu ( g_{n,k}(y_Q,\cdot) ) \bigr)
    \leqs n^{-\kappa}.
  \]
  Therefore, using the notation
  \[
    \hat{\psi}_{n,\beta_n}^{\partition_n}
    \text{ char.\ func.\ of }
    x \mapsto a_n^{-1} \sum_{k=1+\beta_n}^{n+\beta_n}
    \sum_{Q \in \partition_n} \charfun_Q(x) \bigl( g(y_Q,T^{k}x)
    - \mu( g(y_Q,\cdot) ) \bigr),
  \]
  we have, by choosing $\kappa$ sufficiently large,
  \begin{equation}
  \label{eq:clt:lem:psi:2}
    | \hat{\psi}_{n,\beta_n}(t)
      - \hat{\psi}_{n,\beta_n}^{\partition_n}(t) |
    \leqs_t \frac{1}{a_n}\sum_{k=1+\beta_n}^{n+\beta_n} (n^{-\kappa}
      + \mu(Y_n) )
    \leqs_t \frac{n^{1-\kappa}}{a_n}
    \leqs_t \frac{1}{a_n}
  \end{equation}
  since $\mu(Y_n) \leqs n^{-2\kappa}$.

  Using that elements of $\partition_n$ are disjoint, we have
  \begin{multline}
  \label{eq:clt:lemma:psi-partition}
    \hat{\psi}_{n,\beta_n}^{\partition_n}(t) \\
    = \sum_{Q \in \partition_n}
    \int_X \charfun_Q(x) \exp\biggl( \frac{it}{a_n}
    \sum_{k=1+\beta_n}^{n+\beta_n} g_{n,k}(y_Q, T^{k}x)
    - \mu( g_{n,k}(y_Q,\cdot) ) \biggr) \diff\mu(x)
  \end{multline}
  Using the densities $\{\rho_{Q,n}\}$ and the fact that
  $\lnorm{\rho_{Q,n} - \mu(Q)^{-1} \charfun_Q}{1} \leqs n^{-2\kappa}$,
  we have, for all $Q \in \partition_n$,
  \begin{multline}
  \label{eq:clt:lemma:psi-partition:rho}
    \biggl| \mu(Q)^{-1} \int_X \charfun_Q(x) \exp\biggl( \frac{it}{a_n}
    \sum_{k=1+\beta_n}^{n+\beta_n} g_{n,k}(y_Q, T^{k}x)
    - \mu( g_{n,k}(y_Q, \cdot) ) \biggr) \diff\mu(x) \\
    - \int_X \rho_{Q,n} \exp\biggl( \frac{it}{a_n}
    \sum_{k=1+\beta_n}^{n+\beta_n} g_{n,k}(y_Q, T^{k}x)
    - \mu( g_{n,k}(y_Q, \cdot) ) \biggr) \diff\mu(x) \biggr| \\
    \leqs n^{-2\kappa}.
  \end{multline}
  Since $\exp(f+g) = \exp(f)\exp(g)$, the exponential in the integrand
  can be written as a product of observables
  $\exp(it a_n^{-1} (g_{n,k}(y_Q, \cdot) - \mu(g_{n,k}(y_Q,\cdot))))$
  composed with iterates of $T$. Hence, we can apply
  Condition~\ref{cond:multiple-doc} with $m=1$, $m'=n$, $N = \beta_n$,
  $k_1 = 0$ and $l_j = j$ ($j=1,\ldots,n)$. That is, we have
  \begin{multline}
  \label{eq:clt:lemma:psi-partition:3}
    \biggl| \int_X \rho_{Q,n} \exp\biggl( \frac{it}{a_n}
      \sum_{k=1+\beta_n}^{n+\beta_n} g_{n,k}(y_Q, T^{k}x)
      - \mu( g_{n,k}(y_Q, \cdot) ) \biggr) \diff\mu(x) \\
    - \mu(\rho_{Q,n})\int_X \exp\biggl( \frac{it}{a_n}
      \sum_{k=1+\beta_n}^{n+\beta_n} g_{n,k}(y_Q, T^{k}x)
      - \mu( g_{n,k}(y_Q, \cdot) ) \biggr) \diff\mu(x) \biggr| \\
    \leqs_t
      (n^{4(d_0+2)\kappa} + n^{1+2(d_0+1)\kappa + \kappa} ) P_r(n)
      e^{-\mdoc \beta_n},
  \end{multline}
  where we have used that
  \[
  \lipnorm{\rho_{Q,n}} \leqs n^{4(d_0+2)\kappa},
  \quad
  \supnorm{\rho_{Q,n}} \leqs n^{2(d_0+1)\kappa},
  \quad
  \holnorm{g_{n,k}}{\alpha} \leqs n^{\kappa}.
  \]
  Since $\mu(\rho_{Q,n}) = 1 + \bigo(n^{-2\kappa})$, the estimates
  \eqref{eq:clt:lemma:psi-partition}--%
  \eqref{eq:clt:lemma:psi-partition:3}
  imply
  \[
    \Bigl| \hat{\psi}_{n,\beta_n}^{\partition_n}(t)
    - \sum_{Q\in\partition_n}\mu(Q)\hat{\psi}_{n,\beta_n}(t,y_Q) \Bigr|
    \leqs_t n^{-2\kappa}
      + n^{4(d_0+2)\kappa} P_r(n) e^{-\mdoc \beta_n}.
  \]
  We can bound the right-hand side by $n^{-1} \leq a_n^{-1}$ by
  choosing $\kappa$ and $D$ in $\beta_n = D \log n$ sufficiently large.
  Therefore,
  \begin{equation}
  \label{eq:clt:lemma:psi:pre-integral}
    \Bigl| \hat{\psi}_{n,\beta_n}^{\partition_n}(t)
    - \sum_{Q\in\partition_n}\mu(Q)\hat{\psi}_{n,\beta_n}(t,y_Q) \Bigr|
    \leqs_t \frac{1}{a_n}.
  \end{equation}

  As before, we may use the Lipschitz continuity of $g_{n,k}$, and
  hence that of $y \mapsto \hat{\psi}_{n,\beta_n}(t,y)$, to obtain,
  for each $Q \in \partition_n$,
  \begin{equation}
  \label{eq:clt:lemma:integral-argument}
    \Bigl| \mu(Q) \hat{\psi}_{n,\beta_n}(t,y_Q)
    - \int_Q \hat{\psi}_{n,\beta_n}(t,y) \diff\mu(y) \Bigr|
    \leqs_t \frac{1}{a_n}.
  \end{equation}
  Using that $Y_n = X \setminus \bigcup_{Q \in \partition_n} Q$ has
  measure $\mu(Y_n) \leq n^{-\kappa}$, we can combine
  \eqref{eq:clt:lemma:psi:pre-integral} and
  \eqref{eq:clt:lemma:integral-argument} to obtain that
  \[
    \Bigl| \hat{\psi}_{n,\beta_n}^{\partition_n}(t)
    - \int_X \hat{\psi}_{n,\beta_n}(t,y) \diff\mu(y) \Bigr|
    \leqs_t \frac{1}{a_n}.
  \]
  Using this estimate together with \eqref{eq:clt:lem:psi:time-shift},
  \eqref{eq:clt:lem:psi:1} and \eqref{eq:clt:lem:psi:2}, we obtain
  \[
    \Bigl| \psi_n(t) - \int_X \psi_{n,\beta_n}(t,y) \diff\mu(y)
    \Bigr| \leqs_t \frac{\log n}{a_n}.
  \]
  A similar argument to the one establishing \eqref{eq:clt:lem:psi:1}
  allows us to shift time back by $\beta_n$ and obtain
  \[
    \Bigl| \psi_n(t) - \int_X \psi_n(t,y) \diff\mu(y)
    \Bigr| \leqs_t \frac{\log n}{a_n},
  \]
  which is exactly \eqref{eq:clt:lem:goal}.
\end{proof}

\subsection{Proof of Theorem~\ref{thm:alt}}
With Lemma~\ref{lem:clt:general-rk} established, we now use it to prove
the ALT for recurrence. Recall that
\[
  \exsig_n^2 = \var\biggl( \sum_{k=1}^{n} \charfun_{\exE_k}
    \biggr)
  \quad \text{and} \quad
  \exs_n^2(y) = \var\biggl( \sum_{k=1}^{n} \charfun_{B(y,r_k)} \circ
    T^{k} \biggr).
\]

\begin{proof}[Proof of Theorem~\ref{thm:alt}]
  For convenience, let
  $\bar{h} = \sqrt{h/\mu(h)}$ and use the following notation:
  \begin{align*}
    \psi_n(t) &\text{ char.\ func.\ of }
    x \mapsto \frac{1}{\exsig_n} \sum_{k=1}^{n}
      \bigl(\charfun_{B(x,r_k)}(T^{k}x) - \mu(B(x,r_k)) \bigr), \\
    \psi_n^{\sigma}(t,y) &\text{ char.\ func.\ of }
    x \mapsto \frac{1}{\exsig_n} \sum_{k=1}^{n}
      \bigl(\charfun_{B(y,r_k)}(T^{k}x) - \mu(B(y,r_k)) \bigr), \\
    \psi_n^{s}(t,y) &\text{ char.\ func.\ of }
    x \mapsto
      \frac{\bar{h}(y)}{\exs_n(y)}
      \sum_{k=1}^{n} \bigl(\charfun_{B(y,r_k)}(T^{k}x) - \mu(B(y,r_k))
      \bigr).
  \end{align*}
  We may apply Lemma~\ref{lem:clt:general-rk} to $\abr_n = r_n$ and
  $a_n = \exsig_n$ by the hypothesis~\ref{item:alt:r} and the fact that
  $\exsig_n^2 \leqs n^2$, trivially. As a result,
  \begin{equation}
  \label{eq:alt:proof:1}
    \Bigl| \psi_n(t) - \int_X \psi_n^{\sigma}(t,y) \diff\mu(y)
    \Bigr| \leqs_t \frac{\log n}{\exsig_n}.
  \end{equation}

  Let
  $S_{n,y} = \sum_{k=1}^{n} \charfun_{B(y,r_k)} \circ T^{k}
  - \mu(B(y,r_k))$.
  Then $\int | S_{n,y} | \diff\mu \leq \exs_n(y)$ by the
  Cauchy--Schwarz inequality. Therefore,
  \begin{equation}
  \label{eq:alt:proof:2}
  \begin{split}
    \int_X \bigl| \psi_n^{\sigma}(t,y) - \psi_n^{s}(t,y) \bigr|
      \diff\mu(y)
    &\leqs_t \int_X \Bigl| \frac{1}{\exsig_n}
      - \frac{\bar{h}(y)}{\exs_n(y)} \Bigr|
      \int_X | S_{n,y}(x) | \diff\mu(x) \diff\mu(y) \\
    &\leq \int_X \Bigl| \frac{1}{\exsig_n}
      - \frac{\bar{h}(y)}{\exs_n(y)} \Bigr| |\exs_n(y)|
      \diff\mu(y) \\
    &= \int_X \Bigl| \frac{\exs_n(y)}{\exsig_n} - \bar{h}(y)
      \Bigr| \diff\mu(y) \\
    &\to 0 \text{ as $n\to \infty$},
  \end{split}
  \end{equation}
  where the last line follows by the hypothesis \ref{item:alt:var}.
  By the hypothesis \ref{item:alt:clt}, the functions 
  $y \mapsto \psi_n^{s}(t,y)$ converge pointwise $\mu$-a.e.\ to
  $y \mapsto \exp( - \bar{h}(y)^2t^2/2)$ as $n \to \infty$.
  Additionally, since the functions $y \mapsto \psi_n^{s}(t,y)$ are
  uniformly bounded by $1$, we have
  \begin{equation}
  \label{eq:alt:proof:3}
    \lim_{n \to \infty}
    \int_X \biggl| \psi_n^{s}(t,y)
    - \exp\biggl( -\frac{\bar{h}^2}{2}t^2 \biggr) \biggr| \diff\mu(y)  
    = 0.
  \end{equation}
  by the Dominated Convergence Theorem.

  By Lemma~\ref{lem:r-explicit:liminf},
  $\exsig_n^2 \geqs n^{1-\seqlow \dim X}$, where
  $\seqlow \in (0,1/\dim X)$, and hence $\exsig_n^{-1} \log n \to 0$.
  Thus, by combining \eqref{eq:alt:proof:1}--\eqref{eq:alt:proof:3}, we
  conclude the proof.
\end{proof}

\subsection{Proof of Theorem~\ref{thm:clt}}
We now prove that the CLT can be recovered if one appropriately
rescales the radii of the balls.
Recall that
\[
  \imsig_n^2 = \var\biggl( \sum_{k=1}^{n} \charfun_{\imE_k} \biggr)
  \quad \text{and} \quad
  \ims_n^2(y) = \var\biggl( \sum_{k=1}^{n} \charfun_{B(y,\imr_n(y))}
  \circ T^{k}\biggr).
\]

We will use the following general result in probability theory.

\begin{proposition}
[see {\cite[Proposition~4.12]{kallenberg-2002-foundations}}]
\label{prop:kallenberg}
  For a fixed $p > 0$, let $f_1, f_2, \ldots \in \Ell{p}(\mu)$ with
  $f_n \tomeas{\mu} f$. Then
  \[
    \lnorm{f_n - f}{p} \to 0
    \quad \text{if and only if} \quad
    \lnorm{f_n}{p} \to \lnorm{f}{p} < \infty.
  \]
\end{proposition}

We now proceed with the proof of Theorem~\ref{thm:clt}.
We will apply Lemma~\ref{lem:clt:general-rk} to the sequence of
implicitly defined radius functions $(\imr_n)$, which are given
by $\mu(B(x,\imr_n(x))) = \imM_n$.

\begin{proof}[Proof of Theorem~\ref{thm:clt}]
  We use the following notation:
  \begin{align*}
    \psi_n(t) &\text{ char.\ func.\ of }
    x \mapsto \frac{1}{\imsig_n} \sum_{k=1}^{n}
      (\charfun_{\imE_k}(x) - \mu(\imE_k)), \\
    \psi_n^{\sigma}(t,y) &\text{ char.\ func.\ of }
    x \mapsto \frac{1}{\imsig_n} \sum_{k=1}^{n}
      (\charfun_{B(y,\imr_k(y))}(T^{k}x) - \imM_k), \\
    \psi_n^{s}(t,y) &\text{ char.\ func.\ of }
    x \mapsto \frac{1}{\ims_n(y)} \sum_{k=1}^{n}
      (\charfun_{B(y,\imr_k(y))}(T^{k}x) - \imM_k),
  \end{align*}
  Recall from the discussion in  \text{\S}\ref{sec:discussion:radius}
  that the radius functions $\imr_n \colon X \to (0,\infty)$ are
  Lipschitz continuous with $\lipconst{\imr_n} = 1$, and $\imr_n \to 0$
  uniformly on $\supp\mu$.
  Furthermore, $\imsig_n^2 \leqs n^2$ trivially.
  We may therefore apply Lemma~\ref{lem:clt:general-rk} with
  $\chi_n = \imr_n$ and $a_n = \imsig_n$.
  As a result,
  \[
    \biggl| \int_X \exp \biggl( \frac{it}{\imsig_n} \sum_{k=1}^{n}
      \charfun_{\imE_k}(x) - \imM_k \biggr) \diff\mu(x) - \int_X
      \psi_n^{\sigma}(t,y) \diff\mu(y) \biggr|
    \leqs_t \frac{\log n}{\imsig_n}.
  \]

  Since $| \mu(\imE_k) - \imM_k | \leqs e^{-\limmeas k}$ by
  Lemma~\ref{lem:r-implicit:meas}, we may use \eqref{eq:charfun:L1} to
  obtain that
  \[
    \biggl| \psi_n(t) - \int_X \exp \biggl( \frac{it}{\imsig_n}
    \sum_{k=1}^{n} \charfun_{\imE_k}(x) - \imM_k \biggr) \diff\mu(x)
    \biggr|
    \leqs_t \frac{\log n}{\imsig_n}.
  \]
  Hence,
  \begin{equation}
  \label{eq:clt-implicit-r:proof:psi:1}
    \Bigl| \psi_n(t) - \int_X \psi_n^{\sigma}(t,y) \diff\mu(y) \Bigr|
    \leqs_t \frac{\log n}{\imsig_n}
  \end{equation}

  We conclude the proof by showing
  \begin{equation}
  \label{eq:clt-implicit-r:proof:psi:2}
    \lim_{n \to \infty} \int_X \bigl| \psi_n^{\sigma}(t,y)
      - \psi_n^{s}(t,y) \bigr| \diff\mu(y)
    = 0 \quad \text{for all $t \in \reals$}
  \end{equation}
  and
  \begin{equation}
  \label{eq:clt-implicit-r:proof:psi:3}
    \lim_{n \to \infty} \int_X \psi_n^{s}(t,y) \diff\mu(y)
    = e^{-t^2/2} \quad \text{for all $t \in \reals$}.
  \end{equation}
  Since $\imsig_n^2 \geqs n^{1-\seqlow}$ and
  $\imsig_n^{-1} \log n \to 0$ by
  Lemma~\ref{lem:r-implicit:liminf:sigma}, the estimates in
  \eqref{eq:clt-implicit-r:proof:psi:1}--%
  \eqref{eq:clt-implicit-r:proof:psi:3} imply the result.

  We first show \eqref{eq:clt-implicit-r:proof:psi:2}. Let
  $S_{n,y} = \sum_{k=1}^{n} \charfun_{B(y,\imr_k(y))} \circ T^{k} -
  \imM_k$.
  Then $\int |S_{n,y}| \diff\mu \leq \ims_n(y)$ by the Cauchy--Schwarz
  inequality. Therefore, by \eqref{eq:charfun:L1},
  \[
  \begin{split}
    \int_X \bigl| \psi_n^{\sigma}(t,y) - \psi_n^{s}(t,y) \bigr|
      \diff\mu(y)
    &\leqs_t \int_X \Bigl| \frac{1}{\imsig_n} - \frac{1}{\ims_n(y)}
      \Bigr| \int_X | S_{n,y}(x) | \diff\mu(x) \diff\mu(y) \\
    &\leq \int_X \Bigl| \frac{1}{\imsig_n} - \frac{1}{\ims_n(y)} \Bigr|
      |\ims_n(y)| \diff\mu(y) \\
    &= \int_X \Bigl| \frac{\ims_n(y)}{\imsig_n} - 1 \Bigr| \diff\mu(y).
  \end{split}
  \]
  The hypothesis~\ref{item:clt:var}, namely that
  $\ims_n^2/\imsig_n^2 \todistr 1$, implies $\ims_n^2/\imsig_n^2 \to 1$
  in measure since the limiting distribution is constant.
  Furthermore, by Proposition~\ref{prop:r-implicit:mean},
  $\lnorm{\ims_n^2/\imsig_n^2}{1} \to 1 < \infty$.
  Thus, by Proposition~\ref{prop:kallenberg},
  $\ims_n^2/\imsig_n^2 \to 1$ in $\Ell{1}$. Now,
  \[
    \Bigl| \frac{\ims_n}{\imsig_n} - 1 \Bigr|
    = \frac{ | \ims_n^2/\imsig_n^2 - 1 | }{|\ims_n/\imsig_n + 1| }
    \leq 
    \Bigl| \frac{\ims_n^2}{\imsig_n^2} - 1 \Bigr|
  \]
  since the denominator is greater than $1$. We therefore conclude that
  \[
    \lim_{n \to \infty} \int_X \Bigl| \frac{\ims_n(y)}{\imsig_n} - 1
    \Bigr| \diff\mu(y) 
    \leq \lim_{n \to \infty} \int_X \Bigl|
      \frac{\ims_n^2(y)}{\imsig_n^2} - 1 \Bigr| \diff\mu(y) 
    = 0.
  \]
  We have therefore shown \eqref{eq:clt-implicit-r:proof:psi:2}.

  To show \eqref{eq:clt-implicit-r:proof:psi:3}, notice that, for
  fixed $t \in \reals$, the function $y \mapsto \psi_n^{s}(t,y)$
  is uniformly bounded by $1$. It converges to $\exp(-t^2/2)$ as
  $n \to \infty$ by hypothesis~\ref{item:clt:clt}. By the Dominated
  Convergence Theorem, we obtain \eqref{eq:clt-implicit-r:proof:psi:3}
  as desired.
\end{proof}

\section{Proof of ASIP for Axiom~A systems}
We devote this section to the proof of the ASIP for Axiom~A systems.
Recall that $T \colon X \to X$ is an Axiom~A diffeomorphism restricted
to a mixing component, and $\mu$ is a Gibbs measure corresponding to a
H{\"o}lder continuous potential. Exponential decay of correlations for
H{\"o}lder continuous observables, i.e., Condition~\ref{cond:doc}, is
known to hold in this setting \cite{bowen-2008-equilibrium}.
We begin with establishing a reduction to the non-invertible case
(Theorem~\ref{thm:asip:non-invertible:sup-growth}) by encoding
$(X,T,\mu)$ as a sub-shift of finite type (SFT) and then reducing to
the one-sided shift space. We then prove the general ASIP
(Theorem~\ref{thm:asip:hyperbolic}), and finally the ASIP for the
shrinking target problem (Theorem~\ref{thm:asip:hyperbolic:targets}).

\subsection{Reduction to an SFT}
There exists a Markov partition of $X$, allowing one to encode
$(X,T,\mu)$ as an (SFT), which we will denote by
$(\sft,\shiftop,\mu^{*})$ (see \cite{bowen-2008-equilibrium}).
We consider the following function space on $\sft$:
For $\theta \in (0,1)$, let $\sftfun_{\theta}$ denote the space of
observables $\varphi \colon \sft \to \reals$ such that there exists
some $C > 0$ for which
$\totalvar_n \varphi < C \theta^{n}$ for all $n \geq 0$, where
\[
  \totalvar_n \varphi
  \coleq \sup \{| \varphi(\underline{x}) - \varphi(\underline{y}) |
  : x_i = y_i \text{ for all } i \in [-n, n]\}.
\]
The space $\sftfun_{\theta}$ is a Banach space under the norm
\[
  \sftnorm{\varphi}{\theta} = \supnorm{\varphi}
  + \sup_{n \geq 1} (\theta^{-n} \totalvar_n \varphi).
\]
Let $\pi \colon \sft \to X$ denote the quotient map corresponding to
the choice of Markov partition.
Then $\pi \circ \shiftop = T \circ \pi$ and $\mu = \pi^{*} \mu^{*}$.
Given a sequence $(\varphi_n)$ of $\alpha$\nobreakdash-H{\"o}lder
continuous observables on $X$, we construct a corresponding sequence
$(\varphi_n^{*})$ of observables on $\sft$ by
$\varphi_n^{*} = \varphi_n \circ \pi$. Clearly
$\supnorm{\varphi_n^{*}} = \supnorm{\varphi_n}$.

The following lemma, due to Hirsch and Pugh~\cite{hirsch-1970-stable}
allows us to transfer the H{\"o}lder continuity of $\varphi_n$ to
$\varphi_n^{*}$. We use result as formulated in
\cite[Lemma~4.1]{bowen-2008-equilibrium}.

\begin{lemma}[\cite{hirsch-1970-stable}]
\label{lem:hirsch}
  There exist $\varepsilon > 0$ and $\theta \in (0,1)$ for which the
  following is true:
  if $x,y \in X$ and $d(T^{k}x,T^{k}y) \leq \varepsilon$ for all
  $k \in [-N,N]$, then $d(x,y) < \theta^{N}$.
\end{lemma}

We may assume that the Markov partition has diameter small than the
$\varepsilon$ appearing in Lemma~\ref{lem:hirsch}.
Thus, Lemma~\ref{lem:hirsch} implies that there exists $C_0 > 0$ and
$\theta \in (0,1)$ such that $(\varphi_n^{*}) \subset \sftfun_{\theta}$
with
\[
  \sftnorm{\varphi_n^{*}}{\theta} \leq C_0 \holnorm{\varphi_n}{\alpha}.
\]
Thus, the norm properties of $(\varphi_n)$ transfer to
$(\varphi_n^{*})$.

We now state and prove a modification of the `Sinai trick'. The usual
Sinai trick (see, e.g., \cite{bowen-2008-equilibrium,parry-1999-zeta})
is used to reduce certain properties~-- such as the existence of Gibbs
measures~-- of two-sided SFTs to those of one-sided SFTs. For a
function $\varphi \colon \sft \to \reals$ on the two-sided shift-space,
it is used to obtain a function $\varphi^{+} \colon \sft \to \reals$
depending \emph{only on future coordinates}, i.e.,
$\varphi^{+}(\underline{x}) = \varphi^{+}(\underline{y})$ whenever
$\underline{x}, \underline{y} \in \sft$ with $x_i = y_i$ for all
$i \geq 0$. One can then view $\varphi^{+}$ as a function on the
one-sided shift-space $\sftpos$.
The modification we make accounts for the non-stationary aspect of our
problem; we consider a sequence of functions $(\varphi_n^{*})$ as
opposed to just a single function $\varphi$.
The non-autonomous version of the Sinai trick used in
\cite[Corollary~6.2]{haydn-2017-almost} assumed that $(\varphi_n^{*})$
is uniformly bounded in $\sftnorm{\cdot}{\theta}$.
In our case, we cannot assume a uniform bound. At the same time, we
keep track of the expectation and variance of the new sequence of
observables, showing that they are comparable to the original sequence.

\begin{lemma}[A non-autonomous Sinai Trick]
\label{lem:sinai-trick}
  Let $\theta \in (0,1)$ and suppose that
  $(\varphi_n^{*}) \subset \sftfun_\theta$ is a sequence of observables
  $\varphi_n^{*} \colon \sft \to \reals$ such that
  \[
    \sup_{n \geq 1} \supnorm{\varphi_n^{*}} < \infty
    \quad \text{ and } \quad
    \holnorm{\varphi_n^{*}}{\theta} \leqs n^{p}
  \]
  for some constant $p > 0$. Denote the variance corresponding to the
  sequence $(\varphi_n^{*})$ by
  $\sftvar_n^2
  = \var( \sum_{k=1}^{n} \varphi_k^{*} \circ \shiftop^{k})$.
  Then there exists $\theta_0 \in (0,1)$ and a sequence
  $(\varphi_n^{+}) \subset \sftfun_{\theta_0}$ of observables
  $\varphi_n^{+} \colon \sft \to \reals$ depending only on future
  coordinates such that
  \begin{gather}
  \label{eq:sinai:norms}
    \supnorm{\varphi_n^{+}} \leqs \log n
    \quad \text{ and } \quad
    \holnorm{\varphi_n^{+}}{\theta_0} \leqs n^{p} \\
  \label{eq:sinai:ae}
    \biggl| \sum_{k=1}^{n} \varphi_k^{+} \circ \shiftop^{k}
      - \sum_{k=1}^{n}
    \varphi_k^{*} \circ \shiftop^{k} \biggr| \leqs \log n
    \quad \text{$\mu^{*}$-a.e.}, \\
  \label{eq:sinai:expectation}
    \biggl| \sum_{k=1}^{n} \mu^{*}(\varphi_k^{+}) - \sum_{k=1}^{n}
    \mu^{*}(\varphi_k^{*}) \biggr| \leqs \log n, \\
  \label{eq:sinai:variance}
    \biggl| \var\biggl( \sum_{k=1}^{n} \varphi_k^{+} \circ \shiftop^{k}
    \biggr) - \sftvar_n^2
    \biggr| \leqs \sftvar_n \log n + (\log n)^{2}.
  \end{gather}
\end{lemma}

Our proof modifies the one of the usual Sinai trick (see
\cite{boshernitzan-1993-quantitative,parry-1999-zeta}), and
extends the modification of \cite[Corollary~6.2]{haydn-2017-almost}.

\begin{proof}
  We begin by constructing the sequence of functions $(\varphi_n^{+})$.
  For each allowable digit $i$ in $\sft$, fix an element
  $\underline{x}^{i} \in \sft$ with $x_0^{i} = i$.
  Define the function $G \colon \sft \to \sft$ by
  $\underline{y} = G(\underline{x})$, where
  \[
    y_k = \begin{cases}
      x_k   &  k \geq 0, \\
      x_k^{i} & k < 0 \text{ and } x_0 = i.
    \end{cases}
  \]
  This function clearly only depends on future coordinates.
  For each $n \geq 1$, define
  \[
    v_n(\underline{x}) = \sum_{k=n}^{\infty} \varphi_k^{*}(
    \shiftop^{k-n} \underline{x}) - \varphi_k^{*}( \shiftop^{k-n} G
    \underline{x}).
  \]
  Notice that the series converges everywhere and
  \begin{equation}
  \label{eq:sinai:vn:supnorm}
    \supnorm{v_n} \leqs \log n.
  \end{equation}
  Indeed, let $\beta_n = D \log n$ for some constant
  $D > - p/\log \theta$. Then, for all $\underline{x} \in \sft$,
  \[
    |v_n(\underline{x})|
    \leq 2 \beta_n \log(n + \beta_n)
    + \sum_{k=n+\beta_n}^{\infty} |\varphi_k^{*}(\shiftop^{k-n}
    \underline{x}) - \varphi_k^{*}(\shiftop^{k-n}G\underline{x})|
  \]
  since $\supnorm{\varphi_k^{*}} \leq \log(n + \beta_n)$ for
  $k \leq n + \beta_n$. By the definition of $\beta_n$, we have
  $\log (n + \beta_n) \leqs \beta_n$. Combined with the fact that
  $(\varphi_k^{*}) \subset \sftfun_\theta$, we have
  \[
  \begin{split}
    |v_n(\underline{x})|
    &\leqs \beta_n + \sum_{k=n+\beta_n}^{\infty}
      \sftnorm{\varphi_k^{*}}{\theta} \theta^{k-n} \\
    &\leqs \beta_n + \sum_{k=n+\beta_n} k^{p} \theta^{k-n} \\
    &\leqs \beta_n + n^{p} \theta^{\beta_n}.
  \end{split}
  \]
  Since $n^{p}\theta^{\beta_n} = n^{p + D \log \theta}\to 0$ by the
  choice of $D$, we obtain \eqref{eq:sinai:vn:supnorm}.

  Now, for each $n \geq 1$ and $\underline{x} \in \sft$, define
  \begin{equation}
  \label{eq:sinai:defn}
  \begin{split}
    \varphi_n^{+}(\underline{x})
    &= \varphi_n^{*}(\underline{x}) - v_n(\underline{x})
      + v_{n+1}(\shiftop \underline{x}) \\
    &= \varphi_n^{*}(G \underline{x}) + \sum_{k=n+1}^{\infty}
      \varphi_k^{*}(\shiftop^{k-n} G \underline{x})
      - \varphi_k^{*}(\shiftop^{k-n-1} G \shiftop \underline{x}).
  \end{split}
  \end{equation}
  Clearly, $\varphi_n^{+}$ depends only on future coordinates, and
  $\supnorm{\varphi_n^{+}} \leqs \log n$ by \eqref{eq:sinai:vn:supnorm}
  and the assumption on $(\varphi_n^{*})$.

  We will now show that $\varphi_n^{+} \in \sftfun_{\theta_0}$ with
  $\sftnorm{\varphi_n^{+}}{\theta_0} \leqs n^{p}$ for
  $\theta_0 = \theta^{1/6}$. 
  It suffices to show that $v_n \in \sftfun_{\theta_0}$ and
  $\sftnorm{v_n}{\theta_0} \leqs n^{p}$.
  In fact, we establish
  \begin{equation}
  \label{eq:totalvar:3N}
    \totalvar_{3N} v_n \leq C_0 n^{p} \theta^{N/2}
    \quad N \geq 0
  \end{equation}
  for some constant $C_0 > 0$. This implies
  \[
    \totalvar_{3N+j} v_n \leq \Bigl(\frac{C_0}{\theta^{1/3}}\Bigr)
    n^{p} (\theta^{1/6})^{3N + j}
    \leqs n^{p} \theta_0^{3N+j}
    \quad N \geq 0, j = 0,1,2.
  \]

  Fix $N \geq 0$ and $n \geq 1$.
  Let $\underline{x}, \underline{y} \in \sft$ such that $x_i = y_i$ for
  all $|i| \leq 3N$. Then
  \begin{align*}
    | \varphi_k^{*}(\shiftop^{k-n}\underline{x})
      - \varphi_k^{*}(\shiftop^{k-n}\underline{y}) |
    &\leqs \sftnorm{\varphi_k^{*}}{\theta} \theta^{3N - (k-n)}
      \quad \text{for all } n \leq k \leq n+N, \\
    | \varphi_k^{*}(\shiftop^{k-n}G\underline{x})
      - \varphi_k^{*}(\shiftop^{k-n}G\underline{y}) |
    &\leqs \sftnorm{\varphi_k^{*}}{\theta} \theta^{3N - (k-n)}
      \quad \text{for all } n \leq k \leq n+N,
  \end{align*}
  and
  \begin{align*}
    | \varphi_k^{*}(\shiftop^{k-n}\underline{x})
      - \varphi_k^{*}(\shiftop^{k-n}G\underline{x}) |
    &\leq \sftnorm{\varphi_k^{*}}{\theta} \theta^{k-n}
      \quad \text{for all } k \geq n, \\
    | \varphi_k^{*}(\shiftop^{k-n}\underline{y})
      - \varphi_k^{*}(\shiftop^{k-n}G\underline{y}) |
    &\leq \sftnorm{\varphi_k^{*}}{\theta} \theta^{k-n}
      \quad \text{for all } k \geq n.
  \end{align*}
  Therefore,
  \[
  \begin{split}
    |v_n(\underline{x}) - v_n(\underline{y})|
    &\leq 2\sum_{k=n}^{n+N} \sftnorm{\varphi_k^{*}}{\theta}
      \theta^{3N-(k-n)} + 2\sum_{k=n+N}^{\infty}
      \sftnorm{\varphi_k^{*}}{\theta} \theta^{k-n} \\
    &\leqs (n+N)^{p} \sum_{k=n}^{n+N} \theta^{3N-(k-n)}
      + \sum_{k=n+N}^{\infty} k^{p} \theta^{k-n} \\
    &\leqs (n+N)^{p} \theta^{2N} + (n+N)^{p} \theta^{N} \\
    &\leqs n^{p} \theta^{N} + N^{p} \theta^{N} \\
    &\leqs n^{p} \theta^{N} + \theta^{N/2} \\
    &\leqs n^{p} \theta^{N/2},
  \end{split}
  \]
  where we kill the polynomial factor $N^{p}$ in the fourth inequality
  with the exponential factor $\theta^{N/2}$, leaving an exponential
  factor $\theta^{N/2}$.
  We have therefore established \eqref{eq:totalvar:3N}, from which it
  follows $\sftnorm{v_n}{\theta_0} \leqs n^{p}$.
  Hence, we have shown \eqref{eq:sinai:norms} to hold for
  $(\varphi_n^{+})$.

  The estimates \eqref{eq:sinai:ae}--\eqref{eq:sinai:variance} now
  follow from simple calculations. From \eqref{eq:sinai:defn}, we see
  \[
    \sum_{k=1}^{n} \varphi_k^{+} \circ \shiftop^{k} - \sum_{k=1}^{n}
    \varphi_k^{*} \circ \shiftop^{k}
    = v_{n+1} \circ \shiftop^{n+1} - v_1 \circ \shiftop.
  \]
  The estimates in \eqref{eq:sinai:ae} and \eqref{eq:sinai:expectation}
  follow from a straightforward application of
  \eqref{eq:sinai:vn:supnorm}.
  Lastly, notice that
  \begin{multline*}
    \biggl( \sum_{k=1}^{n} \tilde{\varphi}_k^{+} \circ \shiftop^{k}
      \biggr)^2
    = \biggl( \sum_{k=1}^{n} \tilde{\varphi}_k^{*} \circ \shiftop^{k}
      \biggr)^2 + 2\biggl( \sum_{k=1}^{n} \tilde{\varphi}_k^{*} \circ
      \shiftop^{k} \biggr) (\tilde{v}_{n+1} \circ \shiftop^{n+1}
      - \tilde{v}_1) \\
    + (\tilde{v}_{n+1} \circ \shiftop^{n+1} - \tilde{v}_1)^2,
  \end{multline*}
  which implies, together with \eqref{eq:sinai:vn:supnorm}, that
  \[
    \biggl| \var \biggl( \sum_{k=1}^{n} \tilde{\varphi}_k^{+} \circ
      \shiftop^{k} \biggr) - \sftvar_n^2 \biggr|
    \leqs \sftvar_n \log n + (\log n)^{2},
  \]
  where we have used H{\"o}lder's inequality to show
  \[
    \int \biggl| \sum_{k=1}^{n} \tilde{\varphi}_k^{*} \circ
    \shiftop^{k} \biggr| \diff\mu^{*} \leq \sftvar_n.
  \]
  We have thus established \eqref{eq:sinai:variance} and we conclude
  the proof.
\end{proof}

\subsection{Proof of Theorem~\ref{thm:asip:hyperbolic}}
The proof is now a simple application of Lemma~\ref{lem:sinai-trick}.

\begin{proof}
  Given the sequence $(\varphi_n)$ of $\alpha$\nobreakdash-H{\"o}lder
  continuous observables on $X$, obtain the corresponding sequence
  $(\varphi_n^{*})$ on $\sft$ such that
  \[
    \sup_{n \geq 1} \supnorm{\varphi_n^{*}} < \infty
    \quad \text{ and } \quad
    \sftnorm{\varphi_n^{*}}{\theta} \leqs \holnorm{\varphi_n}{\alpha}.
  \]
  By Lemma~\ref{lem:sinai-trick}, we obtain the sequence
  $(\varphi_n^{+}) \subset \sftfun_{\theta_0}$ of observables
  depending only on future coordinates. Viewing $(\varphi_n^{+})$ as a
  sequence of observables on $\sftpos$, we look to apply
  Theorem~\ref{thm:asip:non-invertible:sup-growth}.

  Let
  $\sftvarpos_n^2
  = \var( \sum_{k=1}^{n} \varphi_k^{+} \circ \shiftop^{k})$.
  By Lemma~\ref{lem:sinai-trick}, the assumptions on the original
  sequence $(\varphi_n)$ and variance $\asipvar_n^2$ transfer to
  $(\varphi_n^{+})$ and $\sftvarpos_n^2$.
  Since space $(\sftpos, \sigmapos, \mu^{*})$ satisfies
  Condition~\ref{cond:transfer-operator},
  we can therefore conclude that
  $(\tilde{\varphi}_n^{+} \circ \shiftop^{n})$ satisfies \eqref{asip}
  for any $\beta < \min \{\frac{1}{2}, 1 - \frac{1}{2\delta}\}$ with
  $\sum_{k=1}^{n} \EE[Z_k^2] = \sftvarpos_n^2
  + \bigo( \sftvarpos_n (\log \sftvarpos_n)^2 )$.
  By Lemma~\ref{lem:sinai-trick}, this transfers to the sequence
  $(\tilde{\varphi}_n^{*} \circ \shiftop^{n})$, and thus to
  $(\tilde{\varphi}_n \circ T^{n})$.
\end{proof}

\subsection{Proof of Theorem~\ref{thm:asip:hyperbolic:targets}}
We now apply Theorem~\ref{thm:asip:hyperbolic} to the shrinking target
problem. We will use Conditions~\ref{cond:frostman} and
\ref{cond:thin-annuli} to approximate the indicator functions of balls
by H{\"o}lder continuous observables. We first state a lemma which
ensures that the hypothesis~\ref{item:asip:hyperbolic:targets:var} in
Theorem~\ref{thm:asip:hyperbolic:targets} implies a lower growth bound
on the variance.
Recall that $B_n = B(y_n,r_n)$ is a sequence of balls, and
\[
  \asipvarlem_n^2 \coleq \var\biggl( \sum_{k=1}^{n}
    \charfun_{B(y_k,r_k)} \circ T^{k} \biggr)
  \quad \text{and} \quad
  \exmsum_n = \sum_{k=1}^{n} \mu(B(y_k,r_k)).
\]

\begin{lemma}
\label{lem:var:shrinking-target:sigma}
  Suppose $(X,T,\mu)$ is an m.p.s.\ and let $B_n = B(y_n,r_n$) be a
  sequence of balls with centres $y_n \in \supp\mu$.
  Write $\asipmeas_n = \mu(B_n)$.
  Assume that
  \begin{enumerate}[label=(\roman*)]
    \item
      Condition~\ref{cond:doc} holds;
    \item
      $\mu$ satisfies
      Conditions~\ref{cond:frostman} and \ref{cond:thin-annuli};
    \item
      $(\asipmeas_n)$ satisfies Condition~\ref{cond:sequence} for some
      $\seqlow \in (0,1)$ and $\seqhigh > 1$.
  \end{enumerate}
  Then
  \[
    \liminf_{n \to \infty} \frac{\asipvarlem_n^2}{\exmsum_n} \geq 1.
  \]
  In particular, $\asipvarlem_n^2 \geqs n^{1-\seqlow}$.
\end{lemma}

\begin{proof}
  The proof of the lemma is almost identical to that of
  Lemma~\ref{lem:r-implicit:liminf:s} (see
  \text{\S}\ref{sec:proof:r-implicit:liminf}).
\end{proof}

We now prove the ASIP for the shrinking target problem. We approximate
the indicator functions $\charfun_{B_n}$ by H{\"o}lder continuous
functions. We then establish the ASIP for the approximations, and
conclude by transferring the ASIP back to the indicator functions.

\begin{proof}[Proof of Theorem~\ref{thm:asip:hyperbolic:targets}]
  Write $B_n = B(y_n,r_n)$ for a sequence of points $y_n \in \supp \mu$
  and a sequence of radii $r_n \in [0,\infty)$.
  Since the centres $y_n$ are in $\supp\mu$, we have $r_n \to 0$. Hence
  by Lemma~\ref{lem:wlog}, we may assume without loss of generality
  that $r_n \leq \thinanr$ for all $n$.
  Write $\frostan = \min \{\frost,\thinan\}$.

  Let $\kappa > 0$ be a constant to be specified later and write
  $\varepsilon_k = k^{-\kappa}$. Apply Lemma~\ref{lem:construction}
  with $\abr = r_k$ and $\varepsilon = \varepsilon_k$ and obtain
  $\alpha$\nobreakdash-H{\"o}lder continuous functions
  $\varphi_k \colon X \to [0,1]$ such that
  \begin{equation}
  \label{eq:cor:hyp:regularity}
    \charfun_{B_k(y_k, r_k)} \leq \varphi_k
    \leq \charfun_{B_k(y_k, r_k+\varepsilon_k)}
    \quad \text{ and } \quad
    \holnorm{\varphi_k}{\alpha} \leqs k^{-\kappa}.
  \end{equation}
  Thus $(\varphi_n)$ satisfies the regularity conditions of
  Theorem~\ref{thm:asip:hyperbolic}.

  \step[establishing the ASIP for $(\varphi_n)$]
  To conclude the ASIP for $(\tilde{\varphi}_n \circ T^{n})$, we must
  first show
  \begin{equation}
  \label{eq:approximate:variance}
    \asipvarA_n^2
    \coleq \var \biggl( \sum_{k=1}^{n} \varphi_k \circ T^{k} \biggr)
    \geq n^{\delta}
  \end{equation}
  for some $\delta > 1/2$.
  By Lemma~\ref{lem:var:shrinking-target:sigma}, we have
  $\asipvar_n^2 \geqs \asipmsum_n \geq n^{1-\seqlow}$.
  Since we assume $\seqlow < 1/2$, we have
  $\asipvar_n^2 \geqs n^{\delta}$ for $\delta = 1 - \seqlow > 1/2$.
  This transfers to \eqref{eq:approximate:variance} by the
  relationship
  \begin{equation}
  \label{eq:cor:hyp:variance}
    \asipvarA_n^2 = \asipvar_n^2 + \bigo(\log \asipvar_n),
  \end{equation}
  which we will show now.

  Expanding the expressions for the variances, we have
  \begin{align*}
    \asipvar_n^2 &= \sum_{k=1}^{n} \asipmeas_k - \asipmeas_k^2
      + \sum_{k=1}^{n-1} \sum_{j=k+1}^{n}
      \covar( \charfun_{B_k}, \charfun_{B_j} \circ T^{j-k}), \\
    \asipvarA_n^2 &= \sum_{k=1}^{n} \mu(\varphi_k^2)
      - \mu(\varphi_k)^2 + \sum_{k=1}^{n-1} \sum_{j=k+1}^{n} \covar(
      \varphi_k, \varphi_j \circ T^{j-k}).
  \end{align*}
  By \eqref{eq:cor:hyp:regularity},
  \[
    \biggl|\sum_{k=1}^{n} \mu(\varphi_k^2) - \mu(\varphi_k)^2
    - \sum_{k=1}^{n} \asipmeas_k - \asipmeas_k^2 \biggr| \leqs 1
  \]
  Hence,
  \begin{equation}
  \label{eq:cor:hyp:var:1}
    |\asipvarA_n^2 - \asipvar_n^2|
    \leqs 1 + \biggl| \sum_{k=1}^{n-1} \sum_{j=k+1}^{n}
    \covar(\varphi_k, \varphi_j \circ T^{j-k})
    - \covar(\charfun_{B_k}, \charfun_{B_j} \circ T^{j-k}) \biggr|.
  \end{equation}
  We split the covariance estimation up into when $j-k$ is small and
  when it is large. We estimate first when $j-k$ is large.

  Let $D > 0$ be a constant to be specified later, and write
  $\beta_n = D \log n$. Applying Lemma~\ref{lem:construction} with
  $r = r_k$ and $\varepsilon = \varepsilon_n$
  obtain $\alpha$\nobreakdash-H{\"o}lder continuous functions
  $g_{n,k} \colon X \to [0,1]$ such that, for all $k\leq n$,
  \begin{gather*}
    \holnorm{g_{n,k}}{\alpha} \leqs n^{\kappa}, \\
    \charfun_{B(y_k,r_k)} \leq g_{n,k}
      \leq \charfun_{B(y_k,r_k+\varepsilon_n)}, \\
    |\mu(g_{n,k}) - \asipmeas_k| \leqs n^{-\frostan\kappa}.
  \end{gather*}
  As a consequence, we obtain
  $|\mu(g_{n,k}) \mu(g_{n,j}) - \asipmeas_k \asipmeas_j|
  \leqs \varepsilon_n^{-\frostan\kappa}$
  and therefore
  \[
    \Bigl| \int_X g_{n,k} g_{n,j} \circ T^{j-k} \diff\mu - \asipmeas_k
    \asipmeas_j \Bigr|
    \leqs n^{-\frostan\kappa} + n^{2\frostan\kappa} e^{-\doc (j-k)}
  \]
  by Condition~\ref{cond:doc}. Hence,
  \begin{equation*}
  \begin{split}
    \sum_{k=1}^{n-1} \sum_{j=k+\beta_n+1}^{n} \int_X \charfun_{B_k} & 
      \charfun_{B_j} \circ T^{j-k} \diff\mu - \asipmeas_k \asipmeas_j \\
    &\leq \sum_{k=1}^{n-1} \sum_{j=k+\beta_n+1}^{n} \int_X g_{n,k}
      g_{n,j} \circ T^{j-k} \diff\mu - \asipmeas_k \asipmeas_j \\
    &\leqs \sum_{k=1}^{n-1} \sum_{j=k+\beta_n+1}^{n}
      n^{-\frostan\kappa} + n^{2\frostan\kappa} e^{-\doc(j-k)} \\
    &\leqs n^{2-\frostan\kappa}
      + n^{1+2\frostan\kappa} e^{-\doc \beta_n} 
      = \smallo(1),
  \end{split}
  \end{equation*}
  by choosing $\kappa$ sufficiently large, and then $D$ in
  $\beta_n = D \log n$ sufficiently large.

  Thus,
  \[
    \sum_{k=1}^{n-1} \sum_{j=k+\beta_n+1}^{n} \int_X \charfun_{B_k}
    \charfun_{B_j} \circ T^{j-k} \diff\mu - \asipmeas_k \asipmeas_j
    \leq \smallo(1).
  \]
  By instead considering functions bounding $\charfun_{B_k}$ by below,
  we may obtain the reverse inequality and conclude that
  \begin{equation}
  \label{eq:cor:hyp:var:2}
    \biggl| \sum_{k=1}^{n-1} \sum_{j=k+\beta_n+1}^{n} \int_X \varphi_k
    \varphi_j \circ T^{j-k} \diff\mu - \asipmeas_k \asipmeas_j \biggr|
    = \smallo(1).
  \end{equation}

  Estimating the covariance part of $(\varphi_k)$ is more
  straightforward. By Condition~\ref{cond:doc}, we have
  \begin{equation}
  \label{eq:cor:hyp:var:3}
  \begin{split}
    \sum_{k=1}^{n-1} \sum_{j=k+\beta_n+1}^{n} \Bigl| \int_X \varphi_k
      \varphi_j \circ T^{j-k} \diff\mu - \mu(&\varphi_k)\mu(\varphi_j)
      \Bigr| \\
    &\leqs \sum_{k=1}^{n-1} \sum_{j=k+\beta_n+1}^{n}
      k^{\kappa} j^{\kappa} e^{-\tau(j-k)} \\
    &\leqs n^{2\kappa+1} e^{-\doc \beta_n}
    = \smallo(1),
  \end{split}  
  \end{equation}
  again by choosing $D$ in $\beta_n = D \log n$ sufficiently large.

  Combining \eqref{eq:cor:hyp:var:1}--\eqref{eq:cor:hyp:var:3}, we have
  \begin{equation}
  \label{eq:cor:hyp:var:4}
    |\asipvarA_n^2 - \asipvar_n^2| \leqs 1 + \biggl| \sum_{k=1}^{n-1}
    \sum_{j=k+1}^{k+\beta_n}
    \covar(\varphi_k, \varphi_j \circ T^{j-k})
    - \covar(\charfun_{B_k}, \charfun_{B_j} \circ T^{j-k}) \biggr|.
  \end{equation}

  We now treat the case when $j-k$ is small.
  By \eqref{eq:cor:hyp:regularity}, Condition~\ref{cond:thin-annuli}
  and the Cauchy--Schwarz inequality,
  \[
  \begin{split}
    \biggl|\int_X \varphi_k \varphi_j \circ T^{j-k} \diff\mu - \int_X
      & \charfun_{B_k} \charfun_{B_j} \circ T^{j-k} \diff\mu \biggr| \\
    &\leq \lnorm{\varphi_k - \charfun_{B_k}}{2} \lnorm{\varphi_j}{2}
      + \lnorm{\varphi_j - \charfun_{B_j}}{2} \lnorm{\varphi_k}{2} \\
    &\leqs \varepsilon_k^{\thinan/2} \asipmeas_j^{1/2}
      + \varepsilon_j^{\thinan} \asipmeas_k^{1/2} \\
    &\leqs \varepsilon_k^{\thinan/2} = k^{-\thinan\kappa/2},
  \end{split}
  \]
  where we have used $\varepsilon_j \leq \varepsilon_k$ whenever
  $k \leq j$. Therefore,
  \begin{equation}
  \label{eq:cor:hyp:short:1}
    \sum_{k=1}^{n-1} \sum_{j=k+1}^{k+\beta_n} \Bigl| \int_X \varphi_k
      \varphi_j \circ T^{j-k} \diff\mu - \int_X \charfun_{B_k}
      \charfun_{B_j} \circ T^{j-k} \diff\mu \Bigr|
    \leqs \beta_n \sum_{k=1}^{n-1}
      k^{-\thinan\kappa/2}
    \leqs \beta_n
  \end{equation}
  again by choosing $\kappa$ sufficiently large.
  At the same time,
  $|\mu(\varphi_k) \mu(\varphi_j) - \asipmeas_k \asipmeas_j|
  \leqs k^{-\thinan\kappa}$
  and
  \begin{equation}
  \label{eq:cor:hyp:short:2}
    \sum_{k=1}^{n-1} \sum_{j=k+1}^{k+\beta_n}
      | \mu(\varphi_k)\mu(\varphi_j) - \asipmeas_k \asipmeas_j|
    \leqs \sum_{k=1}^{n-1} \sum_{j=k+1}^{k+\beta_n}
      k^{-\thinan\kappa}
    \leqs \beta_n.
  \end{equation}
  Combining \eqref{eq:cor:hyp:short:1} and \eqref{eq:cor:hyp:short:2},
  we have
  \[
    \sum_{k=1}^{n-1} \sum_{j=k+1}^{k+\beta_n} |
    \covar(\varphi_k, \varphi_j \circ T^{j-k})
    - \covar(\charfun_{B_k}, \charfun_{B_j} \circ T^{j-k}) |
    \leqs \beta_n
  \]
  and therefore obtain
  \[
    | \asipvarA_n^2 - \asipvar_n^2 | \leqs \beta_n \leqs \log n
  \]
  by \eqref{eq:cor:hyp:var:4}. Since $\asipvar_n \geqs n^{\delta}$, we
  have $\log \asipvar_n \geqs \log n$ and have therefore established
  \eqref{eq:cor:hyp:variance}.

  \step[establishing the ASIP for $(\charfun_{B_n})$]
  By \eqref{eq:cor:hyp:regularity},
  \[
    0 \leq \sum_{k=1}^{n} \varphi_k \circ T^{k}
      - \charfun_{B_k} \circ T^{k}
    \leq \sum_{k=1}^{n} \charfun_{A_k} \circ T^{k},
  \]
  where $A_k = B(y_k,r_k+\varepsilon_k) \setminus B(y_k,r_k)$. Now
  $\mu(A_k) \leqs \varepsilon_k^{\thinan} \leqs k^{-\thinan\kappa}$
  by Condition~\ref{cond:thin-annuli}. Choosing $\kappa$ sufficiently
  large, we have $\sum_{k=1}^{\infty} \mu(A_k) < \infty$.
  By the Borel--Cantelli lemma, we conclude that
  $\sum_{k=1}^{\infty} \charfun_{A_k} \circ T^{k}$ is finite $\mu$-a.e.
  Furthermore, $\mu(\varphi_k) - \mu(B_k) \leq \mu(A_k)$ is
  summable. Thus,
  \begin{equation}
  \label{eq:approximate:BC}
    \biggl| \sum_{k=1}^{n} \tilde{\varphi}_k \circ T^{k}(x)
    - \sum_{k=1}^{n} \tilde{\charfun}_{B_k} \circ T^{k}(x) \biggr|
    \leqs_x 1 \quad \text{$\mu$-a.e.\ $x \in X$}.
  \end{equation}
  The estimates \eqref{eq:cor:hyp:variance} and
  \eqref{eq:approximate:BC} imply that
  $(\tilde{\charfun}_{B_k} \circ T^{k})$ satisfies \eqref{asip} with
  rate $\beta < \min \{\frac{1}{2}, 1 - \frac{1}{2\delta}\}$ and
  $\sum_{k=1}^{n} \EE[Z_k^2] = \asipvar_n^2 + \bigo( \asipvar_n (\log
  \asipvar_n)^2)$
  if and only if
  $(\tilde{\varphi}_k \circ T^{k})$ satisfies \eqref{asip} with the
  same rate $\beta$ and
  $\sum_{k=1}^{n} \EE[Z_k^2] = \asipvarA_n^2
  + \bigo( \asipvarA_n (\log \asipvarA_n)^2)$.
\end{proof}

\section{Proof of ASIP for non-invertible systems}
We modify the approach laid out in \cite{haydn-2017-almost}, which is
based on the following abstract martingale framework by Cuny and
Merlev\`ede~\cite{cuny-2015-strong}. 
We incorporate their remark \cite[Remark~2.4]{cuny-2015-strong} into
the statement, which weakens the conditions on the sequence of
second-moments in order to allow for some growth.

\begin{theorem}[{\cite[Theorem~2.3]{cuny-2015-strong}}]
  \label{thm:asip:abstract}
  Let $(X_n)$ be a sequence of square-integrable random variables
  adapted to a non-increasing filtration $(\gcal_n)_n$
  (i.e., $X_n$ is $\gcal_n$-measurable for all $n \in \naturals$).
  Assume that $\CE{X_n}{\gcal_{n+1}} = 0$ a.s., that
  $\abstractvar_n^2 \coleq \sum_{k=1}^{n} \EE[X_k^2] \to \infty$ and
  that there exists $s \in [0,1)$ such that
  $\EE[X_n^2] = \bigo( \abstractvar_n^{2s} )$ for all $n$.
  Let $(a_n)_n$ be a
  non-decreasing sequence of positive numbers such that
  $(a_n/\abstractvar_n^2)_n$ is non-increasing and
  $(a_n/\abstractvar_n)_n$ is non-decreasing. Assume that
  \begin{gather}
    \tag{A}\label{eq:condition:A}
    \sum_{k=1}^{n} \CE{X_k^2}{\gcal_{k+1}} - \EE[X_k^2] =
    \smallo(a_n) \quad \text{$\prob$-a.s.,} \\
    \tag{B}\label{eq:condition:B}
    \sum_{n=1}^{\infty} a_n^{-v} \EE[ |X_n|^{2v} ] < \infty \quad
    \text{for some $1 \leq v \leq 2$}.
  \end{gather}
  Then, enlarging our probability space if necessary, it is possible to
  find a sequence $(Z_k)_{k \geq 1}$ of independent centred Gaussian
  variables with $\EE[Z_k^2] = \EE[X_k^2]$ such that
  \[
    \sup_{1 \leq k \leq n} \biggl| \sum_{i=1}^{k} X_i - \sum_{i=1}^{k}
    Z_i \biggr|
    = \smallo\bigl( (a_n( \log(\abstractvar_n^2/a_n)
    + \log \log a_n))^{1/2} \bigr) \quad \text{$\prob$-a.s.}
  \]
\end{theorem}

In order to establish Condition~\ref{eq:condition:A}, we employ the
G\`al--Koksma lemma~-- a standard tool for establishing quantitative
strong law of large numbers. We state the lemma as formulated by
Harman~\cite[Lemma~1.5]{harman-1998-metric}.

\begin{lemma}[G\`al--Koksma]
  \label{lem:galkoksma}
  Let $\mu$ be a probability measure on the space $X$.
  Let $(f_k)$ be a sequence of non-negative $\mu$-measurable functions,
  and let $(g_k)$ and $(\gamma_k)$ be sequences of real numbers such
  that
  \[
    0 \leq g_k \leq \gamma_k \quad k \in \naturals.
  \]
  Write $\Gamma(N) = \sum_{k=1}^{N} \gamma_k$
  and suppose that $\Gamma(N) \to \infty$ as $N \to \infty$. Suppose
  that, for arbitrary integers $1 \leq m \leq n$,
  \[
    \int_{X} \biggl( \sum_{k=m}^{n} (f_k(x) - g_k) \biggr)^2
    \diff\mu(x) \leqs \sum_{k=m}^{n} \gamma_k.
  \]
  Then, for any given $\varepsilon > 0$ and for $\mu$-a.e.\ $x$, we
  have
  \[
    \biggl| \sum_{k=1}^{N}f_k(x) - \sum_{k=1}^{N} g_k \biggr|
    \leqs_{x,\varepsilon} \Gamma(N)^{1/2} \bigl( \log \Gamma(N)
    \bigr)^{3/2 + \varepsilon} + \max_{1 \leq k \leq N} g_k.
  \]
\end{lemma}

\subsection{Proof of Theorem~\ref{thm:asip:non-invertible:sup-growth}}
We now have the tools necessary for proving the main ASIP for
non-invertible systems.
We divide the proof into three steps.
First, we define the relevant quantities and establish some their
properties.
Then, we verify Condition~\ref{eq:condition:A} and conclude by
verifying \eqref{eq:condition:B}.

\begin{proof}[Proof of
Theorem~\ref{thm:asip:non-invertible:sup-growth}]
  By the abstract framework provided by
  Theorem~\ref{thm:asip:abstract}, we need only verify that
  Conditions~\ref{eq:condition:A} and \ref{eq:condition:B} hold for an
  appropriate sequence $(X_n)$ of reverse martingale differences.

  \step[Set-up and properties of $(X_n)$]
  \label{step:asip:non-invertible:1}
  For each $n \geq 1$, define
  $h_n = \sum_{k=1}^{n} P^{k} \varphi_{n-k}$ and
  \[
    \psi_n = \varphi_n + h_n - h_{n+1} \circ T
    = \tilde{\varphi}_n + \tilde{h}_n - \tilde{h}_{n+1} \circ T.
  \]
  Then
  $\CE{\psi_n \circ T^{n}}{T^{-n-1}\borel} = \CE{\psi_n}{T^{-1} \borel}
  \circ T^{n} = (P\psi_n) \circ T^{n} = 0$ $\mu$-a.e.\, which is the
  crucial property that allows us the define a reverse martingale
  difference.

  We will show that
  \begin{align}
  \label{eq:hn:log-growth}
    \supnorm{h_n} &\leqs (\log n)^2, \\
  \label{eq:hn:poly-growth}
    \holnorm{h_n}{\alpha} &\leqs n^{p},
  \end{align}
  but momentarily delay this in order to establish some basic
  properties of the sequence $(\psi_n \circ T^{n})$.
  Let $X_n = \psi_n \circ T^{n}$ and $\borel_n = T^{-n} \borel$ so that
  $(X_n)$ is a sequence of reverse martingale differences with respect
  to the filtration $(\borel_n)$.
  Define $\abstractvar_n^2 = \sum_{k=1}^{n} \EE[X_k^2]$ and notice that
  \[
    \int_X X_k X_j \diff\mu = \int_X \psi_k \psi_j \circ T^{j-k}
    \diff\mu = \int_X (P\psi_k) \psi_j \circ T^{j-k-1} \diff\mu = 0
  \]
  for $j > k$. Hence,
  \begin{equation}
  \label{eq:asip:X:var}
    \EE \biggl( \sum_{k=1}^{n} X_k \biggr)^2
    = \sum_{k=1}^{n} \EE[X_k^2] = \abstractvar_n^2.
  \end{equation}
  Since it is also the case that
  \begin{multline*}
    \biggl( \sum_{k=1}^{n} X_k \biggr)^2
    = \biggl( \sum_{k=1}^{n} \varphi_k \circ T^{k} \biggr)^2
      + 2\biggl( \sum_{k=1}^{n} \varphi_k \circ T^{k} \biggr)
      ( h_1 \circ T - h_{n+1} \circ T^{n+1}) \\
    + ( h_1 \circ T - h_{n+1} \circ T^{n+1})^2,
  \end{multline*}
  we obtain by \eqref{eq:hn:log-growth} that
  \[
    | \abstractvar_n^2 - \asipvar_n^2 |
    \leqs \asipvar_n (\log n)^2 + (\log n)^4.
  \]
  Since $\asipvar_n^2 \geqs \asipmsum_n \geqs n^{\delta}$,
  we have $\asipvar_n \to \infty$ as $n \to \infty$ and
  $\log n \leqs \log \asipvar_n$, which implies that
  \begin{equation}
  \label{eq:sigmahat:sigma}
    | \abstractvar_n^2 - \asipvar_n^2 |
    \leqs \asipvar_n (\log \asipvar_n)^2.
  \end{equation}
  Therefore $\abstractvar_n^2 \to \infty$ as $n \to \infty$.
  Furthermore, by \eqref{eq:hn:log-growth}, we have
  \begin{equation}
  \label{eq:Xn:variance}
    \EE[X_n^2] \leqs (\log n)^4 \leqs (\log \abstractvar_n)^4
    \leqs \abstractvar_n.
  \end{equation}

  We can now define the sequence $(a_n)$. Recall that $\delta > 1/2$,
  and let $\varepsilon < \min\{1,2 - 1/\delta\}$.
  The reason for this choice of $\varepsilon$ will become apparent when
  we establish Condition~\ref{eq:condition:B}.
  Define $a_n = \abstractvar_n^{2-\varepsilon}$.
  Then $a_n/\abstractvar_n^2 = \abstractvar_n^{-\varepsilon}$ is
  non-increasing and
  $a_n/\abstractvar_n = \abstractvar_n^{1-\varepsilon}$ is
  non-decreasing.

  It will also be useful to obtain some norm estimates for
  $\psi_n$.
  It follows directly from \eqref{eq:hn:log-growth} that
  \begin{equation}
  \label{eq:psin:log-growth}
    \supnorm{\psi_n} \leqs (\log n)^2.
  \end{equation}
  We also obtain from \eqref{eq:hn:poly-growth} that, up to replacing
  $P(\psi_n^2)$ by a $\mu$-a.e.\ identical observable,
  \begin{equation}
  \label{eq:psin:poly-growth}
    \holnorm{P(\psi_n^2)}{\alpha} \leqs n^{p} (\log n)^2.
  \end{equation}
  To see the second inequality, notice that the right-hand side of
  \[
    P(\psi_n^2) = P\bigl((\varphi_n + h_n)^2\bigr) - 2h_{n+1}
    P(\varphi_n + h_n) + h_{n+1}^2
    \quad \text{$\mu$-a.e.}
  \]
  consists entirely of H{\"o}lder continuous observables.
  Using that $\holnorm{\varphi_n}{\alpha} \leqs n^{p}$ together with
  \eqref{eq:hn:log-growth} and \eqref{eq:hn:poly-growth}, we obtain the
  desired estimate.
  We will apply \eqref{eq:psin:poly-growth} only after integrating, and
  hence the substitution of $P(\varphi_n^2)$ with a H{\"o}lder
  continuous representative is justified.

  With these properties established, it remains to prove
  \eqref{eq:hn:log-growth} and \eqref{eq:hn:poly-growth}.
  Let $D$ be a constant to be specified later, and write
  $\beta_n = D \log n$.
  We have
  \begin{equation}
  \label{eq:asip:1:h:sup}
    \supnorm{h_n}
    \leq \sum_{k=1}^{\beta_n} \supnorm{\varphi_{n-k}}
    + \sum_{k=\beta_n+1}^{n} \holnorm{P^{k} \varphi_{n-k}}{\alpha},
  \end{equation}
  where we have used that $\supnorm{Pf} \leq \supnorm{f}$ and
  $\supnorm{f} \leq \holnorm{f}{\alpha}$.
  Since $\supnorm{\varphi_n} \leqs \log n$ and
  $\holnorm{P^{k} \varphi_{n-k}}{\alpha}
  \leqs \holnorm{\varphi_{n-k}}{\alpha} e^{-\spgap k} \leqs (n-k)^{p}
  e^{-\spgap k}$
  by Condition~\ref{cond:transfer-operator}, we have
  \[
    \supnorm{h_n} \leqs \beta_n \log n
      + \sum_{k=\beta_n+1}^{n} (n-k)^{p} e^{-\spgap k}
    \leqs (\log n)^2 + n^{p+1} e^{-\spgap \beta_n}.
  \]
  Choosing $D$ in $\beta_n = D \log n$ sufficiently large, we have
  $n^{p + 1} e^{-\spgap \beta_n} = \smallo(1)$.
  This establishes \eqref{eq:hn:log-growth}.

  For \eqref{eq:hn:poly-growth}, notice that
  \[
    \holnorm{h_n}{\alpha}
    \leq \sum_{k=1}^{n} \holnorm{P^{k} \varphi_{n-k}}{\alpha}
    \leqs \sum_{k=1}^{n} \holnorm{\varphi_{n-k}}{\alpha} e^{-\spgap k}
    \leqs \sum_{k=1}^{n} (n-k)^{p} e^{-\spgap k}
    \leqs n^{p}.
  \]

  \step[Verifying Condition~\ref{eq:condition:A}]
  \label{step:asip:non-invertible:2}
  We now establish Condition~\ref{eq:condition:A} by using
  Lemma~\ref{lem:galkoksma}. As such, it suffices to show that
  \begin{equation}
    \label{eq:conditionA:galkoksma}
    \int_X \biggl( \sum_{k=m}^{n} \CE{X_k^2}{\borel_{k+1}} - \EE[X_k^2]
    \biggr)^2 \diff\mu \leqs \sum_{k=m}^{n} \gamma_k
    \quad \text{for all $m \leq n$},
  \end{equation}
  where $\gamma_k \geq \EE[X_k^2]$, and
  \begin{equation}
  \label{eq:Gamma}
    \Gamma(N) = \sum_{k=1}^{N} \gamma_k
    \leqs \abstractvar_N^{2} (\log \abstractvar_N)^{5}.
  \end{equation}
  Indeed, the conclusion of Lemma~\ref{lem:galkoksma} then implies
  that, for all $\varepsilon > 0$,
  \[
  \begin{split}
    \biggl| \sum_{k=1}^{N} \CE{X_k^2}{\borel_{k+1}} - \EE[X_k^2]
      \biggr|
    &\leqs_\varepsilon \Gamma(N)^{1/2}
      (\log \Gamma(N) )^{3/2+\varepsilon} + \max_{1 \leq k \leq N}
      \EE[X_k^2] \\
    &\leqs_\varepsilon \abstractvar_N
      (\log \abstractvar_N)^{3 + \varepsilon} \\
    &= \smallo ( a_n )
  \end{split}
  \]
  since $\max_{1 \leq k\leq N} \EE[X_k^2] \leqs \abstractvar_N$ by
  \eqref{eq:Xn:variance}, and $a_n = \abstractvar_n^{2-\varepsilon}$.

  To establish \eqref{eq:conditionA:galkoksma}, we reuse $D$ and
  $\beta_n = D \log n$ from Step~\ref{step:asip:non-invertible:1} and
  write the left-hand side of \eqref{eq:conditionA:galkoksma} as
  $A_{m,n} + 2B_{m,n}^{(1)} + 2B_{m,n}^{(2)}$, where
  \begin{align*}
    A_{m,n} &= \sum_{k=m}^{n} \int_X (\CE{X_k^2}{\borel_{k+1}}
      - \EE[ X_k^2])^2 \diff\mu, \\
    B_{m,n}^{(1)} &= \sum_{k=m}^{n-1} \sum_{j=k+1}^{k+\beta_k}
      \int_X \CE{X_k^2}{\borel_{k+1}} \CE{X_j^2}{\borel_{j+1}}
      - \EE[ X_k^2] \EE[ X_j^2] \diff\mu, \\
    B_{m,n}^{(2)} &= \sum_{k=m}^{n-1} \sum_{j=k+\beta_k+1}^{n}
      \int_X \CE{X_k^2}{\borel_{k+1}} \CE{X_j^2}{\borel_{j+1}}
      - \EE[ X_k^2] \EE[ X_j^2] \diff\mu.
  \end{align*}
  We proceed by estimating each term separately.

  It is straightforward to estimate $A_{m,n}$, namely
  \begin{equation}
    \label{eq:galkoksma:A}
    \begin{split}
      |A_{m,n}|
      &\leq \sum_{k=m}^{n} \int_X \CE{X_k^2}{\borel_{k+1}}^2 \diff\mu
        + \EE[X_k^2]^2 \\
      &\leq \sum_{k=m}^{n} \supnorm{X_k^2} \EE[X_k^2] \\
      &\leqs \sum_{k=m}^{n} (\log k)^4 \EE[X_k^2],
    \end{split}
  \end{equation}
  where we have used $\supnorm{X_k^2} \leq C (\log k)^4$, which follows
  from \eqref{eq:psin:log-growth}.

  For $B_{m,n}^{(1)}$, notice that
  \[
    B_{m,n}^{(1)} = \sum_{k=m}^{n-1} \sum_{j=k+1}^{k+\beta_k}
      \int_X \CE{X_k^2}{\borel_{k+1}} \bigl(\CE{X_j^2}{\borel_{j+1}}
      - \EE[X_j^2]\bigr) \diff\mu.
  \]
  Hence,
  \begin{equation}
  \label{eq:B1:1}
  \begin{split}
    |B_{m,n}^{(1)}| &\leq \sum_{k=m}^{n-1} \sum_{j=k+1}^{k+\beta_k}
      \int_X \CE{X_k^2}{\borel_{k+1}} \bigl|\CE{X_j^2}{\borel_{j+1}}
      - \EE[X_j^2]\bigr| \diff\mu \\
    &\leq \sum_{k=m}^{n-1} \sum_{j=k+1}^{k+\beta_k}
      2 \supnorm{X_j^2} \int_X \CE{X_k^2}{\borel_{k+1}} \diff\mu.
  \end{split}
  \end{equation}
  Now $k + \beta_k \leq (1+D) k$ for all $k$. In combination with
  \eqref{eq:psin:log-growth}, we obtain that
  $\supnorm{X_j^2} \leqs ( \log k )^4$ for all $j \leq k + \beta_k$.
  Hence, from \eqref{eq:B1:1} we obtain that
  \begin{equation}
    \label{eq:galkoksma:B1}
    \begin{split}
      |B_{m,n}^{(1)}| &\leqs \sum_{k=m}^{n-1}
        \sum_{j=k+1}^{k+\beta_k} (\log k)^4 \EE[X_k^2] \\
      &\leqs \sum_{k=m}^{n} (\log k)^{5} \EE[X_k^2],
    \end{split}
  \end{equation}

  For $B_{m,n}^{(2)}$, we have, using the fact that
  $\CE{X_j^2}{\borel_{j+1}}$ is $\borel_{k+1}$-measurable for
  $j \geq k$, that
  \begin{equation*}
  \begin{split}
    B_{m,n}^{(2)}
    &= \sum_{k=m}^{n-1} \sum_{j=k+\beta_k+1}^{n} \int_X
      \CE{X_k^2}{\borel_{k+1}} \CE{X_j^2}{\borel_{j+1}} \diff\mu
      - \EE[X_k^2] \EE[X_j^2] \\
    &= \sum_{k=m}^{n-1}
      \sum_{j=k+\beta_k+1}^{n} \int_X X_k^2 \CE{X_j^2}{\borel_{j+1}}
      \diff\mu - \EE[X_k^2] \EE[X_j^2] \\
    &= \sum_{k=m}^{n-1} \sum_{j=k+\beta_k+1}^{n}
      \int_X \psi_k^2 \CE{\psi_j^2}{\borel_1} \circ T^{j-k} \diff\mu
      - \EE[\psi_k^2] \EE[\psi_j^2] \\
    &= \sum_{k=m}^{n-1} \sum_{j=k+\beta_k+1}^{n}
      \int_X P(\psi_k^2) P(\psi_j^2) \circ T^{j-k} \diff\mu
      - \EE[\psi_k^2] \EE[\psi_j^2],
  \end{split}
  \end{equation*}
  where we have used that
  $\CE{f}{\borel_1} = \CE{f}{T^{-1}\borel} = (Pf) \circ T$ $\mu$-a.e.
  Using decay of correlations (see \eqref{eq:doc:transfer-operator}),
  we obtain that
  \[
    \biggl| \int_X P(\psi_k^2) P(\psi_j^2) \circ T^{j-k}
    \diff\mu - \EE[\psi_k^2] \EE[\psi_j^2] \biggr|
    \leqs \holnorm{P(\psi_k^2)}{\alpha} \lnorm{\psi_j^2}{1}
    e^{-\spgap(j-k)}.
  \]
  By \eqref{eq:psin:log-growth} and \eqref{eq:psin:poly-growth}, we
  therefore have
  \begin{equation}
  \label{eq:galkoksma:B2}
  \begin{split}
    B_{m,n}^{(2)}
    &\leqs \sum_{k=m}^{n-1} \sum_{j=k+\beta_k+1}^{n}
      \holnorm{P(\psi_k^2)}{\alpha} \holnorm{P(\psi_j^2)}{\alpha}
      e^{-\spgap(j-k)} \\
    &\leqs \sum_{k=m}^{n-1} \sum_{j=k+\beta_k+1}^{n} k^{p} (\log k)^2
      j^{p} (\log j)^2 e^{-\spgap(j-k)} \\
    &\leqs \sum_{k=m}^{n-1} k^{2p+1} e^{-\spgap \beta_k}.
  \end{split}
  \end{equation}
  Thus, by combining \eqref{eq:galkoksma:A}, \eqref{eq:galkoksma:B1}
  and \eqref{eq:galkoksma:B2}, we see that
  \[
    A_{m,n} + 2B_{m,n}^{(1)} + 2B_{m,n}^{(2)}
    \leqs \sum_{k=m}^{n} \gamma_k,
  \]
  where we have let
  $\gamma_k = (\log k)^{5} \EE[X_k^2] + k^{2p+1} e^{-\spgap \beta_k}$.
  Hence, \eqref{eq:conditionA:galkoksma} is satisfied.

  To establish \eqref{eq:Gamma}, notice that
  $\sum_{k=1}^{N} \EE[X_k^2] \to \infty$, and that
  $k^{2p+1} e^{-\spgap \beta_k}$ is summable by choosing $D$ in
  $\beta_k = D \log k$ sufficiently large. Hence,
  \[
    \Gamma(N) = \sum_{k=1}^{N} \gamma_k
    \leqs (\log N)^{5} \sum_{k=1}^{N} \EE[X_k^2]
    \leqs\abstractvar_N^2 (\log \abstractvar_N)^{5},
  \]
  where we have used that $\log N \leqs \log \abstractvar_N$.

  \step[Verifying Condition~\ref{eq:condition:B}]
  We now establish Condition~\ref{eq:condition:B} for $\nu = 2$.
  By the choice of $a_n$, we have
  $a_n^2 = \abstractvar_n^{2(1-\varepsilon)}
  \geqs n^{(2-\varepsilon)\delta}$, where we recall that
  $\abstractvar_n^2 \sim \asipvar_n^2$ by \eqref{eq:sigmahat:sigma}
  and $\asipvar_n^2 \geqs n^{\delta}$ by assumption. Thus,
  \[
    \sum_{n=1}^{\infty} a_n^{-2}\EE[X_n^{4}]
    \leqs \sum_{n=1}^{\infty}
    \frac{(\log n)^{8}}{n^{(2-\varepsilon)\delta}}
  \]
  converges as $(2 - \varepsilon) \delta > 1$.

  \conclusion
  We have now verified all the conditions of
  Theorem~\ref{thm:asip:abstract}. Therefore, there exists a sequence
  of independent centred Gaussian variables $(Z_k)$ such that
  $\EE[Z_k^2] = \EE[X_k^2]$ and
  \[
    \sup_{1 \leq k \leq n} \biggl| \sum_{i=1}^{k} X_i - \sum_{i=1}^{k}
    Z_i \biggr|
    = \smallo\bigl( (a_n( \log(\asipvar_n^2/a_n) + \log \log a_n))^{1/2}
    \bigr) \quad \text{$\prob$-a.s.}
  \]
  We have
  \[
    \sum_{i=1}^{k} X_i
    = \sum_{i=1}^{k} \tilde{\varphi}_i \circ T^{i}
    + \bigo((\log k)^2)
    = \sum_{i=1}^{k} \tilde{\varphi}_i \circ T^{i}
    + \bigo((\log \asipvar_k)^2)
  \]
  by the choice of $X_i$ and \eqref{eq:hn:log-growth}.
  Furthermore,
  \[
    \sum_{k=1}^{n} \EE[Z_k^2]
    = \asipvar_n^2 + \bigo(\asipvar_n (\log \asipvar_n)^2).
  \]
  by \eqref{eq:asip:X:var} and \eqref{eq:sigmahat:sigma}.
  Finally, as $a_n = \abstractvar_n^{2-\varepsilon}$,
  \[
    a_n\bigl( \log ( \abstractvar_n^2/a_n ) + \log \log a_n\bigr)^{1/2}
    \leqs \abstractvar_n^{1-\varepsilon/2} \log \abstractvar_n.
  \]
  By the choice of $\varepsilon$, the sequence
  $(\tilde{\varphi}_n \circ T^{n})$ therefore satisfies \eqref{asip}
  for any $\beta < \min \{\frac{1}{2}, 1 - \frac{1}{2\delta}\}$.
\end{proof}

\subsection{Proof of Theorem~\ref{thm:asip:non-invertible:targets}}
We now prove the secondary ASIP for non-invertible systems.
As in the proof of Theorem~\ref{thm:asip:hyperbolic:targets}, we
approximate the indicator functions $\charfun_{B_n}$ by H{\"o}lder
continuous functions $\varphi_n$. We then verify the conditions of
Theorem~\ref{thm:asip:abstract}, most of which has been done already in
Steps~\ref{step:asip:non-invertible:1} and
\ref{step:asip:non-invertible:2} of the proof of
Theorem~\ref{thm:asip:non-invertible:sup-growth}.
However, we approach Condition~\ref{eq:condition:B} more carefully as
is done in the proof of \cite[Theorem~5.1]{haydn-2017-almost}.
We include this part for completeness.
Recall that $\asipmeas_n = \mu(\varphi_n)$ and
$\asipmsum_n = \sum_{k=1}^{n} \asipmeas_k$.

\begin{proof}[Proof of Theorem~\ref{thm:asip:non-invertible:targets}]
  Notice first that
  $\asipvar_n^2 \geqs \asipmsum_n \geqs n^{1-\seqlow}$
  by Lemma~\ref{lem:var:shrinking-target:sigma}.
  As in the proof of Theorem~\ref{thm:asip:hyperbolic:targets}, we
  obtain H{\"o}lder continuous approximations
  $\varphi_k \colon X \to [0,1]$ such that
  \begin{gather}
  \nonumber
    \holconst{\varphi_k}{\alpha} \leqs k^{p}, \\
  \label{eq:asip:non-inv:targets:var}
    \asipvarA_n^2 = \asipvar_n^2 + \bigo(\log\asipvar_n), \\
  \label{eq:asip:non-inv:targets:sum}
    \biggl| \sum_{k=1}^{n} \tilde{\varphi}_k \circ T^{k}(x)
      - \sum_{k=1}^{n} \tilde{\charfun}_{B_k} \circ T^{k}(x) \biggr|
      \leqs_x 1 \quad \text{$\mu$-a.e.\ $x \in X$}
  \end{gather}
  (see \eqref{eq:cor:hyp:regularity}--\eqref{eq:approximate:BC}),
  where $p > 0$ and
  $\asipvarA_n^2 = \var( \sum_{k=1}^{n} \varphi_k \circ T^{k})$. Thus,
  we only need to verify the conditions in
  Theorem~\ref{thm:asip:abstract} for the sequence $(\varphi_k)$.

  Define $(h_n)$, $(\psi_n)$ and $(X_n)$ as in the proof of
  Theorem~\ref{thm:asip:non-invertible:sup-growth}, and
  let $\abstractvar_n^2 = \sum_{k=1}^{n} \EE[X_k^2]$.
  Since $\sup_{n \geq 1} \supnorm{\varphi_n} < \infty$, we now have
  \begin{align}
  \label{eq:hn:log-growth:2}
    \supnorm{h_n} &\leqs \log n, \\
  \nonumber
    \holnorm{h_n}{\alpha} &\leqs n^{p}.
  \end{align}
  Comparing \eqref{eq:hn:log-growth} to \eqref{eq:hn:log-growth:2},
  we see that there is a saving of a factor of $\log n$ throughout
  compared to proof of
  Theorem~\ref{thm:asip:non-invertible:sup-growth}.
  Let $\varepsilon \in (0,1)$, and define
  $a_n = \abstractvar_n^{2-\varepsilon}$. Similarly to
  \eqref{eq:sigmahat:sigma}, we have
  \begin{equation}
  \label{eq:test}
    \abstractvar_n^2 = \asipvarA_n^2
    + \bigo( \asipvarA_n \log \asipvarA_n),
  \end{equation}
  where we again have a saving of a factor of $\log n$.
  All the hypotheses apart from
  Condition~\ref{eq:condition:B} are established in the proof of
  Theorem~\ref{thm:asip:non-invertible:sup-growth}.

  We now verify Condition~\ref{eq:condition:B} for $\nu = 2$.
  Let $D > 0$ be a constant that is to be chosen sufficiently large.
  Recall that
  $X_n = \varphi_n \circ T^{n} - h_n \circ T^{n} + h_{n+1} \circ
  T^{n+1}$.
  By \eqref{eq:hn:log-growth:2},
  \begin{multline}
  \label{eq:estimate:Un:L4}
    \int_X X_n^{4} \diff\mu
    = \int_X (\tilde{\varphi}_n + \tilde{h}_n - \tilde{h}_{n+1} \circ
    T)^{4} \diff\mu \\
    \leqs (\log n)^{3} \lnorm{\tilde{\varphi}_n}{1}
    + \lnorm{\tilde{h}_n - \tilde{h}_{n+1} \circ T}{4}^{4}.
  \end{multline}
  As established in the proof of
  Theorem~\ref{thm:asip:non-invertible:sup-growth}
  (see \eqref{eq:asip:1:h:sup}),
  \[
    | \tilde{h}_n | \leqs \sum_{k=1}^{\beta_n}
    |P^{k} \tilde{\varphi}_{n-k}| + n^{p} e^{-\spgap \beta_n}.
  \]
  Selecting $D$ sufficiently large, we have
  $n^{p} e^{-\spgap \beta_n} = n^{p - \spgap D} \leq n^{-2}$.
  Therefore,
  \begin{equation}
  \label{eq:hn:l4}
    \lnorm{\tilde{h}_n}{4} \leqs \sum_{k=1}^{\beta_n}
    \lnorm{\tilde{\varphi}_{n-k}}{4} + n^{-2}.
  \end{equation}
  Now, for $1 \leq k \leq \beta_n$ we have
  \[
    \lnorm{\tilde{\varphi}_{n-k}}{4}
    \leqs \asipmeas_{n-k}^{1/4} \bigl(\log(n-k)\bigr)^{3/4}
    \leqs \asipmeas_{n-\beta_n}^{1/4} (\log n)^{3/4}
  \]
  by the monotonicity of $(\asipmeas_n)$.
  Using \eqref{eq:hn:l4}, we have
  \[
    \lnorm{\tilde{h}_n}{4}
    \leqs \beta_n \asipmeas_{n-\beta_n}^{1/4} (\log n)^{3/4}
    \leqs \asipmeas_{n-\beta_n}^{1/4} (\log n)^2.
  \]
  Therefore,
  \[
    \int_X X_n^{4} \diff\mu \leqs \asipmeas_{n-\beta_n} (\log n)^{8}
  \]
  by \eqref{eq:estimate:Un:L4}.

  We also have
  $a_n^2 = \abstractvar_n^{4-2\varepsilon}
  \geqs \abstractvar_n^{4-2\varepsilon}$
  by \eqref{eq:test}.
  Since $\asipvar \geqs \asipmsum_n$, we have
  $\abstractvar_n^2 \geqs \asipmsum_n$ by
  \eqref{eq:asip:non-inv:targets:var} and \eqref{eq:test}. Thus,
  \begin{equation}
  \label{eq:h2}
    \sum_{k=1}^{n} a_k^{2} \mu(X_k^{4})
    \leqs \sum_{k=1}^{n}\frac{\asipmeas_{k-\beta_k} (\log k)^{8}}{
      \asipmsum_k^{2-\varepsilon}}
    \leqs \sum_{k=1}^{n}\frac{\asipmeas_{k-\beta_k} (\log k)^{8}}{
      \asipmsum_{k-\beta_k}^{2-\varepsilon}}.
  \end{equation}
  Now
  \[
    \frac{\asipmsum_n^{2-\varepsilon}}{\asipmeas_n}
    \geq \biggl( \sum_{k=1}^{n} \asipmeas_k^{1-\frac{1}{2-\varepsilon}}
    \biggr)^{2-\varepsilon}
    \geqs \biggl( \sum_{k=1}^{n} k^{-\seqlow(1-\frac{1}{2-\varepsilon})}
    \biggr)^{2-\varepsilon}
    \geqs n^{2-\varepsilon - \seqlow(1-\varepsilon)}
    = n^{1+\varepsilon'},
  \]
  where $\varepsilon' = (1-\varepsilon)(1-\seqlow) > 0$ since
  $\varepsilon,\seqlow < 1$. Therefore, by \eqref{eq:h2},
  \[
    \sum_{k=1}^{n} a_k^{2} \mu(X_k^{4})
    \leqs \sum_{k=1}^{n}
    \frac{(\log k)^{4}}{(k-\beta_k)^{1+\varepsilon'}}
    \leqs \sum_{k=1}^{n}
    \frac{(\log k)^{4}}{k^{1+\varepsilon'}},
  \]
  which converges as $n \to \infty$ as required by
  Condition~\ref{eq:condition:B}.

  In conclusion, there exists a sequence of independent centred
  Gaussian variables $(Z_k)$ (on an enlarged space $(\Omega,\prob)$ if
  necessary) such that $\EE[Z_k^2] = \EE[X_k^2]$ and
  \begin{equation}
  \label{eq:asip:non-inv:targets:conclusion:1}
    \sup_{1 \leq k \leq n} \biggl| \sum_{i=1}^{k} X_i
      - \sum_{i=1}^{k} Z_i \biggr|
    = \smallo\bigl( (a_n( \log(\abstractvar_n^2/a_n)
      + \log \log a_n))^{1/2} \bigr) \quad \text{$\prob$-a.s.}
  \end{equation}
  We have
  \begin{equation}
  \label{eq:asip:non-inv:targets:conclusion:2}
    \sum_{i=1}^{k} X_i(x)
    = \sum_{i=1}^{k} \tilde{\varphi}_i \circ T^{i}(x)
    + \bigo(\log k)
    = \sum_{i=1}^{k} \tilde{\charfun}_i \circ T^{i}(x)
    + \bigo_x(\log \asipvar_k)
  \end{equation}
  by definition of $X_i$ and \eqref{eq:asip:non-inv:targets:sum}.
  Furthermore,
  \begin{equation}
  \label{eq:asip:non-inv:targets:conclusion:3}
    \sum_{k=1}^{n} \EE[Z_k^2]
    = \abstractvar_n^2
    = \asipvar_n^2 + \bigo(\asipvar_n \log \asipvar_n)
  \end{equation}
  by \eqref{eq:asip:non-inv:targets:var} and \eqref{eq:test}.
  Finally, as $a_n = \abstractvar_n^{2-\varepsilon}$, we have
  \begin{equation}
  \label{eq:asip:non-inv:targets:conclusion:4}
    \bigl(a_n( \log( \abstractvar_n^2/a_n ) + \log\log a_n)\bigr)^{1/2}
    \leqs \abstractvar_n^{1-\varepsilon/2} \log \abstractvar_n
    \leqs \asipvar_n^{1-\varepsilon/2} \log \asipvar_n,
  \end{equation}
  again, by \eqref{eq:asip:non-inv:targets:var} and \eqref{eq:test}.
  Since we are free to choose $\varepsilon \in (0,1)$,
  the estimates
  \eqref{eq:asip:non-inv:targets:conclusion:1}%
  --\eqref{eq:asip:non-inv:targets:conclusion:4}
  therefore imply that \eqref{asip} holds for any $\beta < 1/2$.
\end{proof}

\section{Proof of variance estimates for Section~\ref{sec:clt:alt}}
\label{sec:proof:var:explicit}
We now prove the correlation and variance results stated in
\text{\S}\ref{sec:var:explicit}.

\subsection{Proof of Lemma~\ref{lem:r-explicit:correlations}}
The proof is a simple application of
Lemmas~\ref{lem:construction} and \ref{lem:correlation}.
We prove only the upper bound since the lower bound follows in a
similar manner.

\begin{proof}[Proof of Lemma~\ref{lem:r-explicit:correlations}]
  Let $p$ be the H{\"o}lder conjugate of $q > 1$.
  By the hypothesis~\ref{item:r-explicit:correlations:2}, namely that
  $h \in L^{q}(\leb)$,
  Condition~\ref{cond:frostman} is satisfied with $\frost = \dim X/p$,
  and \ref{cond:thin-annuli} is satisfied with $\thinan = 1$ and for
  some $\thinanr$.
  Since $r_k \to 0$, we may without loss of generality assume that
  $r_k \leq \thinanr$ for all $k$.
  For convenience, we write $\frostan = \min \{1,\dim X/p\}$ and
  $d_0 = \dim X$.

  Let $\kappa > 0$ be a constant to be determined,
  and let $\varepsilon_n = n^{-\kappa}$.
  By applying Lemma~\ref{lem:construction} to $\abr = r_k$ 
  and $\varepsilon = \varepsilon_n$, we obtain $\alpha$-H{\"o}lder
  continuous functions $g_{n,k} \colon X \times X \to [0,1]$ that
  satisfy
  \begin{equation}
  \label{eq:lem:cor:explicit:1}
    \mu(\exE_k \cap \exE_j)
    \leq \int_X g_{n,k}(x,T^{k}x) g_{n,j}(x,T^{j}x) \diff\mu(x)
    \quad (0 < k < j \leq n)
  \end{equation}
  and
  \begin{equation}
  \label{eq:lem:cor:explicit:2}
    \int_X \bigl\lvert g_{n,k}(x,y) - \mu( B(x,r_k) ) \bigr\rvert
      \diff\mu(y) 
    \leqs \varepsilon_n^{\frostan}
    \leqs n^{-\kappa \frostan}
    \quad  (k \leq n)
  \end{equation}
  uniformly in $x \in \supp\mu$.

  By part~\ref{item:lem:correlation:2} of Lemma~\ref{lem:correlation}
  with $p = \kappa$,
  \begin{multline}
  \label{eq:lem:cor:explicit:3}
    \int_X g_{n,k}(x,T^{k}x) g_{n,j}(x,T^{j}x) \diff\mu(x) \\
    \leq \iint g_{n,k}(y,T^{k}x) g_{n,j}(y,T^{j}x)
      \diff\mu(x) \diff\mu(y) \\
    + \bigo \bigl( n^{-\kappa} + n^{4(d_0+3)\kappa/\alpha}
      e^{-\mdoc k} \bigr)
  \end{multline}
  Now, for every $y \in X$, apply part~\ref{item:lem:correlation:1} of
  Lemma~\ref{lem:correlation} to the observables
  $g(x,z) = g_{n,k}(y,x) g_{n,j}(y,z)$. This yields
  \begin{multline}
  \label{eq:lem:cor:explicit:4}
    \int_X g_{n,k}(y,T^{k}x) g_{n,j}(y,T^{j}x) \diff\mu(x) \\
    \leq \iint g_{n,k}(y,x) g_{n,j}(y,z)
      \diff\mu(x) \diff\mu(z) \\
    + \bigo \bigl( n^{-\kappa} + n^{4(d_0+3)\kappa/\alpha}
      e^{-\doc (j-k)} \bigr)
  \end{multline}
  uniformly in $y$.

  Combining
  \eqref{eq:lem:cor:explicit:1}--\eqref{eq:lem:cor:explicit:4}, yields
  \begin{multline*}
    \mu(\exE_k \cap \exE_j)
    \leq \int_X \mu(B(y,r_k)) \mu(B(y,r_j)) \diff\mu(y) \\
      + \bigo\bigl( n^{-\kappa\frostan} + n^{-\kappa}
      + n^{4(d_0+3)\kappa/\alpha} (e^{-\mdoc k} + e^{-\doc (j-k)})
    \bigr).
  \end{multline*}
  By choosing $\kappa > \max \{ 2, 2/\frostan \}$ and
  \[
    \lexcorsum = \min \{\kappa, \kappa \frostan \} > 2,
    \quad
    \lexcorcoeff = \frac{4(d_0+3)\kappa}{\alpha},
    \quad
    \lexcor = \min \{ \doc, \mdoc \},
  \]
  we obtain the upper bound.
  The lower bound is established in a similar way.
\end{proof}

\subsection{Proof of Lemma~\ref{lem:r-explicit:liminf}}
The proof is divided into three main steps. The first step is to
introduce a shift in time in the variance $\exsig_n^2$, which will
allow us to obtain better decorrelation estimates.
We then write the variance as two sums, and estimate each sum
separately, which comprise the second and third steps.

\begin{proof}[Proof of Lemma~\ref{lem:r-explicit:liminf}]
  We begin with some notation and a simple reduction, which will assure
  that the assumptions of Lemma~\ref{lem:construction} are satisfied.
  Let
  \[
    S_n(x) = \sum_{k=1}^{n} \charfun_{B(x,r_k)}(T^{k}x) -
    \mu(B(x,r_k)).
  \]
  Then we can write $\exsig_n^2 = \var(S_n)$.
  Let $p$ be the H{\"o}lder conjugate of $q$.
  By the hypothesis~\ref{item:r-explicit:liminf:2}, namely that
  $h \in L^{q}(\leb)$ for some $q > 1$,
  Condition~\ref{cond:frostman} is satisfied with $\frost = \dim X/p$,
  and \ref{cond:thin-annuli} is satisfied with $\thinan = 1$ and for
  some $\thinanr$.
  For convenience, we write $\frostan = \min \{1,\dim X/p\}$ and
  $d_0 = \dim X$.
  By Lemma~\ref{lem:wlog}, we may assume without loss of generality
  that $r_k \leq \thinanr$ for all $k$.

  \step[shifting time]
  We shift time in order to effectively use decay of correlations. 
  Let $D > 0$ be a constant to be specified later, and write
  $\beta_n = D \log n$. Write
  \[
    \exsig_{n,\beta_n}^2 = \int_X \biggl( \sum_{k=1}^{n}
    \charfun_{B(x,r_k)}(T^{k}x) - \mu(B(x,r_k)) \biggr)^2 \diff\mu(x)
  \]
  and notice that
  \[
    \sum_{k=1}^{n} \charfun_{B(x,r_k)}(T^{k}x) - \mu(B(x,r_k))
    = S_n(x) + \bigo(\beta_n).
  \]
  By squaring and integrating both sides of the equation, we have
  \begin{equation}
  \label{eq:var:explicit:shift}
    | \exsig_{n,\beta_n}^2 - \exsig_n^2|
    \leqs \beta_n \exsig_{n,\beta_n} + \beta_n^2.
  \end{equation}
  It therefore suffices to show that
  \begin{equation}
  \label{eq:var:explicit:shift:liminf}
    \liminf_{n \to \infty} \frac{\exsig_{n,\beta_n}^2}{\exesum_n}
    \geq 1
  \end{equation}
  since it will imply $\exsig_{n,\beta_n}^2 \sim \exsig_n^2$
  by \eqref{eq:var:explicit:shift}, and therefore the main part of the
  lemma.

  Expanding the expression for $\exsig_{n,\beta_n}^2$, we have
  \[
    \exsig_{n,\beta_n}^2 = \sum_{k=1}^{n} A_{k+\beta_n}
    + 2\sum_{k=1}^{n-1}\sum_{j=k+1}^{n} B_{k+\beta_n,j+\beta_n},
  \]
  where
  \[
    A_k = \mu(\exE_k)
    - 2 \int_X \charfun_{B(x,r_k)}(T^{k}x) \mu(B(x,r_k)) \diff\mu(x)
    + \int_X \mu(B(x,r_k))^2 \diff\mu(x)
  \]
  and
  \begin{multline*}
    B_{k,j} = \mu(\exE_k \cap \exE_j)
    - \int_X \charfun_{B(x,r_k)}(T^{k}x) \mu(B(x,r_j)) \diff\mu(x) \\
    - \int_X \charfun_{B(x,r_j)}(T^{j}x) \mu(B(x,r_k)) \diff\mu(x) \\
    + \int_X \mu(B(x,r_k)) \mu(B(x,r_j)) \diff\mu(x).
  \end{multline*}

  \step[estimating the leading sum]
  \label{step:r-explicit:liminf:1}
  We first estimate the sum involving $A_{k+\beta_n}$. Let
  $\kappa > 0$ be a constant to be specified later, and let
  $\varepsilon_n = n^{-\kappa}$.
  By applying Lemma~\ref{lem:construction} to $\abr = r_k$ and
  $\varepsilon = \varepsilon_n$, we obtain $\alpha$-H{\"o}lder
  continuous functions $g_{n,k} \colon X \times X \to [0,1]$ that
  satisfy
  \begin{multline}
  \label{eq:var:explicit:upper}
    \int_X \charfun_{B(x,r_k)}(T^{k}x) \mu(B(x,r_k)) \diff\mu(x) \\
    \leq \iint g_{n,k}(x,T^{k}x) g_{n,k}(x,y) \diff\mu(y)
      \diff\mu(x)
    \quad (k \leq n)
  \end{multline}
  and
  \begin{equation}
  \label{eq:var:explicit:approx}
    \int_X \bigl| g_{n,k}(x,y) - \mu(B(x,r_k)) \bigr| \diff\mu(y)
    \leqs \varepsilon_n^{\frostan} \leqs n^{-\kappa\frostan}
    \quad (k \leq n)
  \end{equation}
  uniformly for $x \in \supp\mu$.
  Applying part~\ref{item:lem:correlation:1} of
  Lemma~\ref{lem:correlation} with $p = \kappa$ to the functions
  $g(x,z) = g_{n,k}(x,z) \int g_{n,k}(x,y) \diff\mu(y)$, we have
  \begin{multline*}
    \iint g_{n,k}(x,T^{k}x) g_{n,k}(x,y) \diff\mu(y) \diff\mu(x)
    \\
    \leq \iiint g_{n,k}(x,z) g_{n,k}(x,y) \diff\mu(x)
      \diff\mu(y) \diff\mu(z) \\
    \quad + \bigo\bigl(n^{-\kappa} + n^{4(d_0+3)\kappa/\alpha}
      e^{-\doc k} \bigr).
  \end{multline*}
  By \eqref{eq:var:explicit:approx}, we have
  \begin{multline*}
    \iiint g_{n,k}(x,z) g_{n,k}(x,y) \diff\mu(y)
      \diff\mu(z) \diff\mu(x) \\
    \leq \int_X \mu(B(x,r_k))^2 \diff\mu(x)
      + \bigo( n^{-\kappa\frostan}),
  \end{multline*}
  which implies, by \eqref{eq:var:explicit:upper},
  \begin{multline*}
    \int_X \charfun_{B(x,r_k)}(T^{k}x) \mu(B(x,r_k)) \diff\mu(x) \\
    \leq \int_X \mu(B(x,r_k))^2 \diff\mu(x) \\
      + \bigo\bigl( n^{-\kappa \frostan} + n^{-\kappa}
      + n^{4(d_0+3)\kappa/\alpha} e^{-\doc k} \bigr).
  \end{multline*}
  Hence,
  \begin{multline*}
    \sum_{k=1}^{n} A_{k+\beta_n}
    \geq \sum_{k=1}^{n} \mu(\exE_{k+\beta_n})
      - \int_X \mu(B(x,r_{k+\beta_n}))^2 \diff\mu(x) \\
    + \bigo\bigl( n^{1-\kappa\frostan} + n^{1-\kappa}
      + n^{4(d_0+3)\kappa/\alpha} e^{-\doc \beta_n}\bigr).
  \end{multline*}
  Selecting $\kappa > \max \{1, 1/\frostan\}$, and then $D$ in
  $\beta_n = D \log n$ sufficiently large, we have that the error
  converges to $0$ as $n \to \infty$.
  By shifting time again, we therefore have
  \[
    \sum_{k=1}^{n} A_{k+\beta_n}
    \geq \sum_{k=1}^{n} \biggl( \mu(\exE_k) - \int_X \mu(B(x,r_k))^2
      \diff\mu(x) \biggr) - \beta_n + \smallo(1).
  \]

  Now, recall that $\mu(B(x,r)) \leqs r^{d_0/p}$.
  Thus,
  \[
    \int_X \mu(B(x,r_k))^2 \diff\mu(x)
    \leq r_k^{d_0/p} \int_X \mu(B(x,r_k)) \diff\mu(x)
    = \smallo( \mu(\exE_k) ),
  \]
  where we have employed $r_k \to 0$ and
  Lemma~\ref{lem:r-explicit:meas} in the last inequality.
  Applying this estimate to the bound above, we have
  \[
    \sum_{k=1}^{n} A_{k+\beta_n}
    \geq \sum_{k=1}^{n} \bigl( \mu(\exE_k) - \smallo(\exE_k) \bigr)
    - \beta_n + \smallo(1).
  \]
  By Lemma~\ref{lem:r-explicit:meas} and the
  hypothesis~\ref{item:r-explicit:liminf:3}, we have
  \[
    \mu(\exE_k) \geq \int_X \mu(B(x,r_k)) \diff\mu(x) - e^{-\lexmeas k}
    \geqs r_k^{d_0} - e^{-\lexmeas k}
    \geqs k^{-\seqlow d_0},
  \]
  where we recall that $\seqlow \in (0,1/d_0)$.
  In particular,
  \begin{equation}
  \label{eq:var:explicit:liminf:esum}
    \exesum_n = \sum_{k=1}^{n} \mu(\exE_k) \geqs n^{1-\seqlow d_0}
  \end{equation}
  and $\beta_n = \smallo(\exesum_n)$. Therefore,
  \begin{equation}
  \label{eq:var:explicit:sum:1}
    \sum_{k=1}^{n} A_{k+\beta_n}
    \geq \exesum_n - \smallo( \exesum_n),
  \end{equation}

  \step[estimating the double sum]
  We now estimate the double sum involving $B_{k+\beta_n,j+\beta_n}$.
  We consider the cases $j \leq k + \beta_n$ and $j > k + \beta_n$
  separately.

  For $j \leq k + \beta_n$, we have
  \[
    \int_X \charfun_{B(x,r_k)}(T^{k}x) \mu(B(x,r_j)) \diff\mu(x)
    \leqs r_j^{d_0/p} \mu(\exE_k) \leqs r_k^{d_0/p} \mu(\exE_k),
  \]
  where the second inequality follows by the monotonicity of $(r_k)$.
  Similarly, 
  \[
    \int_X \charfun_{B(x,r_j)}(T^{j}x) \mu(B(x,r_k)) \diff\mu(x)
    \leqs r_k^{d_0/p} \mu(\exE_k),
  \]
  as the monotonicity of $(r_k)$ transfers to monotonicity up to a
  constant coefficient of $(\mu(\exE_k))$ by
  Lemma~\ref{lem:r-explicit:meas}. Hence,
  \[
    \sum_{k=1}^{n-1} \sum_{j=k+1}^{k+\beta_n} B_{k+\beta_n, j+\beta_n}
    \geqs - \beta_n \sum_{k=1}^{n-1} r_{k+\beta_n}^{d_0/p}
      \mu(\exE_{k+\beta_n})
    \geqs -\beta_n^2 - \beta_n \sum_{k=1}^{n-1} r_k^{d_0/p}
      \mu(\exE_k),
  \]
  where we shift time back in the last inequality.
  Now let $s \in (0,1-\seqlow d_0)$.
  Then
  \[
    \beta_n \sum_{k=1}^{n-1} r_k^{d_0/p} \mu(\exE_k)
    \leqs \beta_n n^{s} + \sum_{k=n^{s}}^{n-1} r_k^{d_0/p} \mu(\exE_k)
      \log k
    = \smallo(\exesum_n),
  \]
  where we have used that $\exesum_n \geqs n^{1-\seqlow d_0}$ (see
  \eqref{eq:var:explicit:liminf:esum}), and $r_k \to 0$ and
  $r_k \leqs (\log k)^{-\seqhigh}$ for some $\seqhigh > p/d_0$
  (hypothesis~\ref{item:r-explicit:correlations:3}).
  The two previous equations therefore imply
  \begin{equation}
  \label{eq:var:explicit:sum:2:short}
    \sum_{k=1}^{n-1} \sum_{j=k+1}^{k+\beta_n} B_{k+\beta_n, j+\beta_n}
    \geq - \smallo(\exesum_n).
  \end{equation}

  For $j > k + \beta_n$, a similar argument to the one used in
  Step~\ref{step:r-explicit:liminf:1} can be used to equate
  $\sum_{k=1}^{n} \sum_{j=k+\beta_n+1}^{n} B_{k+\beta_n,j+\beta_n}$
  with
  \begin{multline}
  \label{eq:var:explicit:sum:2:1}
    \sum_{k=1}^{n-1} \sum_{j=k+\beta_n+1}^{n} \Bigl(
      \mu(\exE_{k+\beta_n} \cap \exE_{k+\beta_n})
      - \int_X \mu(B(x,r_{k+\beta_n})) \mu(B(x,r_{j+\beta_n}))
      \diff\mu(x) \\
    + \bigo\bigl( n^{-\kappa'} + n^{4(d_0+3)\kappa'}
      e^{-\mdoc (k+\beta_n)} \bigr) \Bigr),
  \end{multline}
  where $\kappa' > 0$ is a constant to be specified.
  By Lemma~\ref{lem:r-explicit:correlations}, there exist
  $\lexcorsum > 2$ and $\lexcorcoeff, \lexcor > 0$ such that, for all
  $k < j \leq n$,
  \begin{multline*}
    \Bigl| \mu(\exE_{k+\beta_n} \cap \exE_{k+\beta_n})
      - \int_X \mu(B(x,r_{k+\beta_n})) \mu(B(x,r_{j+\beta_n}))
      \diff\mu(x) \Bigr| \\
    \leqs n^{-\lexcorsum} + n^{\lexcorcoeff}
      \bigl( e^{-\lexcor (k+\beta_n)} + e^{-\lexcor(j-k)} \bigr).
  \end{multline*}
  In view of \eqref{eq:var:explicit:sum:2:1}, we therefore have
  \[
  \begin{split}
    \sum_{k=1}^{n-1} \sum_{j=k+\beta_n+1}^{n} |B_{k+\beta_n,j+\beta_n}|
    &\leqs \sum_{k=1}^{n-1} \sum_{j=k+\beta_n+1}^{n}
      \bigl( n^{-p} + n^{q} \bigl( e^{-\eta'(k+\beta_n)}
      + e^{-\eta'(j-k)} \bigr) \bigr) \\
    &\leqs n^{2-p} + n^{q} \sum_{k=1}^{n-1} \bigl(
      e^{-\eta'(k+\beta_n)} + e^{-\eta'\beta_n} \bigr) \\
    &\leqs n^{2-p} + n^{q+1} e^{-\eta'\beta_n},
  \end{split}
  \]
  for $p = \min \{\lexcorsum, \kappa'\}$, and some $\eta',q > 0$.
  By choosing $\kappa' > 2$, we have $p > 2$. In turn, choosing $D > 0$
  sufficiently large in $\beta_n = D \log n$ yields
  \begin{equation}
  \label{eq:var:explicit:sum:2:long}
    \sum_{k=1}^{n-1} \sum_{j=k+\beta_n+1}^{n} B_{k+\beta_n,j+\beta_n}
    = \smallo(1).
  \end{equation}

  In view of \eqref{eq:var:explicit:sum:2:short} and
  \eqref{eq:var:explicit:sum:2:long}, we have
  \begin{equation}
  \label{eq:var:explicit:sum:2}
    \sum_{k=1}^{n-1}\sum_{j=k+1}^{n} B_{k+\beta_n, j+\beta_n}
    \geq - \smallo(\exesum_n).
  \end{equation}

  \conclusion
  Combining \eqref{eq:var:explicit:sum:1} and
  \eqref{eq:var:explicit:sum:2}, we have
  \[
    \exsig_{n,\beta_n}^2
    \geq \exesum_n - \smallo( \exesum_n ),
  \]
  which implies \eqref{eq:var:explicit:shift:liminf}.
  We have thus established the main part of the lemma.
  By \eqref{eq:var:explicit:liminf:esum}, we infer that
  $\sigma_n^2 \geqs n^{1-\seqlow d_0}$.
\end{proof}

\subsection{Proof of Proposition~\ref{prop:r-explicit:L1}}
We will first state and prove a lemma that converts
Condition~\ref{cond:sr:explicit} to an estimate involving the sets
$\exE_k$.

\begin{lemma}
  \label{lem:r-explicit:sr}
  Assume that Condition~\ref{cond:sr:explicit} holds with the constant
  $\srelog > 1$. Then there exists $\sremin > 0$ such that
  \[
    \mu( \exE_k \cap \exE_{k+l} )
    \leqs \mu(\exE_k)^{1+\sremin}
  \]
  for all $k \geq 1$ and $1 \leq l \leq (\log k)^{\srelog}$.
\end{lemma}

\begin{proof}
  Recall from the proof of Lemma~\ref{lem:r-explicit:liminf} that
  Conditions~\ref{cond:frostman} and \ref{cond:thin-annuli} are
  satisfied with $\frost = \dim X/p$ and $\thinan = 1$, where $p$ is the
  H{\"o}lder conjugate of $q$. We may also assume without loss of
  generality that $r_k \leq \thinanr$ for all $k$.
  For convenience, we write $\frostan = \min \{1,\dim X/p\}$ and
  $d_0 = \dim X$.

  Let $\kappa > 0$ be a constant to be specified later, and let
  $\varepsilon_k = k^{-\kappa}$.
  By applying Lemma~\ref{lem:construction} to $\abr = r_k$ and
  $\varepsilon = \varepsilon_k$, we obtain $\alpha$-H{\"o}lder
  continuous functions $g_k \colon X \times X \to [0,1]$
  that satisfy
  \begin{equation}
  \label{eq:var:explicit:goal:1:upper}
    \mu(\exE_k \cap \exE_{k+l})
    \leq \int_X g_k(x,T^{k}x) g_{k+l}(x,T^{k+l}x) \diff\mu(x)
    \quad (k,l \geq 1)
  \end{equation}
  and
  \begin{equation}
  \label{eq:var:explicit:goal:1:measure}
    \int_X \bigl| g_k(x,y) - \charfun_{B(x,r_k)}(y) \bigr| \diff\mu(y)
    \leqs \varepsilon_k^{\frostan} \leqs k^{-\kappa \frostan}
    \quad (k \geq 1)
  \end{equation}
  uniformly in $x \in \supp\mu$. Using
  part~\ref{item:lem:correlation:2} of Lemma~\ref{lem:correlation} with
  $p = \kappa$ and $n = k + l \leqs k$, we obtain
  \begin{multline*}
    \int_X g_k(x,T^{k}x) g_{k+l}(x,T^{k+l}x) \diff\mu(x)
    = \iint g_k(y,x) g_{k+l}(y,T^{l}x) \diff\mu(x) \diff\mu(y) \\
    + \bigo(k^{-\kappa} + k^{4(d_0+3)\kappa/\alpha} e^{-\mdoc k} ).
  \end{multline*}
  By \eqref{eq:var:explicit:goal:1:measure},
  \[
    \int_X \bigl| g_k(y,x) g_{k+l}(y,T^{l}x) - \charfun_{B(y,r_k)}(x)
    \charfun_{B(y,r_{k+l})}(T^{l}x) \bigr| \diff\mu(x)
    \leqs k^{-\kappa\frostan}
  \]
  uniformly in $y \in \supp\mu$. In view of
  \eqref{eq:var:explicit:goal:1:upper}, we therefore have
  \begin{multline*}
    \mu(\exE_k \cap \exE_{k+l})
    \leq \int_X \mu\bigl(B(y,r_k) \cap T^{-l} B(y,r_{k+l}) \bigr)
      \diff\mu(y) \\
    + \bigo( k^{-\kappa \frostan} + k^{-\kappa}
    + k^{4(d+3)\kappa/\alpha} e^{-\mdoc k} ).
  \end{multline*}

  Let
  $A(k,l) = \{y \in X : \mu(B(y,r_k) \cap T^{-l} B(y,r_k)) >
  \exAv_k^{1+\srein}\}$,
  where we recall that $\exAv_k = \int \mu(B(x,r_k)) \diff\mu(x)$.
  We now use Condition~\ref{cond:sr:explicit}. For all $k \geq 1$ and
  $1 \leq l \leq (\log k)^{\srelog}$, we have
  \begin{equation}
  \label{eq:var:explicit:l1:1:sr}
  \begin{split}
    \int_X \mu\bigl(B(y,r_k) \cap T^{-l} B(y,r_{k+l}) \bigr)
      \diff\mu(y)
    &\leq \exAv_k^{1+\srein} + \int_{A(k,l)} \mu(B(y,r_k))
      \diff\mu(y) \\
    &\leq \exAv_k^{1+\srein} + \mu(A(k,l))\exAv_k \\
    &\leq \exAv_k^{1+\srein} + \exAv_k^{1+\sreout} \\
    &\leqs \exAv_k^{1+\sremin},
  \end{split}
  \end{equation}
  where $\sremin = \min \{\srein, \sreout\}$.
  Therefore,
  \[
    \mu(\exE_k \cap \exE_{k+l})
    \leqs \exAv_k^{1+\sremin}  + k^{-\kappa \frostan} + k^{-\kappa}
      + k^{4(d+3)\kappa/\alpha} e^{-\mdoc k}.
  \]

  By the hypotheses~\ref{item:r-explicit:L1:2} and
  \ref{item:r-explicit:L1:4}, we have
  $\exAv_k \geq r_k^{\dim X} \geqs k^{-\seqlow \dim X}$.
  Therefore, by choosing $\kappa > 0$ sufficiently large, we have
  \[
    \mu(\exE_k \cap \exE_{k+l}) \leqs \exAv_k^{1+\sremin}.
  \]
  By Lemma~\ref{lem:r-explicit:meas}, we have
  $\exAv_k \leqs \mu(\exE_k)$, and we may conclude the proof.
\end{proof}

We now proceed with the proof of Proposition~\ref{prop:r-explicit:L1}.
We will use the general probabilistic result of
Proposition~\ref{prop:kallenberg}.

\begin{proof}[Proof of Proposition~\ref{prop:r-explicit:L1}]
  We divide the proof into three steps. First, we show convergence of
  the average of $\exs_n^2/\exsig_n^2$
  (Step~\ref{step:prop:r-explicit:1}), which follows mostly
  from the proof of Lemma~\ref{lem:r-explicit:liminf} together with an
  application of Lemma~\ref{lem:r-explicit:sr}.
  Then, we establish a pointwise asymptotic formula for
  $\sum_{k=1}^{n}\mu(B(y,r_k))$ (Step~\ref{step:prop:r-explicit:2})
  using the Hardy--Littlewood maximal function,
  which will then be used in establishing convergence in measure of
  $\exs_n^2/\exsig_n^2$ (Step~\ref{step:prop:r-explicit:3}) using
  Markov's inequality in the process.
  We then conclude by Proposition~\ref{prop:kallenberg}.

  As before, Conditions~\ref{cond:frostman} and \ref{cond:thin-annuli}
  are satisfied with $\frost = \dim X/p$ and $\thinan = 1$, where $p$ is
  the H{\"o}lder conjugate of $q$.
  By Lemma~\ref{lem:wlog}, we may assume without loss of generality
  that $r_k \leq \thinanr$ for all $k$.
  For convenience, we write $\frostan = \min \{1,\dim X/p\}$,
  $d_0 = \dim X$,
  $\exAv_k = \int \mu(B(x,r_k)) \diff\mu(x)$,
  \[
    \exesum_n = \sum_{k=1}^{n} \mu(\exE_k)
    \quad \text{and} \quad
    \exmsum_n = \sum_{k=1}^{n} \exAv_k.
  \]
  In Steps~\ref{step:prop:r-explicit:1}
  and \ref{step:prop:r-explicit:3}, we will consider short returns, and
  so we introduce $\beta_n = D \log n$, where $D > 0$ is a constant to
  be specified later.

  \step[Convergence of the average]
  \label{step:prop:r-explicit:1}
  We first prove
  \begin{equation}
  \label{eq:var:explicit:l1:goal:2}
    \lim_{n \to \infty} \int_X \frac{\exs_n^2(y)}{\exsig_n^2}
    \diff\mu(y) = 1.
  \end{equation}
  The limit follows from the following two:
  \begin{equation}
  \label{eq:var:explicit:sigma:limit}
    \lim_{n \to \infty} \frac{\exsig_n^2}{\exesum_n} = 1
  \end{equation}
  and
  \begin{equation}
  \label{eq:var:explicit:s:limit}
    \lim_{n \to \infty} \int_X \frac{\exs_n^2(y)}{\exmsum_n}
    \diff\mu(y) = 1.
  \end{equation}
  Indeed,
  \[
    \lim_{n \to \infty} \int_X \frac{\exs_n^2(y)}{\exsig_n^2}
    \diff\mu(y)
    = \lim_{n \to \infty} \frac{\exesum_n}{\exsig_n^2}
    \cdot \lim_{n \to \infty} \frac{\exmsum_n}{\exesum_n} \cdot
    \lim_{n \to \infty} \int_X \frac{\exs_n^2(y)}{\exmsum_n}
    \diff\mu(y) = 1,
  \]
  where we have used that $\exesum_n \sim \exmsum_n$
  by Lemma~\ref{lem:r-explicit:meas}.
  We choose to only show \eqref{eq:var:explicit:sigma:limit} since
  \eqref{eq:var:explicit:s:limit} follows from a similar argument. By
  Lemma~\ref{lem:r-explicit:liminf}, it remains to show
  \[
    \limsup_{n \to \infty} \frac{\exsig_n^2}{\exesum_n} \leq 1.
  \]
  By \eqref{eq:var:explicit:sum:2:long}, it in turn remains to show
  \begin{equation}
  \label{eq:var:explicit:goal:2:main}
    \sum_{k=1}^{n} \mu(\exE_k) + \int_X \mu(B(x,r_k))^2 \diff\mu(x)
    \leq \exesum_n + \smallo(\exesum_n)
  \end{equation}
  and
  \begin{equation}
  \label{eq:var:explicit:goal:2:short}
    \sum_{k=1}^{n-1} \sum_{j=k+1}^{k+\beta_n}
    \mu(\exE_k \cap \exE_j) +
    \int_X \mu(B(x,r_k)) \mu(B(x,r_j)) \diff\mu(x)
    = \smallo(\exesum_n).
  \end{equation}

  To see \eqref{eq:var:explicit:goal:2:main}, notice that
  \[
    \sum_{k=1}^{n} \int_X \mu(B(x,r_k))^2 \diff\mu(x)
    \leqs \sum_{k=1}^{n} \exAv_k r_k^{d_0/p}
    \leqs \sum_{k=1}^{n} \mu(\exE_k) r_k^{d_0/p}
    = \smallo(\exesum_n)
  \]
  by Lemma~\ref{lem:r-explicit:meas} and the hypothesis that
  $r_k \to 0$.

  To see \eqref{eq:var:explicit:goal:2:short}, notice that,
  by Lemmas~\ref{lem:r-explicit:meas} and \ref{lem:r-explicit:sr},
  \[
    \mu( \exE_k \cap \exE_j )
    \leqs \mu(\exE_k)^{1+\sremin}
    \leqs \mu(\exE_k) \exAv_k^{\sremin}
  \]
  for all $1 \leq j - k \leq (\log k)^{\srelog}$.
  By the hypothesis~\ref{item:r-explicit:L1:2},
  $\exAv_k \leqs r_k^{d_0/p}$. Hence
  \[
    \mu( \exE_k \cap \exE_j )
    \leqs \mu(\exE_k) r_k^{\sremin d_0/p}
  \]
  for all $1 \leq j - k \leq (\log k)^{\srelog}$.
  Also, the monotonicity of $(r_k)$ implies
  \[
    \int_X \mu(B(x,r_k)) \mu(B(x,r_j)) \diff\mu(x)
    \leqs \exAv_k r_k^{d_0/p}
    \leqs \mu(\exE_k) r_k^{\sremin d_0/p}.
  \]
  Hence, for all $s \in (0,1-\seqlow d_0)$,
  \[
  \begin{split}
    \sum_{k=1}^{n-1} \sum_{j=k+1}^{k+\beta_n}
    \mu(\exE_k \cap \exE_j) + \int_X \mu(B(x,r_k)) & \mu(B(x,r_j))
      \diff\mu(x) \\
    &\leqs \beta_n n^{s} + \sum_{k=n^{s}}^{n-1} \sum_{j=k+1}^{k+\beta_n}
      \mu(\exE_k) r_k^{\sremin d_0/p} \\
    &\leqs \beta_n n^{s} + \sum_{k=n^{s}}^{n-1}
      \mu(\exE_k) r_k^{\sremin d_0/p} \log k.
  \end{split}
  \]
  Now, 
  $\beta_n n^{s} = \smallo(n^{1-\seqlow d_0}) = \smallo(\exesum_n)$
  and
  \[
    \mu(\exE_k) r_k^{\sremin d_0/p} \log k
    = \smallo( \mu(\exE_k) )
  \]
  by the hypothesis~\ref{item:r-explicit:L1:4}.
  We have therefore shown \eqref{eq:var:explicit:sigma:limit}, which
  concludes the first step.

  \step[Pointwise convergence of the sum]
  \label{step:prop:r-explicit:2}
  We now prove
  \begin{equation}
  \label{eq:var:explicit:l1:goal:3}
    \lim_{n \to \infty}
    \frac{\sum_{k=1}^{n}\mu(B(y,r_k))}{\sum_{k=1}^{n} \exAv_k}
    = \frac{h(y)}{\mu(h)}
    \quad \text{$\mu$-a.e.\ $y$}.
  \end{equation}

  For each $k \geq 1$, define $F_k \colon X \to [0,\infty]$ by
  \[
    F_k(y) \coleq \frac{\mu(B(y,r_k))}{\leb(B(y,r_k))}
    = \frac{1}{\leb(B(y,r_k))} \int_{B(y,r_k))} h \diff\leb
  \]
  By the Lebesgue differentiation theorem,
  $\lim_{k \to \infty}F_k(y) = h(y)$ for $\leb$-a.e.\ $y$.
  We will first show that
  \begin{equation}
  \label{eq:var:explicit:l1:goal:3:F}
    \lim_{n \to \infty} \int_X F_k \diff\mu = \int_X h \diff\mu
  \end{equation}
  by using the Hardy--Littlewood maximal function. We will then use
  \eqref{eq:var:explicit:l1:goal:3:F} to obtain estimates for
  $\exAv_k$.

  The Hardy--Littlewood maximal function $\hlmax$ is defined by
  \[
    (\hlmax f)(x) = \sup_{r > 0} \frac{1}{\leb(B(x,r))} \int_{B(x,r)}
    |f| \diff\leb
  \]
  for any locally integrable function $f$.
  It is clear that $F_k \leq \hlmax h$ for all $k \geq 1$.
  If $f \in \Ell{p}(\leb)$ for some $p > 1$, then 
  \[
    \lnorm{\hlmax f}{p} \leqs_p \lnorm{f}{p}
  \]
  see for example \cite[Theorem~1]{stein-1970-singular}. In particular, 
  since $h \in \Ell{q'}(\leb)$ for all $q' \in (1,q]$,
  \[
    \lnorm{F_k}{p} \leq \lnorm{\hlmax h}{q'} \leqs_{q'} \lnorm{h}{q'}
    \quad \text{for all } q' \in (1,q].
  \]
  Therefore, using that $q > 2$, the limit in
  \eqref{eq:var:explicit:l1:goal:3:F} and the fact that $(X,\leb)$ is a
  finite measure-space (by compactness), we have
  $\lim_{k \to \infty} \normsp{F_k - h}{\Ell{2}(\leb)} = 0$.
  By the Cauchy--Schwarz inequality, we therefore have
  \[
    \lim_{k \to \infty} \normsp{F_k - h}{\Ell{1}(\mu)} = 0.
  \]

  At the same time
  \[
    \exAv_k
    = \int_X \mu\bigl(B(x,r_k)\bigr) \diff\mu(x)
    = \int_X \leb(B(x,r_k)) F_k(x) \diff\mu(x).
  \]
  Since $(X,d)$ is a compact Riemannian manifold, we have
  $\leb(B(x,r)) = \leb(B(y,r))( 1 + \bigo(r))$ for all $x,y \in X$ and
  all sufficiently small $r > 0$. Thus, for sufficiently large $k$ and
  any $x_0 \in X$, we have
  \[
    \exAv_k = \leb(B(x_0,r_k)) ( 1 + \bigo(r_k)) \int_X F_k(x)
    \diff\mu(x).
  \]
  By \eqref{eq:var:explicit:l1:goal:3:F},
  \[
    \exAv_k = \leb(B(x_0,r_k)) ( 1 + \bigo(r_k))
    \Bigl( \int_X h \diff\mu + \smallo(1) \Bigr)
  \]
  for all sufficiently large $k$. By summing over $k=1,\ldots,n$,
  \[
    \sum_{k=1}^{n} \exAv_k
    = \Bigl( \int_X h \diff\mu + \smallo(1) \Bigr)
    \sum_{k=1}^{n} \leb(B(x_0,r_k)) ( 1 + \bigo(r_k)).
  \]
  By the hypothesis~\ref{item:r-explicit:L1:4},
  $r_k \to 0$ and
  $\sum_{k=1}^{\infty} \leb(B(x_0,r_k)) \geqs \sum_{k=1}^{\infty}
  r_k^{d_0} = \infty$. Therefore,
  \[
    \lim_{n \to \infty} \frac{\sum_{k=1}^{n} \exAv_k}%
    {\sum_{k=1}^{n} \leb( B(y,r_k) )} = \int_X h \diff\mu
    \quad \text{for all } y \in X.
  \]
  Using a similar argument, we have
  \[
    \lim_{n \to \infty} \frac{\sum_{k=1}^{n} \mu( B(y,r_k) )}%
    {\sum_{k=1}^{n} \leb(B(y,r_k))}
    = h(y) \quad \text{$\mu$-a.e.\ $y \in X$}.
  \]
  By combining the previous two equations, we obtain
  \eqref{eq:var:explicit:l1:goal:3}.

  \step[Convergence in measure]
  \label{step:prop:r-explicit:3}
  Using \eqref{eq:var:explicit:l1:goal:3}, we will now prove
  \begin{equation}
  \label{eq:var:explicit:l1:goal:4}
    \frac{\exs_n^2}{\exsig_n^2} \xrightarrow[n \to \infty]{\mu}
    \frac{h}{\mu(h)}.
  \end{equation}

  Expanding the expression for $\exs_n^2$, we have
  \begin{multline}
  \label{eq:prop:r-explicit:s}
    \exs_n^2(y) = \sum_{k=1}^{n} \mu(B(y,r_k)) - \mu(B(y,r_k))^2 \\
    + \sum_{k=1}^{n-1}\sum_{j=k+1}^{n} \covar\bigl(
      \charfun_{B(y,r_k)}, \charfun_{B(y,r_j)} \circ T^{j-k} \bigr).
  \end{multline}
  By \eqref{eq:var:explicit:sigma:limit} in
  Step~\ref{step:prop:r-explicit:1}, we have
  $\exsig_n^2 \sim \exesum_n$.
  Since $\exesum_n \sim \exmsum_n = \sum_{k=1}^{n} \exAv_k$ by
  Lemma~\ref{lem:r-explicit:meas}, we therefore have, for all
  $\varepsilon > 0$,
  \begin{equation}
  \label{eq:var:explicit:l1:4:meas:1}
    \lim_{n \to \infty} \mu \biggl\{y : \biggl| \frac{\sum_{k=1}^{n}
    \mu(B(y,r_k))}{\exsig_n^2} - \frac{h(y)}{\mu(h)} \biggr|
    > \varepsilon \biggr\} = 0
  \end{equation}
  by \eqref{eq:var:explicit:l1:goal:3} in
  Step~\ref{step:prop:r-explicit:2}. Additionally, since $r_k \to 0$,
  \begin{equation}
  \label{eq:prop:r-explicit:s:sum:1}
    \sum_{k=1}^{n} \int \mu(B(y,r_k))^2 \diff\mu(y)
    \leqs \sum_{k=1}^{n} \exAv_k r_k^{d_0/p}
    = \smallo(\exmsum_n) = \smallo(\exsig_n^2).
  \end{equation}
  By \eqref{eq:var:explicit:l1:4:meas:1} and
  \eqref{eq:prop:r-explicit:s:sum:1},
  it suffices to show that the double sum in
  \eqref{eq:prop:r-explicit:s} converges to $0$ in measure in order to
  establish \eqref{eq:var:explicit:l1:goal:4}.

  Since the centres of the balls in $\exs_n^2(y)$ do not depend on
  iterated point, we can apply a similar yet simpler argument as in the
  proof of Lemma~\ref{lem:r-explicit:correlations} in order to control
  correlations for $j \geq k + \beta_k$. That is, we obtain
  \begin{equation}
  \label{eq:var:explicit:l1:4:meas:long}
      \sum_{k=1}^{n-1}\sum_{j=k+\beta_k+1}^{n} \bigl| \covar\bigl(
        \charfun_{B(y,r_k)}, \charfun_{B(y,r_j)} \circ T^{j-k} \bigr)
      \bigr|
      = \bigo(1) = \smallo(\exsig_n^2).
  \end{equation}
  It therefore remains to control the terms when $j \leq k + \beta_k$.

  By Markov's inequality,
  \begin{multline}
  \label{eq:var:explicit:l1:4:markov}
    \mu \biggl\{y \in X : \biggl|
      \frac{1}{\exsig_n^2}
      \sum_{k=1}^{n-1}\sum_{j=k+1}^{k+\beta_k} \covar\bigl(
        \charfun_{B(y,r_k)}, \charfun_{B(y,r_j)} \circ T^{j-k} \bigr)
      \biggr| > \varepsilon \biggr\} \\
    \leqs \frac{1}{\varepsilon^2 \exsig_n^2}
    \sum_{k=1}^{n-1}\sum_{j=k+1}^{k+\beta_k} \int_X \mu\bigl(B(y,r_k)
      \cap T^{-(j-k)}B(y,r_j)\bigr) \diff\mu(y).
  \end{multline}
  By Condition~\ref{cond:sr:explicit}, we have
  \[
    \sum_{j=k+1}^{k+\beta_k} \int_X \mu\bigl(B(y,r_k) \cap
      T^{-(j-k)} B(y,r_j)\bigr) \diff\mu(y)
    \leqs \exAv_k^{1+\sremin}\beta_k 
    \leqs \exAv_k r_k^{\sremin d_0/p} \log k
  \]
  where $\sremin = \min \{\srein, \sreout\}$ (see
  \eqref{eq:var:explicit:l1:1:sr} in the proof of
  Lemma~\ref{lem:r-explicit:sr}).
  Now, by the hypothesis~\ref{item:r-explicit:L1:4},
  $r_k \leqs (\log k)^{-\seqhigh}$ for some
  $\seqhigh > p(\sremin d_0)^{-1}$. Therefore,
  \[
    \sum_{j=k+1}^{k+\beta_k} \int_X \mu\bigl(B(y,r_k) \cap
      T^{-(j-k)} B(y,r_j)\bigr) \diff\mu(y)
    = \smallo( \exAv_k )
  \]
  and
  \[
    \sum_{k=1}^{n-1}\sum_{j=k+1}^{k+\beta_k} \int_X \mu\bigl(B(y,r_k)
      \cap T^{-(j-k)}B(y,r_j)\bigr) \diff\mu(y)
    = \smallo(\exmsum_n) = \smallo(\exsig_n^2).
  \]
  Hence,
  \begin{equation}
  \label{eq:var:explicit:l1:4:meas:2}
    \lim_{n \to \infty} \mu \biggl\{y \in X : \biggl|
    \frac{1}{\exsig_n^2} \sum_{k=1}^{n-1}\sum_{j=k+1}^{k+\beta_k}
    \covar\bigl( \charfun_{B(y,r_k)}, \charfun_{B(y,r_j)} \circ T^{j-k}
      \bigr)
    \biggr| > \varepsilon \biggr\} = 0
  \end{equation}
  for all $\varepsilon > 0$ in view of
  \eqref{eq:var:explicit:l1:4:markov}.

  Finally, we obtain \eqref{eq:var:explicit:l1:goal:4} by combining
  \eqref{eq:var:explicit:l1:4:meas:1},
  \eqref{eq:var:explicit:l1:4:meas:long} and
  \eqref{eq:var:explicit:l1:4:meas:2}.

  \conclusion
  We now employ Proposition~\ref{prop:kallenberg}, after which the
  conclusion follows from a simple calculation.

  By \eqref{eq:var:explicit:l1:goal:4} in
  Step~\ref{step:prop:r-explicit:3}, we have
  $f_n \tomeas{\mu} f$, where
  $f_n \coleq \exs_n^2/\exsig_n^2 \in \Ell{1}(\mu)$
  and $f \coleq h/\mu(h)$.
  By \eqref{eq:var:explicit:l1:goal:2}, we have
  $\lnorm{f_n}{1} \to 1 = \lnorm{h/\mu(h)}{1} < \infty$.
  Hence, by Proposition~\ref{prop:kallenberg},
  \[
    \lim_{n \to \infty} \int_X \biggl|
    \frac{\exs_n^2(y)}{\exsig_n^2} - \frac{h(y)}{\mu(h)} \biggr|
    \diff\mu(y) = 0
  \]
  By the hypothesis~\ref{item:r-explicit:L1:2},
  \[
    \biggr| \frac{\exs_n}{\exsig_n}
      + \frac{\sqrt{h}}{\sqrt{\mu(h)}} \biggr|
    \geq \frac{\sqrt{h}}{\sqrt{\mu(h)}}
    \geq \sqrt{\frac{c}{\mu(h)}} > 0.
  \]
  Therefore,
  \[
    \lim_{n \to \infty} \int_X \biggl| \frac{\exs_n}{\exsig_n}
      - \frac{\sqrt{h}}{\sqrt{\mu(h)}} \biggr| \diff\mu
    \leqs \lim_{n \to \infty} \int_X \biggl|
    \frac{\exs_n^2}{\exsig_n^2} - \frac{h}{\mu(h)} \biggr|
    \diff\mu = 0,
  \]
  which concludes the proof.
\end{proof}

\section{Proof of variance estimates for Section~\ref{sec:clt:clt}}
We now prove the correlation and variance results of
\text{\S}\ref{sec:var:implicit}. The proofs will follow a similar
structure to those in \text{\S}\ref{sec:proof:var:explicit}, but we
will see that many of the estimates are simpler due to the rescaled
radii.

\subsection{Proof of Lemma~\ref{lem:r-implicit:correlations}}
Similarly to the proof of Lemma~\ref{lem:r-explicit:correlations}, the
result follows from a simple application of
Lemmas~\ref{lem:construction} and \ref{lem:correlation}.
Again, we prove only the upper bound since the lower bound follows in a
similar manner.
Recall the constants $\frost, \thinan, \thinanr$ from
Conditions~\ref{cond:frostman} and \ref{cond:thin-annuli}.

\begin{proof}[Proof of Lemma~\ref{lem:r-implicit:correlations}]
  The hypothesis~\ref{item:r-implicit:correlations:3}, namely
  $\imM_k \to 0$, implies $\imr_k \to 0$ uniformly on $\supp\mu$ (see
  the discussion in \text{\S}\ref{sec:discussion:radius}).
  Thus, we may without loss of generality assume that
  $\imr_k \leq \thinanr$ for all $k$.
  Furthermore, $\lipconst{\imr_k} = 1$ for all $k$.
  For convenience, let $\frostan = \min\{\frost,\thinan\}$ and
  $d_0 = \dim X$.

  Let $\kappa > 0$ be a constant to be specified later, and let
  $\varepsilon_n = n^{-\kappa}$.
  We apply Lemma~\ref{lem:construction} to $\abr = \imr_k$ and
  $\varepsilon = \varepsilon_n$ to obtain H{\"o}lder continuous
  approximations of $\charfun_{\imE_k}$.
  Proceeding as in the proof of
  Lemma~\ref{lem:r-explicit:correlations}, we obtain the estimate
  \begin{multline*}
    \mu(\imE_k \cap \imE_j)
    \leq \int_X \mu\bigl(B(y,\imr_k(y))\bigr)
      \mu\bigl(B(y,\imr_j(y))\bigr) \diff\mu(y) \\
    + \bigo\bigl( n^{-\kappa} +
      n^{-\frostan\kappa} + n^{4(d_0+3)\kappa/\alpha}
      (e^{-\mdoc k} + e^{-\mdoc(j-k)})\bigr).
  \end{multline*}

  In contrast to the proof of Lemma~\ref{lem:construction}, we have
  have $\mu(B(y,\imr_k(y))) = \imM_k$ for all $y \in \supp\mu$.
  Therefore,
  \[
    \mu(\imE_k \cap \imE_j) \leq \imM_k \imM_j
    + \bigo\bigl( n^{-\kappa} + n^{-\frostan\kappa}
    + n^{4(d_0+3)\kappa/\alpha}
    (e^{-\mdoc k} + e^{-\mdoc(j-k)})\bigr).
  \]
  By Lemma~\ref{lem:r-implicit:meas},
  $|\mu(\imE_k)\mu(\imE_j) - \imM_k\imM_j| \leqs e^{-\limmeas k}$.
  Therefore,
  \begin{multline*}
    \mu(\imE_k \cap \imE_j) \leq \mu(\imE_k) \mu(\imE_j)
    + \bigo\bigl( n^{-\kappa} + n^{-\frostan\kappa}
    + e^{-\limmeas k} \\
    + n^{4(d_0+3)\kappa/\alpha} (e^{-\mdoc k} + e^{-\mdoc(j-k)})
    \bigr).
  \end{multline*}
  By choosing $\kappa > \max \{2, 2/\frostan\}$ and
  \[
    \limcorsum = \min \{\kappa, \frostan\kappa\} > 2
    \quad
    \limcorcoeff = \frac{4(d_0+3)\kappa}{\alpha},
    \quad
    \limcor = \min \{\mdoc, \limmeas\},
  \]
  we obtain the upper bound.
  The lower bound is established in a similar way.
\end{proof}

\subsection{Proofs of Lemmas~\ref{lem:r-implicit:liminf:sigma} and
\ref{lem:r-implicit:liminf:s}}
\label{sec:proof:r-implicit:liminf}
We will give a full proof of Lemma~\ref{lem:r-implicit:liminf:sigma}.
The proof of Lemma~\ref{lem:r-implicit:liminf:s} is almost identical,
except it uses a weaker decorrelation condition. Instead of repeating
the proof, we choose to highlight at which points we use the weaker
condition. We begin with the proof of
Lemma~\ref{lem:r-implicit:liminf:sigma}.
As in the proof of corresponding lemma,
Lemma~\ref{lem:r-explicit:liminf}, we will divide the proof into three
main steps. We shift time, then estimate the leading sum of
$\imsig_n^2$, and then estimate the double sum of $\imsig_n^2$.

\begin{proof}[Proof of Lemma~\ref{lem:r-implicit:liminf:sigma}]
  For convenience, we write $\frostan = \min \{\frost,\thinan\}$,
  $d_0 = \dim X$ and
  \[
    \imesum_n = \sum_{k=1}^{n} \mu(\imE_k).
  \]
  As in the proof of Lemma~\ref{lem:r-explicit:liminf}, we may assume
  that $\imr_k \leq \thinanr$ for all $k$.

  \step[shifting time]
  Let $D > 0$ be a constant to be specified later, and write
  $\beta_n = \log n$. Writing
  \[
    \imsig_{n,\beta_n}^2 = \int_X \biggl( \sum_{k=1}^{n}
    \tilde{\charfun}_{\imE_{k+\beta_n}} \biggr)^2 \diff\mu,
  \]
  we have
  \begin{equation}
  \label{eq:var:implicit:shift}
    | \imsig_{n,\beta_n}^2 - \imsig_n^2|
    \leqs \beta_n \imsig_{n,\beta_n} + \beta_n^2
  \end{equation}
  as in \eqref{eq:var:explicit:shift} in the proof of
  Lemma~\ref{lem:r-explicit:liminf}.
  It therefore suffices to show
  \begin{equation}
  \label{eq:var:r-implicit:shift:liminf}
    \liminf_{n \to \infty} \frac{\imsig_{n,\beta_n}^2}{\imesum_n}
    \geq 1
  \end{equation}
  for the main part of the lemma.

  Expanding the expression for $\imsig_{n,\beta_n}^2$, we have
  \begin{equation}
  \label{eq:lem:r-implicit:liminf:sigma:expand}
    \imsig_{n,\beta_n}^2
    = \sum_{k=1}^{n} \mu(\imE_{k+\beta_n}) - \mu(\imE_{k+\beta_n})^2
    + 2\sum_{k=1}^{n-1} \sum_{j=k+1}^{n}
    \covar( \charfun_{\imE_{k+\beta_n}}, \charfun_{\imE_{j+\beta_n}} ).
  \end{equation}
  We now estimate the two sums.

  \step[estimating the leading sum]
  By the hypothesis~\ref{item:r-implicit:liminf:sigma:3},
  $\imM_k \geqs k^{-\seqlow}$. Thus, by Lemma~\ref{lem:r-implicit:meas},
  $\mu(\imE_k) \geqs k^{-\seqlow}$ and $\imesum_n \geqs n^{1-\seqlow}$.
  Hence,
  \begin{equation}
  \label{eq:r-implicit:liminf:sigma:step:1}
    \sum_{k=1}^{n} \mu(\imE_{k+\beta_n}) - \mu(\imE_{k+\beta_n})^2
    = \imesum_n - \smallo(\imesum_n) + \bigo(\beta_n) = \imesum_n -
    \smallo(\imesum_n).
  \end{equation}

  \step[estimating the double sum]
  We split the double sum into two different estimates: one
  for $j \leq k + \beta_n$ and another for $j > k + \beta_n$.

  For $j \leq k + \beta_n$, we have
  \begin{equation}
  \label{eq:r-implicit:liminf:sigma:cov:1}
  \begin{split}
    \sum_{k=1}^{n-1}\sum_{j=k+1}^{k+\beta_n}
    \covar( \charfun_{\imE_{k+\beta_n}}, \charfun_{\imE_{j+\beta_n}} )
    &\geq - \sum_{k=1}^{n-1}\sum_{j=k+1}^{k+\beta_n}
    \mu(\imE_{k+\beta_n})\mu(\imE_{j+\beta_n}) \\
    &\geqs - \beta_n\sum_{k=1}^{n-1} \mu(\imE_{k+\beta_n})^2,
  \end{split}
  \end{equation}
  where we have used the fact that $(\mu(\imE_n))$ is monotone up to a
  constant coefficient. Indeed, $(\imM_k)$ is monotone, thus in
  combination with Lemma~\ref{lem:r-implicit:meas}, we have
  $\mu(\imE_j) \leqs \imM_j \leq \imM_k \leqs \mu(\imE_k)$
  for $k \leq j$. Now, let $s \in (0,1-\seqlow)$ and split the sum up
  into $k < n^{s}$ and $k \geq n^{s}$ as follows:
  \[
    \beta_n \sum_{k=1}^{n-1} \mu(\imE_{k+\beta_n})^2
    \leqs \beta_n \sum_{k=1}^{n-1} \mu(\imE_k)^2
    \leqs \beta_n n^{s} + \sum_{k=n^{s}}^{n-1} \mu(\imE_k)^2 \log k.
  \]
  Note that we used the monotonicity
  up to a multiplicative constant of $\mu(\imE_k)$,
  and $\beta_n \eqs \log n \leqs \log k$ whenever $k \geq n^{s}$.
  Since $\mu(\imE_k) \geqs n^{1-\seqlow}$
  and $\mu(\imE_k) \log k = \smallo(1)$
  by Lemma~\ref{lem:r-implicit:meas} and the
  hypothesis~\ref{item:r-explicit:liminf:3}, we have
  \[
    \beta_n \sum_{k=1}^{n-1} \mu(\imE_{k+\beta_n})^2
    = \smallo(\imesum_n).
  \]
  In view of \eqref{eq:r-implicit:liminf:sigma:cov:1}, we obtain
  \[
    \sum_{k=1}^{n-1}\sum_{j=k+1}^{k+\beta_n}
    \covar( \charfun_{\imE_{k+\beta_n}}, \charfun_{\imE_{j+\beta_n}} )
    \geq \smallo(\imesum_n).
  \]

  For $j > k + \beta_n$, we use
  Lemma~\ref{lem:r-implicit:correlations}. That is, there exists
  $\limcorsum > 2$ and $\limcorcoeff, \limcor  > 0$ such that
  \[
  \begin{split}
    \sum_{k=1}^{n-1}\sum_{j=k+\beta_n+1}^{n}
      \lvert \covar( & \charfun_{\imE_{k+\beta_n}},
      \charfun_{\imE_{j+\beta_n}}) \rvert \\
    &\leqs \sum_{k=1}^{n-1}\sum_{j=k+\beta_n+1}^{n}
      \bigl(n^{-\limcorsum} + n^{\limcorcoeff}
      (e^{-\limcor (k + \beta_n)} + e^{-\limcor (j-k)})\bigr) \\
    &\leqs n^{2-\limcorsum} + n^{2+\limcorcoeff} e^{-\limcor \beta_n}\\
    &= \smallo(\imesum_n).
  \end{split}
  \]
  Choosing $D$ in $\beta_n = D \log n$ sufficiently large, we have
  \begin{equation}
  \label{eq:lem:r-implicit:liminf:sigma:cov:2}
    \sum_{k=1}^{n-1}\sum_{j=k+\beta_n+1}^{n} \lvert \covar(
    \charfun_{\imE_{k+\beta_n}}, \charfun_{\imE_{j+\beta_n}}) \rvert
    = \smallo(\imesum_n).
  \end{equation}
  By \eqref{eq:r-implicit:liminf:sigma:cov:1} and
  \eqref{eq:lem:r-implicit:liminf:sigma:cov:2}, we have
  \begin{equation}
  \label{eq:lem:r-implicit:liminf:sigma:cov}
    \sum_{k=1}^{n-1}\sum_{j=k+1}^{n} \covar(
    \charfun_{\imE_{k+\beta_n}}, \charfun_{\imE_{j+\beta_n}})
    \geq \smallo(\imesum_n).
  \end{equation}

  \conclusion
  In view of \eqref{eq:lem:r-implicit:liminf:sigma:expand},
  \eqref{eq:r-implicit:liminf:sigma:step:1} and
  \eqref{eq:lem:r-implicit:liminf:sigma:cov}, we obtain
  \eqref{eq:var:r-implicit:shift:liminf}, which completes the proof of
  the main part of the lemma. For the second part, notice that
  $\mu(\imE_k) \geqs \imM_k \geqs n^{-\seqlow}$ by
  Lemma~\ref{lem:r-implicit:meas}.
  Thus, $\imsig_n^2 \geqs \imesum_n \geqs n^{1-\seqlow}$.
\end{proof}

We now prove Lemma~\ref{lem:r-implicit:liminf:s}, highlighting where we
use Condition~\ref{cond:doc} instead of \ref{cond:multiple-doc}.
Recall that
\[
  \ims_n^2(y) = \var \biggl( \sum_{k=1}^{n} \charfun_{B(y,\imr_k(y))}
  \circ T^{k} \biggr)
  \quad \text{and} \quad
  \immsum_n = \sum_{k=1}^{n} \imM_k.
\]

\begin{proof}[Proof of Lemma~\ref{lem:r-implicit:liminf:s}]
  By Lemma~\ref{lem:wlog}, we may assume without loss of generality
  that $\imr_k \leq \thinanr$ for all $k$.
  By expanding the expression for $\ims_n^2(y)$, we have
  \begin{multline*}
    \ims_n^2(y)
    = \sum_{k=1}^{n} \imM_k - \imM_k^2 \\
    + 2\sum_{k=1}^{n-1} \sum_{j=k+1}^{n} \mu\bigl( B(y,\imr_k(y)) \cap
    T^{-(j-k)}B(y,\imr_j(y)) \bigr) - \imM_k \imM_j.
  \end{multline*}
  We see directly that the leading sum can be estimated by
  $\immsum_n - \smallo(\immsum_n)$.
  The double sum is treated in a similar way as in the proof of
  Lemma~\ref{lem:r-implicit:liminf:sigma}, but notice that the centre
  of the balls do not depend on the integration variable. Thus, it
  suffices to use Condition~\ref{cond:doc} instead of
  \ref{cond:multiple-doc}. The estimates also do not depend on $y$.
  
  In conclusion, we obtain
  \[
    \ims_n^2(y) \geq \immsum_n - \smallo(\immsum_n).
  \]
  Since $\imM_n \geqs n^{-\seqlow}$, we obtain
  $\ims_n^2(y) \geqs n^{1-\seqlow}$, which concludes the proof.
\end{proof}

\subsection{Proof of Proposition~\ref{prop:r-implicit:mean}}
The main idea of the proof is to use the quasi-partition of
Lemma~\ref{lem:partition} and decay of correlations. We only prove the
conclusion with the limit replaced by the limit inferior, since the
corresponding limit supremum follows in a similar manner.

\begin{proof}[Proof of Proposition~\ref{prop:r-implicit:mean}]
  By Lemma~\ref{lem:wlog}, we may without loss of generality assume
  $\imr_k \leq \thinanr$ for all $k$.
  For convenience, we write $\frostan = \min \{\frost,\thinan\}$
  and $d_0 = \dim X$.

  We only show
  \begin{equation}
  \label{eq:lem:var:implicit:mean:liminf}
    \liminf_{n \to \infty} \int_X \frac{\ims_n^2(y)}{\imsig_n^2}
    \diff\mu(y) \geq 1,
  \end{equation}
  as the corresponding expression for the limit superior can be shown
  in a similar way.

  Similar to previous proofs of variance estimates, we introduce a
  logarithmic time-shift. Let $D > 0$ be a constant to be specified
  later, and write $\beta_n = D \log n$. Write
  \[
    \imsig_{n,\beta_n}^2 = \var \biggl( \sum_{k=1+\beta_n}^{n+\beta_n}
    \tilde{\charfun}_{\imE_k} \biggr)
    \quad \text{and} \quad
    \ims_{n,\beta_n}^2(y) = \var \biggl( \sum_{k=1+\beta_n}^{n+\beta_n}
    \tilde{\charfun}_{B(y,\imr_k(y))} \circ T^{k} \biggr)
  \]
  As in the proof of Lemma~\ref{lem:r-implicit:liminf:sigma} (see
  \eqref{eq:var:implicit:shift}), we have
  \begin{equation}
  \label{eq:var:implicit:mean:shift}
    |\imsig_n^2 - \imsig_{n,\beta_n}^2|
    \leqs \beta_n \imsig_n + \beta_n^2,
  \end{equation}
  and similarly,
  \begin{equation}
  \label{eq:var:implicit:mean:shift:s}
    |\ims_n^2(y) - \ims_{n,\beta_n}^2(y)|
    \leqs \beta_n \ims_n(y) + \beta_n^2
    \quad \text{uniformly in } y \in \supp\mu.
  \end{equation}
  As it will be used later, we expand the expressions
  $\imsig_{n,\beta_n}^2$ and $\ims_{n,\beta_n}^2$ in the following way:
  \begin{align}
  \label{eq:var:implicit:mean:shift:sigma}
    \imsig_{n,\beta_n}^2
    &= \int_X \biggl(\sum_{k=1}^{n} \charfun_{\imE_{k+\beta_n}}
      \biggr)^2 \diff\mu
      - \biggl( \sum_{k=1}^{n} \mu(\imE_{k+\beta_n}) \biggr)^2 \\
  \label{eq:var:implicit:mean:shift:s:expand}
    \ims_{n,\beta_n}^2(y) &=
    \int_X \biggl(\sum_{k=1}^{n} \charfun_{B(y,\imr_k(y))} \circ T^{k}
      \biggr)^2 \diff\mu
      - \biggl( \sum_{k=1}^{n} \imM_{k+\beta_n} \biggr)^2
  \end{align}

  We will now use the quasi-partition from Lemma~\ref{lem:partition} in
  order to relate $\imsig_{n,\beta_n}^2$ and the average of
  $\ims_{n,\beta_n}^2$ in a series of estimates.
  Let $\kappa > 0$ be a constant to be specified later, and write
  $\varepsilon_n = n^{-\kappa}$.
  Apply Lemma~\ref{lem:partition} to $\kappa_0 = 2\kappa$ to obtain the
  set $Y_n$, the partition $\partition_n$ of $X \setminus Y_n$, the
  densities $\{\rho_{Q,n}\}_{Q \in \partition_n}$ and the corresponding
  estimates.
  For each $Q \in \partition_n$, we may select a point
  $x_Q \in Q \cap \supp\mu$ since $Q$ has positive measure.
  This ensures that $\mu(B(x_Q,\imr_k(x_Q))) = M_k$.

  Apply Lemma~\ref{lem:construction} to $\varepsilon = \varepsilon_n$
  and $\abr = \imr_n$, and obtain Lipschitz continuous functions
  $g_{n,k} \colon X \times X \to [0,1]$ such that
  \begin{equation}
  \label{eq:1}
    \lipnorm{g_{n,k}} \leqs \varepsilon_n^{-1} \leqs n^{\kappa},
  \end{equation}
  \begin{equation}
  \label{eq:2}
    \charfun_{B(x,\imr_n(x))}(y) \leq g_{n,k}(x,y) \leq 
    \charfun_{B(x,\imr_n(x) + \varepsilon_n)}(y)
  \end{equation}
  and
  \begin{equation}
  \label{eq:3}
    \int_X \bigl\lvert g_{n,k}(x,y) - M_k \bigr\rvert \diff\mu(y)
    \leqs \varepsilon_n^{\thinan} \leqs n^{-\kappa \frostan}.
  \end{equation}
  In particular, \eqref{eq:2} implies
  $\charfun_{\imE_k}(x) \leq g_{n,k}(x,T^{k}x)$, and \eqref{eq:1}
  implies, for each $x \in Q$,
  \begin{equation}
  \label{eq:var:implicit:mean:approx:1}
    \biggl| \biggl( \sum_{k=1+\beta_n}^{n+\beta_n} g_{n,k}(x,T^{k}x)
      \biggr)^2
    - \biggl( \sum_{k=1+\beta_n}^{n+\beta_n} g_{n,k}(x_Q,T^{k}x)
      \biggr)^2 \biggr|
    \leqs n^2 \cdot n^{\kappa} \diam{\partition_n}
    = \smallo(1).
  \end{equation}
  where the last inequality holds if we select $\kappa > 2$ since
  $|\partition_n| \leqs n^{-2\kappa}$.
  Hence,
  \begin{multline*}
    \int_X \biggl( \sum_{k=1}^{n} \charfun_{\imE_{k+\beta_n}} \biggr)^2
      \diff\mu
    \leq \sum_{Q \in \partition_n} \int_Q
    \biggl(
      \sum_{k=1}^{n} g_{n,k+\beta_n}(x_Q,T^{k+\beta_n}x) \biggr)^2
      \diff\mu(x) \\
      + \smallo(1) + \bigo(n^2 \mu(Y_n)).
  \end{multline*}
  Now, $n^2 \mu(Y_n) \leqs n^{2(1-\kappa)} = \smallo(1)$ since
  $\kappa > 2$; thus, we may replace $\bigo(n^2\mu(Y_n))$ above by
  $\smallo(1)$.

  We now use the densities $\{\rho_{Q,n}\}$, which
  satisfy
  $\lnorm{\rho_{Q,n} - \mu(Q)^{-1} \charfun_Q}{1} \leqs n^{-2\kappa}$.
  Hence,
  \begin{multline}
  \label{eq:var:implicit:mean:approx:3}
    \int_X \biggl( \sum_{k=1}^{n} \charfun_{\imE_{k+\beta_n}} \biggr)^2
      \diff\mu \\
    \leq \sum_{Q \in \partition_n} \mu(Q) \int_X \rho_{Q,n}(x) \biggl(
      \sum_{k=1}^{n} g_{n,k+\beta_n}(x_Q,T^{k+\beta_n}x)
      \biggr)^2 \diff\mu(x) 
    + \smallo(1),
  \end{multline}
  where the error $\bigo(n^{2(1-\kappa)})$ arising from the
  approximation by the densities $\{\rho_{Q,n}\}$ can be replaced by
  $\smallo(1)$, again by selecting $\kappa > 2$. Now,
  \begin{multline*}
    \int_X \rho_{Q,n}(x) \biggl( \sum_{k=1}^{n}
      g_{n,k+\beta_n}(x_Q,T^{k+\beta_n}x) \biggr)^2 \diff\mu(x) \\
    = \sum_{k,j} \int_X \rho_{Q,n}(x)
      g_{n,k+\beta_n}(x_Q,T^{k+\beta_n}x)
      g_{n,j+\beta_n}(x_Q,T^{j+\beta_n}x) \diff\mu(x) 
  \end{multline*}
  We apply Condition~\ref{cond:multiple-doc} with observables
  $\varphi_1 = \rho_{Q,n}$, $\varphi_2 = g_{n,k+\beta_n}(x_Q, \cdot)$
  and $\varphi_3 = g_{n,j+\beta_n}(x_Q, \cdot)$ when $k < j$, and
  $\varphi_1 = \rho_{Q,n}$ and
  $\varphi_2 = g_{n,k+\beta_n}(x_Q, \cdot)^2$ when $k=j$.
  Since $\mu(\rho_{Q,n}) = 1 + \bigo(n^{-2\kappa})$
  and $\holnorm{\rho_{Q,n}}{\alpha} \leqs n^{4(d+2)\kappa}$, we have
  \begin{multline*}
    \int_X \rho_{Q,n}(x) \biggl( \sum_{k=1}^{n}
      g_{n,k+\beta_n}(x_Q,T^{k+\beta_n}x) \biggr)^2 \diff\mu(x) \\
    = \int_X \biggl( \sum_{k=1}^{n}
      g_{n,k+\beta_n}(x_Q,T^{k+\beta_n}x) \biggr)^2 \diff\mu(x) \\
      + \bigo\bigl(n^{2(1-\kappa)}
      + n^2 n^{4(d+2)\kappa} n^{\kappa} P_r(n)
      e^{-\mdoc \beta_n}\bigr).
  \end{multline*}
  Again, selecting $\kappa > 2$ and then $D$ sufficiently large in
  $\beta_n = D \log n$ allows us to replace the error term by
  $\smallo(1)$.

  In view of \eqref{eq:var:implicit:mean:approx:3}, we therefore have
  \begin{multline}
  \label{eq:var:implicit:mean:approx:2}
    \int_X \biggl( \sum_{k=1}^{n} \charfun_{\imE_{k+\beta_n}} \biggr)^2
      \diff\mu \\
    \leq \sum_{Q \in \partition_n} \mu(Q) \int_X \biggl( \sum_{k=1}^{n}
      g_{n,k+\beta_n}(x_Q,T^{k}x) \biggr)^2 \diff\mu(x) + \smallo(1).
  \end{multline}
  Similarly to how it was used to show
  \eqref{eq:var:implicit:mean:approx:1}, we use the Lipschitz continuity
  of $g_{n,k}$ to obtain
  \[
    \biggl| \mu(Q) \biggl( \sum_{k=1}^{n} g_{n,k+\beta_n}(x_Q,T^{k}x)
    \biggr)^2 - \int_Q \biggl( \sum_{k=1}^{n} g_{n,k+\beta_n}(y,T^{k}x)
    \biggr)^2 \diff\mu(y) \biggr|
    = \smallo(1).
  \]
  By \eqref{eq:var:implicit:mean:approx:2}, we therefore have
  \begin{multline*}
    \int_X \biggl( \sum_{k=1}^{n} \charfun_{\imE_{k+\beta_n}} \biggr)^2
      \diff\mu \\
    \leq \sum_{Q \in \partition_n} \int_Q \int_X \biggl(
      \sum_{k=1}^{n} g_{n,k+\beta_n}(y,T^{k}x) \biggr)^2
      \diff\mu(x)\diff\mu(y)
    + \smallo(1).
  \end{multline*}

  Again, using that $\partition_n$ is a partition of $X \setminus Y_n$
  with $n^2 \mu(Y_n) \leqs n^{2(1-\kappa)} = \smallo(1)$, we have
  \[
    \int_X \biggl( \sum_{k=1}^{n} \charfun_{\imE_{k+\beta_n}} \biggr)^2
      \diff\mu
    \leq \iint \biggl( \sum_{k=1}^{n}
      g_{n,k+\beta_n}(y,T^{k}x) \biggr)^2 \diff\mu(x)\diff\mu(y)
    + \smallo(1).
  \]
  By \eqref{eq:2} and \eqref{eq:3}, we have
  \begin{multline*}
    \biggl| \iint \biggl( \sum_{k=1}^{n}
      g_{n,k+\beta_n}(y, T^{k}x) \biggr)^2 \diff\mu(x) \diff\mu(y)\\
    - \iint \biggl( \sum_{k=1}^{n}
      \charfun_{B(y,\imr_{k+\beta_n}(y))} (T^{k}x) \biggr)^2
      \diff\mu(x) \diff\mu(y) \biggr| \\
    \leqs n^{2-\kappa\frostan} = \smallo(1),
  \end{multline*}
  where the last inequality holds by choosing $\kappa > 2/\frostan$.
  By the previous two equations, we therefore have
  \begin{equation}
  \label{eq:var:implicit:mean:approx:4}
    \int \biggl(\sum_{k=1}^{n} \charfun_{\imE_{k+\beta_n}} \biggr)^2
      \diff\mu
    - \iint \biggl( \sum_{k=1}^{n}
      \charfun_{B(y,\imr_{k+\beta_n}(y))}(T^{k}x) \biggr)^2 \diff\mu(x)
      \diff\mu(y)
    \leq \smallo(1).
  \end{equation}

  At the same time, by Lemma~\ref{lem:r-implicit:meas}, we have
  \[
  \begin{split}
    \biggl| \biggl( \sum_{k=1}^{n} \mu(\imE_{k+\beta_n}) \biggr)^2
      - \biggl( \sum_{k=1}^{n} \imM_{k+\beta_n} \biggr)^2 \biggr|
    &\leqs n \sum_{k=1}^{n} \bigl| \mu(\imE_{k+\beta_n})
      - \imM_{k+\beta_n} \bigr| \\
    &\leqs n e^{-\limmeas \beta_n} \\
    &= \smallo(1)
  \end{split}
  \]
  by choosing $D$ in $\beta_n = D \log n$ sufficiently large. Recalling
  the expression for $\imsig_{n,\beta_n}^2$, namely
  \eqref{eq:var:implicit:mean:shift:sigma}, the previous estimate and
  \eqref{eq:var:implicit:mean:approx:4} imply
  \[
    \imsig_{n,\beta_n}^2
    \leq \iint \biggl( \sum_{k=1}^{n}
      \charfun_{B(y,\imr_{k+\beta_n}(y))}(T^{k}x) \biggr)^2 \diff\mu(x)
      \diff\mu(y)
      - \biggl( \sum_{k=1}^{n} \imM_{k+\beta_n} \biggr)^2
      + \smallo(1).
  \]
  Recalling the expression for $\ims_{n,\beta_n}^2(y)$, namely
  \eqref{eq:var:implicit:mean:shift:s:expand}, we have thus shown
  \[
    \imsig_{n,\beta_n}^2 \leq \int_X \ims_{n,\beta_n}^2(y) \diff\mu(y)
    + \smallo(1).
  \]

  In view of \eqref{eq:var:implicit:mean:shift} and
  \eqref{eq:var:implicit:mean:shift:s}, we have therefore established
  \[
    \imsig_n^2 - \int_X \ims_n^2(y) \diff\mu(y) \leqs \beta_n
    \int_X \ims_n(y) \diff\mu(y) + \beta_n \imsig_n + \beta_n^2.
  \]
  Dividing by $\imsig_n^2$ and rearranging, we have
  \begin{equation}
  \label{eq:variance:ineq:5}
    1 \leq \int_X \frac{\ims_n^2}{\imsig_n^2} \diff\mu
    \biggl(1 + \bigo\biggl( \frac{\log n \int \ims_n \diff\mu}{\int
    \ims_n^2 \diff\mu} \biggr)\biggr)
    + \bigo\biggl(\frac{\log n}{\imsig_n}
    + \frac{(\log n)^2}{\imsig_n^2}\biggr),
  \end{equation}
  where we have used that $\beta_n = D \log n$.
  By Lemma~\ref{lem:r-implicit:liminf:sigma}
  $\imsig_n^{-1} \log n \to 0$.
  By Lemma~\ref{lem:r-implicit:liminf:s},
  $\ims_n^2(y) \geqs n^{1-\seqlow}$ uniformly in $y \in \supp\mu$.
  Thus, by Jensen's inequality,
  \[
    \limsup_{n \to \infty} \frac{\log n \int \ims_n(y)
      \diff\mu(y)}{\int \ims_n^2(y) \diff\mu(y)}
    \leqs \limsup_{n \to \infty} \frac{\log n}{\int
      \ims_n(y)\diff\mu(y)}
    \leqs \limsup_{n \to \infty} \frac{\log n}{n^{(1-\seqlow)/2}} = 0.
  \]
  Therefore, by taking $\liminf_{n \to \infty}$ on both sides in
  \eqref{eq:variance:ineq:5}, we obtain
  \eqref{eq:lem:var:implicit:mean:liminf}.

  The corresponding expression for the limit superior is established by
  switching out $g_{n,k}$ for functions approximating
  $\charfun_{\imE_k}$ from below.
  This concludes the proof.
\end{proof}

\subsection{Proof of Proposition~\ref{prop:r-implicit:meas}}
The proof consists of applying Chebyshev's inequality and then
estimating $\var(\ims_n^2)$. Recall that
\[
  \immsum_n = \sum_{k=1}^{n} \imM_k.
\]

\begin{proof}[Proof of Proposition~\ref{prop:r-implicit:meas}]
  By Lemma~\ref{lem:wlog}, we may without loss of generality assume
  that $\imr_k \leq \thinanr$ for all $k$. Let
  $\frostan = \min \{\frost,\thinan\}$ and $d_0 = \dim X$.
  Let $D > 0$ be a constant to be specified later, and write
  $\beta_n = D \log n$.

  Fix $\varepsilon > 0$. We wish to show
  \[
    \lim_{n \to \infty}
    \mu \biggl\{ y \in X : \Bigl \lvert \frac{\ims_n^2(y)}{\imsig_n^2}
    - 1 \Bigr \rvert > \varepsilon \biggr\}
    = 0,
  \]
  which, by Proposition~\ref{prop:r-implicit:mean}, is equivalent to
  \[
    \lim_{n \to \infty} \mu \biggl\{y \in X : \biggl|
    \frac{\ims_n^2(y)}{\imsig_n^2}
    - \int_X \frac{\ims_n^2}{\imsig_n^2} \diff\mu \biggr|
    > \varepsilon \biggr\} = 0.
  \]
  Furthermore, by Chebyshev's inequality, we have
  \[
    \mu \biggl\{y \in X: \biggl| \frac{\ims_n^2(y)}{\imsig_n^2}
    - \int_X \frac{\ims_n^2}{\imsig_n^2} \diff\mu \biggr|
    > \varepsilon \biggr\}
    \leq \frac{\var(\ims_n^2)}{\varepsilon^2 \imsig_n^{4}}.
  \]
  It therefore suffices to show
  $\var(\ims_n^2) = \smallo(\imsig_n^{4})$, which we devote the rest of
  the proof to.

  Expanding the expression for $\ims_n^2(y)$, we have
  \begin{multline*}
    \ims_n^2(y) = \sum_{k=1}^{n} \imM_k - \imM_k^2 \\
    + \sum_{k=1}^{n-1} \sum_{j=k+1}^{n} \mu\bigl( B(y,\imr_k(y))
    \cap T^{-(j-k)} B(y,\imr_j(y)) \bigr) - \imM_k \imM_j.
  \end{multline*}
  We see directly that the leading sum satisfies
  $\sum_{k=1}^{n} \imM_k - \imM_k^2 = \immsum_n - \smallo(\immsum_n)$ 
  since $\sum_{k=1}^{\infty} \imM_k = \infty$ and $\imM_k \to 0$
  by the hypothesis~\ref{item:r-implicit:meas:4}.
  Furthermore, by the proof of Lemma~\ref{lem:r-implicit:liminf:s}, we
  have
  \[
    \sum_{k=1}^{n-1}\sum_{j=k+\beta_n+1}^{n} \mu\bigl( B(y,\imr_k(y))
    \cap T^{-(j-k)} B(y,\imr_j(y)) \bigr) - \imM_k \imM_j
    = \smallo(1)
  \]
  uniformly in $y \in \supp\mu$, provided we choose $D$ in
  $\beta_n = D \log n$ sufficiently large. We also have
  \[
    \sum_{k=1}^{n-1} \sum_{j=k+1}^{k+\beta_n} \imM_k \imM_j
    = \smallo(\immsum_n).
  \]
  Thus,
  \[
    \ims_n^2(y) = \immsum_n + \smallo(\immsum_n)
    + \sum_{k=1}^{n-1} \sum_{j=k+1}^{k+\beta_n} \mu\bigl( B(y,\imr_k(y))
    \cap T^{-(j-k)} B(y,\imr_j(y)) \bigr)
  \]
  uniformly in $y \in \supp\mu$.

  We now estimate
  $\var(\ims_n^2)
  = \int \ims_n^{4} \diff\mu - (\int \ims_n^2 \diff\mu)^2$.
  Let
  \[
    A_{k,j}(y) \coleq
    \mu\bigl( B(y,\imr_k(y)) \cap T^{-(j-k)} B(y,\imr_j(y)) \bigr).
  \]
  Then
  \begin{multline*}
    \ims_n^{4}(y) = \immsum_n^2
    + \smallo( \immsum_n^2 )
    + (2 \immsum_n + \smallo(\immsum_n))
    \sum_{k=1}^{n-1} \sum_{j=k+1}^{k+\beta_n} A_{k,j}(y) \\
    + \sum_{k,l=1}^{n-1} \sum_{j=k+1}^{k+\beta_n}
    \sum_{i=l+1}^{l+\beta_n} A_{k,j}(y) A_{l,i}(y)
  \end{multline*}
  and
  \begin{multline*}
    \biggl( \int_X \ims_n^{2} \diff\mu \biggr)^2
    = \immsum_n^2 + \smallo( \immsum_n^2 )
      + (2 \immsum_n + \smallo(\immsum_n)) \sum_{k=1}^{n-1}
      \sum_{j=k+1}^{k+\beta_n} \int_X A_{k,j}(y) \diff\mu(y) \\
    + \sum_{k,l=1}^{n-1} \sum_{j=k+1}^{k+\beta_n}
      \sum_{i=l+1}^{l+\beta_n} \int_X A_{k,j}(y) \diff\mu(y)
      \int_X A_{l,i}(z) \diff\mu(z).
  \end{multline*}
  Hence,
  \begin{multline}
  \label{eq:r-implicit:meas:var:s}
    \var(\ims_n^2)
    = \sum_{k,l,j,i} \int_X A_{k,j}(y) A_{l,i}(y) \diff\mu(y)
      - \int_X A_{k,j}(y) \diff\mu(y) \int_X A_{l,i}(z) \diff\mu(z) \\
    + \smallo(\immsum_n) \sum_{k=1}^{n-1} \sum_{j=k+1}^{k+\beta_n}
      \int_X A_{k,j}(y) \diff\mu(y) + \smallo( \immsum_n^2 ),
  \end{multline}
  where the first sum is taken over all $1 \leq k,l \leq n-1$, and all
  $k < j \leq k + \beta_n$ and $l < i \leq l + \beta_n$.
  Using the trivial estimate
  $A_{l,i}(y) \leq \mu( B(y,\imr_l(y)) ) = \imM_l$
  for all $l < i$, we have
  \[
  \begin{split}
    \sum_{l=1}^{n-1} \sum_{i=l+1}^{l+\beta_n}
      \int_X A_{k,j}(y) A_{l,i}(y) \diff\mu(y)
    &\leqs \beta_n \sum_{l=1}^{n-1} \imM_l \int_X A_{k,j}(y)
      \diff\mu(y) \\
    &= \beta_n \immsum_n \int_X A_{k,j}(y) \diff\mu(y).
  \end{split}
  \]
  Thus, by also ignoring the negative terms in
  \eqref{eq:r-implicit:meas:var:s}, we have
  \[
    \var(\ims_n^2) \leqs \beta_n \immsum_n \sum_{k=1}^{n-1}
    \sum_{j=k+1}^{k+\beta_n} \int_X A_{k,j}(y) \diff\mu(y)
    + \smallo( \immsum_n^2 ).
  \]

  We now show that the right-hand side can be bounded by
  $\smallo(\immsum_n^2)$. Let $s \in (0,1-\seqlow)$. Then
  \[
    \beta_n \immsum_n \sum_{k=1}^{n^{s}}
      \sum_{j=k+1}^{k+\beta_n} \int_X A_{k,j}(y) \diff\mu(y)
    \leqs n^{s} (\log n)^2 \immsum_n
    = \smallo(\immsum_n^2)
  \]
  since $\immsum_n \geqs n^{1-\seqlow}$ by the
  hypothesis~\ref{item:r-implicit:meas:4}. For $k > n^{s}$, we use
  Condition~\ref{cond:sr:implicit}, which implies
  (see \eqref{eq:var:explicit:l1:1:sr})
  \[
    \int_X A_{k,j}(y) \diff\mu(y) \leqs \imM_k^{1+\srimin}
  \]
  for all $0 < j-k \leq (\log k)^{\srilog}$, where
  $\srimin = \min \{\sriin, \sriout\}$ and $\srilog > 1$.
  Now, $k > n^{s}$ implies $\beta_n \eqs \log n \leqs \log k$. Thus,
  since $\log k \leq (\log k)^{\srilog}$, we have
  \[
    \beta_n \immsum_n \sum_{k=n^{s}+1}^{n-1}
    \sum_{j=k+1}^{k+\beta_n} \int_X A_{k,j}(y) \diff\mu(y)
    \leqs \immsum_n \sum_{k=1}^{n-1} \imM_k^{1+\srimin} (\log k)^2
    = \smallo(\immsum_n^2)
  \]
  as $\imM_k^{\srimin} \leqs (\log k)^{-\seqhigh}$ for some
  $\seqhigh > 2$ by the hypothesis~\ref{item:r-implicit:meas:4}.

  We have therefore shown
  \[
    \var(\ims_n^2) = \smallo(\immsum_n^2).
  \]
  Finally, $\imsig_n^2 \geqs \imesum_n$ by
  Lemma~\ref{lem:r-implicit:liminf:sigma}, and
  $\imesum_n \sim \immsum_n$ by
  Lemma~\ref{lem:r-implicit:meas}. Hence,
  \[
    \var(\ims_n^2) = \smallo(\imsig_n^{4}),
  \]
  which concludes the proof.
\end{proof}

\section{Proof of Lemma~\ref{lem:correlation}}
\label{sec:proof:supporting}
The proof utilises the quasi-partition from Lemma~\ref{lem:partition}
and decay of correlations.

\begin{proof}
  Fix $\alpha \in (0,1]$ and $p > 0$, and write $d_0 = \dim X$.
  Apply Lemma~\ref{lem:partition} to $\kappa_0 = 2p/\alpha$ to obtain
  the set $Y_n$, the partition $\partition_n$ of $X \setminus Y_n$, the
  densities $\{\rho_{Q,n}\}_{Q \in \partition_n}$ and the corresponding
  estimates. For each $Q \in \partition_n$, we may select a point
  $y_Q \in Q \cap \supp\mu$ since $Q$ has positive measure.

  \step[proof of part~\ref{item:lem:correlation:1}]
  \label{step:lem:correlation:1}
  Using the quasi-partition, we have, for all $\alpha$-H{\"o}lder
  continuous functions $g \colon X \times X \to [0,1]$ with
  $\holconst{g}{\alpha} \leq C_1 n^{p}$,
  \begin{equation}
  \label{eq:lemma:correlation:holder}
    | g(y_Q,x) - g(y,x) | \leq C_1 n^{p} \diam{\partition_n}^{\alpha}
    \leqs n^{-p}
    \quad \text{for all } x \in X, y \in Q.
  \end{equation}
  Furthermore, since
  $\mu(Y_n) \leqs n^{-2p/\alpha} \leq n^{-p}$, we have
  \[
  \begin{split}
    \int_X g(x,T^{k}x) \diff\mu(x)
    &= \sum_{Q \in \partition_n} \int_X \charfun_Q(x) g(x,T^{k}x)
      \diff\mu(x) + \bigo(n^{-p}) \\
    &= \sum_{Q \in \partition_n} 
      \int_X \charfun_Q(x) g(y_Q,T^{k}x) \diff\mu(x) + \bigo(n^{-p})
  \end{split}
  \]
  Using the densities $\rho_{Q,n}$, which satisfy
  $\lnorm{\rho_{Q,n}-\mu(Q)^{-1} \charfun_Q}{1} \leqs n^{-2p/\alpha}
  \leqs n^{-p}$,
  we have
  \[
    \int_X \mu(Q)^{-1} \charfun_Q(x) g(y_Q,T^{k}x) \diff\mu(x)
    = \int_X \rho_{Q,n}(x) g(y_Q,T^{k}x) \diff\mu(x) \\
    + \bigo(n^{-p}).
  \]
  Hence,
  \begin{equation}
  \label{eq:lemma:correlation}
    \int_X g(x,T^{k}x) \diff\mu(x) = \sum_{Q \in \partition_n} \mu(Q)
    \int_X \rho_{Q,n}(x) g(y_Q,T^{k}x) \diff\mu(x) + \bigo(n^{-p}).
  \end{equation}

  Using Condition~\ref{cond:doc},
  $\lipnorm{\rho_{Q,n}} \leqs n^{4(d_0+2)p/\alpha}$
  and $\lnorm{\rho_{Q,n}}{1} = 1 + \bigo(n^{-p})$, we have
  \[
  \begin{split}
    \int_X \rho_{Q,n}(x) g( & y_Q,T^{k}x) \diff\mu(x) \\
    &= \int_X \rho_{Q,n} \diff\mu \int_X g(y_Q,x) \diff\mu(x)
    + \bigo(\holnorm{\rho_{Q,n}}{\alpha} \holnorm{g}{\alpha}
    e^{-\doc k}) \\
    &= \int_X g(y_Q,x) \diff\mu(x)
      + \bigo(n^{-p} + n^{4(d_0+2)p/\alpha} \cdot n^{p} e^{-\doc k}).
  \end{split}
  \]
  In view of \eqref{eq:lemma:correlation}, we therefore have
  \begin{multline*}
    \int_X g(x,T^{k}x) \diff\mu(x) = \sum_{Q \in \partition_n} \mu(Q)
    \int_X g(y_Q,x) \diff\mu(x) \\
      + \bigo(n^{-p} + n^{4(d_0+3)p/\alpha} e^{-\doc k}).
  \end{multline*}
  Using \eqref{eq:lemma:correlation:holder} again, we obtain that
  \[
    \Bigl| \mu(Q) \int_X g(y_Q,x) \diff\mu(x)
    - \int_Q \int_X g(y,x) \diff\mu(x) \diff\mu(y) \Bigr|
    \leqs n^{-p}
  \]
  Therefore,
  \begin{multline*}
    \int_X g(x,T^{k}x) \diff\mu(x) =
    \int_{X \setminus Y_n} \int_X g(y,x) \diff\mu(x) \diff\mu(y)  \\
      + \bigo(n^{-p} + n^{4(d_0+3)p/\alpha} e^{-\doc k}).
  \end{multline*}
  Using $\mu(Y_n) \leqs n^{-p}$, we therefore have
  \begin{multline*}
    \int_X g(x,T^{k}x) \diff\mu(x)
    = \iint g(y,x) \diff\mu(x) \diff\mu(y) \\
      + \bigo(n^{-p} + n^{4(d_0+3)p/\alpha} e^{-\doc k}).
  \end{multline*}
  Since none of the suppressed constants depends on $g$,
  we have thus established part~\ref{item:lem:correlation:1}.

  \step[proof of part~\ref{item:lem:correlation:2}]
  We now assume that Condition~\ref{cond:multiple-doc} holds.
  In the same way \eqref{eq:lemma:correlation} was established, we
  obtain
  \begin{multline}
  \label{eq:lem:correlation:2}
    \int_X g(x,T^{k}x) h(x,T^{j}x) \diff\mu(x) \\
    = \sum_{Q \in \partition_n} \mu(Q) \int_X \rho_{Q,n}(x)
      g(y_Q,T^{k}x) h(y_Q,T^{j}x) \diff\mu(x) +
      \bigo(n^{-p}).
  \end{multline}
  Using Condition~\ref{cond:multiple-doc},
  \[
  \begin{split}
    \int_X \rho_{Q,n}(x) g(y_Q,T^{k}x) h(y_Q, & T^{j}x)
    \diff\mu(x) \\
    &= \int_X \rho_{Q,n} \diff\mu \int_X g(y_Q,T^{k}x)
      h(y_Q,T^{j}x) \diff\mu(x) \\
    &\quad + \bigo(\holnorm{\rho_{Q,n}}{\alpha}
    \holnorm{g}{\alpha} \holnorm{h}{\alpha} e^{-\mdoc k}).
  \end{split}
  \]
  Since $\holnorm{g}{\alpha},\holnorm{h}{\alpha} \leqs n^{p}$,
  $\lipnorm{\rho_{Q,n}} \leqs n^{4(d_0+2)p/\alpha}$ and
  $\mu(\rho_{Q,n}) = 1 + \bigo(n^{-p})$, we have
  \begin{multline*}
    \int_X \rho_{Q,n}(x) g(y_Q,T^{k}x) h(y_Q, T^{j}x)
    \diff\mu(x)
    = \int_X g(y_Q,T^{k}x) h(y_Q,T^{j}x) \diff\mu(x) \\
    + \bigo(n^{-p} + n^{4(d_0+3)p/\alpha} e^{-\mdoc k}).
  \end{multline*}
  In view of \eqref{eq:lem:correlation:2}, we therefore have
  \begin{multline*}
    \int_X g(x,T^{k}x) h(x,T^{j}x) \diff\mu(x) \\
    = \sum_{Q \in \partition_n} \mu(Q) \int_X g(y_Q,T^{k}x)
    h(y_Q,T^{j}x) \diff\mu(x) + \bigo(n^{-p} + n^{4(d_0+3)p/\alpha}
    e^{-\mdoc k}),
  \end{multline*}
  Following Step~\ref{step:lem:correlation:1}, we obtain
  \begin{multline*}
    \int_X g(x,T^{k}x)h(x,T^{j}x) \diff\mu(x)
    = \iint g(y,T^{k}x)h(y,T^{j}x) \diff\mu(x)
      \diff\mu(y) \\
    + \bigo(n^{-p} + n^{4(d_0+3)p/\alpha} e^{-\mdoc k}),
  \end{multline*}
  as desired.
\end{proof}

\section{Verification of short returns estimates}
\label{sec:sr:proof}
We devote this section to verifying Conditions~\ref{cond:sr:explicit}
and \ref{cond:sr:implicit} for expanding maps on the interval, and
hyperbolic toral automorphisms. In particular, we verify the relevant
conditions for Examples~\ref{ex:gauss}--\ref{ex:tripling}.
We begin with verifying
\ref{cond:sr:explicit} by slightly modifying the usual approach
by Collet~\cite{collet-2001-statistics}. We include the proof for the
purpose of completeness, and since it is short.
For \ref{cond:sr:explicit}, we modify the approach of Holland, Rabassa
and Sterk~\cite[\text{\S}6.2]{holland-2016-quantitative}, who adapt
Collet's approach to account for singular measures.

\subsection{Proof of
  Condition~\texorpdfstring{\ref{cond:sr:explicit}}{SR} for selected
systems}
\label{sec:srt:proof:explicit}
For this section, we assume that $d\mu = h \diff\leb$ for a density
$h \in \Ell{q}(\leb)$ such that $h \geq c > 0$ for some
$q > 1$ and $c > 0$.
Recall that $\exAv_k = \int \mu(B(x,r_k)) \diff\mu(x)$ and note that
the conditions on $h$ imply
\begin{equation}
\label{eq:rk-Mk}
  r_k^{\dim X} \eqs \leb(B(x,r_k)) \leqs \exAv_k
  \leqs \leb(B(x,r_k))^{1/p} \eqs r_k^{\dim X/p},
\end{equation}
where $p$ is the H{\"o}lder conjugate of $q$.
We also assume the following condition. There exists
$\lsre > 0$ such that, for all $r > 0$,
\begin{equation}
\label{eq:r-explicit:Ekl}
  \mu \{x \in X : d(T^{l}x,x) < r\} \leqs r^{\lsre}
  \quad \text{for all } l \geq 1.
\end{equation}
We state the following result for $\dim X = 1$ for simplicity.

\begin{lemma}
\label{lem:sr:r-explicit}
  In addition to the above assumptions, assume
  $r_k \leqs (\log k)^{-\seqhigh}$ for some
  $\seqhigh > p/\lsre$.
  Then, there exist $\srelog > 1$ and $\srein,\sreout > 0$ such that,
  for all $k \geq 1$ and $l = 1,\ldots,(\log k)^{\srelog}$,
  \[
    \mu \bigl\{x \in X : \mu\bigl( B(x,r_k) \cap T^{-l} B(x,r_k) \bigr)
    > \exAv_k^{1+\srein} \bigr\} \leqs \exAv_k^{\sreout}.
  \]
\end{lemma}

Once we prove Lemma~\ref{lem:sr:r-explicit}, all that remains is to
verify \eqref{eq:r-explicit:Ekl} for the systems we consider.
This has already been done in
\cite{kirsebom-2023-shrinking,holland-2012-extreme}.
More precisely, see \cite[Section~6.2]{kirsebom-2023-shrinking} for 
hyperbolic toral automorphisms equipped with the Lebesgue measure,
and see \cite[Lemmas~3.4, 3.10]{holland-2012-extreme} for
Gibbs--Markov maps and non-uniformly expanding interval maps equipped
with absolutely continuous measures.
In particular, Lemma~\ref{lem:sr:r-explicit} implies that
Examples~\ref{ex:gauss} and \ref{ex:cat} satisfy \eqref{alt} assuming
Lemma~\ref{lem:sr:r-explicit}.

The proof Lemma~\ref{lem:sr:r-explicit} is a slight modification of the
standard argument. We include it since it is short and to serve as a
comparison to the verification of Condition~\ref{cond:sr:implicit}.

\begin{proof}[Proof of Lemma~\ref{lem:sr:r-explicit}]
  Fix $\srelog \in (1, \lsre \seqhigh/p)$.
  Let $\hat{E}_{k,l} = \{x : d(T^{l}x,x) < 2r_k\}$ and notice that
  $\mu(\hat{E}_{k,l}) \leqs r_k^{\lsre} \leqs \exAv_k^{\lsre}$ 
  by \eqref{eq:r-explicit:Ekl} and then \eqref{eq:rk-Mk}. Let
  \[
    \tilde{E}_k \coleq \{x \in X : d(T^{l}x,x) \leq 2 r_k
    \text{ for some } l=1,\ldots,(\log k)^{\srelog} \}
  \]
  and notice that
  $\tilde{E}_k
  \subset \bigcup_{j=1}^{(\log k)^{\srelog}} \hat{E}_{k,l}$.
  Hence, $\mu(\tilde{E}_k) \leqs \exAv_k^{\lsre'}$ for some
  $0 < \lsre' < \lsre - p\srelog/\seqhigh$, where we have used
  $\srelog < \lsre\seqhigh/p$, and
  $\exAv_k \leqs (\log k)^{-\seqhigh/p}$ by \eqref{eq:rk-Mk}.
  Let $\srein < \lsre'$ and define
  \[
    F_k \coleq \bigl\{x \in X : \mu\bigl( B(x,r_k) \cap \tilde{E}_k
      \bigr)
    > \exAv_k^{1+\srein} \bigr\}.
  \]

  Now, define the Hardy--Littlewood maximal function $\hlmax_k$ for
  $\charfun_{\tilde{E}_k}h$ by
  \[
    \hlmax_k(x) \coleq \sup_{r > 0} \frac{1}{\leb(B(x,r))}
    \int_{B(x,r)} \charfun_{\tilde{E}_k} h \diff\leb.
  \]
  Notice that if $x \in F_k$, then
  $\hlmax_k(x) \geqs \leb(B(x,r_k))^{-1} \exAv_k^{1+\srein}
  > C_0 \exAv_k^{\srein}$
  for some constant $C_0 > 0$ by \eqref{eq:rk-Mk}.
  A theorem of Hardy and Littlewood
  (e.g., see \cite[Theorem~1]{stein-1970-singular}) states that
  \[
    \leb \{x : \hlmax_k(x) > c \} \leqs
    \frac{\normsp{\charfun_{\tilde{E}_k}h}{\Ell{1}(\leb)}}{c}
    = \frac{\mu(\tilde{E}_k)}{c}.
  \]
  Consequently,
  \[
    \leb(F_k)
    \leq m \{x : \hlmax_k(x) > C_0 \exAv_k^{\srein} \}
    \leqs \mu(\tilde{E}_k) \exAv_k^{-\srein}
    \leqs \exAv_k^{\lsre' - \srein}
  \]
  and $\mu(F_k) \leqs \leb(F_k)^{1/p} \leqs \exAv_k^{\sreout}$
  for $\sreout = (\lsre' - \srein)/p > 0$.
  Observe that, for all $k \geq 1$ and
  $l = 1,\ldots, (\log k)^{\srelog}$,
  \[
    B(x,r_k) \cap T^{-l}B(x,r_k) \subset B(x,r_k) \cap \tilde{E}_k.
  \]
  Indeed if $y \in B(x,r_k) \cap T^{-l} B(x,r_k)$, then
  $y \in T^{-l}B(y,2r_k)$ and $y \in B(x,r_k) \cap \tilde{E}_k$. Hence,
  \[
    \mu \bigl\{x \in X : \mu\bigl( B(x,r_k) \cap T^{-l} B(x,r_k) \bigr)
    > \exAv_k^{1+\srein} \bigr\}
    \leq \mu(F_k) \leqs \exAv_k^{\sreout},
  \]
  which concludes the proof.
\end{proof}

\subsection{Proof of
  Condition~\texorpdfstring{\ref{cond:sr:implicit}}{SR} for selected
systems}
\label{sec:srt:proof:implicit}
We now verify \ref{cond:sr:implicit} under an analogous `short returns'
estimate to \eqref{eq:r-explicit:Ekl}. Assuming this corresponding
estimate, we will only need that $X = [0,1]$ and $\mu$ is a non-atomic
probability measure. We will then restrict to Gibbs--Markov systems
when verifying this condition.

For a sequence $(M_k) \subset (0,1]$ define $r_k^{L}$ and $r_k^{R}$
such that the intervals $L_k(x) \coleq (x-r_k^{L}(x), x]$ and
$R_k(x) \coleq [x, x+r_k^{R}(x))$ both have measure $M_k$ for
$\mu$-a.e.\ $x \in X$ and all $k \geq 1$. The same argument used to
show that $r_k$ are well-defined $\mu$-a.e.\ and $\lipconst{r_k} = 1$
applies to $r_k^{L}$ and $r_k^{R}$. Finally, define
\[
  E_{k,l} = E_{k,l}(M_k)
  \coleq \{x \in X : T^{l}x \in L_k(x) \cup R_k(x) \}.
\]
The reason for this particular form of $E_{k,l}$ will be made clear in
the proof of the lemma.

\begin{lemma}
\label{lem:sr-implicit}
  Suppose that there exist $C,\lsri > 0$ such that, for all
  sequences $(M_k) \subset (0,1]$ and all $k, l \geq 1$,
  \begin{equation}
  \label{eq:sr-implicit:Ekl:2}
    \mu(E_{k,l}) \leq C M_k^{\lsri}.
  \end{equation}
  Assume now that $(M_k)$ is a sequence such that
  $M_k \leqs (\log k)^{-\seqhigh}$ for some
  $\seqhigh > \lsri^{-1}$.
  Then there exist $\srilog > 1$ and $\sriin,\sriout > 0$ such that, for
  all $k \geq 1$ and $l = 1, \ldots, (\log k)^{\srilog}$,
  \[
    \mu \bigl\{x \in X : \mu\bigl(B(x,r_k(x)) \cap T^{-l} B(x,r_k(x))
    \bigr) > M_k^{1 + \sriin}\bigr\}
    \leqs M_k^{\sriout}.
  \]
\end{lemma}

The proof of Lemma~\ref{lem:sr-implicit} adapts the argument in
\cite[\text{\S}6.2]{holland-2016-quantitative}, which modifies
Collet's approach for measures that are possibly singular with
respect to the Lebesgue measure.

\begin{proof}
  Fix $\srilog \in (1,\lsri\seqhigh)$, let
  \[
    \tilde{E}_k = \{x \in X : d(T^{l}x,x) \in L_k(x) \cup R_k(x)
    \text{ for some } j=1,\ldots,(\log k)^{\srilog} \},
  \]
  and notice that
  $\tilde{E}_k \subset \bigcup_{j=1}^{(\log k)^{\srilog}} E_{k,l}$.
  Hence, by \eqref{eq:sr-implicit:Ekl:2},
  $\mu(\tilde{E}_k) \leqs M_k^{\lsri'}$ for some
  $0 < \lsri' < \lsri - \srilog/\seqhigh$.
  Let $\sriin < \lsri'$ and define
  \[
    F_k \coleq \bigl\{x \in X : \mu\bigl( B(x,r_k(x)) \cap \tilde{E}_k
      \bigr)
    > M_k^{1+\sriin} \bigr\}.
  \]

  Define the maximal function $\hlmax_k$ for $\charfun_{\tilde{E}_k}$
  with respect to the measure $\mu$ by
  \[
    \hlmax_k(x) = \sup_{r > 0} \frac{1}{\mu(B(x,r))}
    \int_{B(x,r)} \charfun_{\tilde{E}_k} \diff\mu.
  \]
  Notice that if $x \in F_k$, then
  $\hlmax_k \geqs \mu(B(x,r_k(x)))^{-1} M_k^{1+\sriin}
  > C_0 M_k^{\sriin}$
  for some constant $C_0 > 0$ since $\mu(B(x,r_k(x))) = M_k$.
  By Fefferman~\cite{fefferman-1981-strong}, we have
  \[
    \mu \{x : \hlmax_k(x) > c \} \leqs
    \frac{\normsp{\charfun_{E_k}}{\Ell{1}(\mu)}}{c}.
  \]
  As a consequence,
  \[
    \mu(F_k) \leq \mu \{x : \hlmax_k(x) > C_0 M_k^{\sriin} \}
    \leqs \mu(\tilde{E}_k) M_k^{-\sriin}
    \leqs M_k^{\lsri' - \sriin}.
  \]

  Let $\sriout = \lsri' - \sriin > 0$. We will now show
  \[
    B(x,r_k(x)) \cap T^{-l} B(x,r_k(x))
    \subset B(x,r_k(x)) \cap \tilde{E}_k.
  \]
  Let $y \in B(x,r_k(x)) \cap T^{-l}B(x,r_k(x))$. Then
  $y,T^{l}y \in B(x,r_k(x))$. The potentially problematic case is when
  $y$ and $T^{l}y$ appear on different sides of $x$, say
  $y \in (x-r_k(x),x]$ and $T^{l}y \in [x, x+r_k(x))$ (the other case
  is treated similarly). Since $B(x,r_k(x))$ and
  $R_k(y) = [y,y+r_k^{R}(y))$ have the same measure and $y \in
  B(x,r_k(x))$, it follows that $[x, x + r_k(x)) \subset R_k(y)$ (see
  Figure~\ref{fig:interval:3}).
  Hence, $T^{l}y \in R_k(y)$ and $y \in B(x,r_k(x)) \cap \tilde{E}_k$.
  In conclusion,
  \[
    \mu \bigl\{x \in X : \mu\bigl( B(x,r_k) \cap T^{-l} B(x,r_k) \bigr)
    > M_k^{1+\sriin} \bigr\}
    \leq \mu(F_k) \leqs M_k^{\sriout},
  \]
  as required.
\end{proof}

\begin{figure}[ht]
  \centering
  \input{figures/srt-double.tex}
  \caption{The intervals $B(x,r_k(x))$ and $R_k(y)$.}
  \label{fig:interval:3}
\end{figure}

All that remains in order to verify Condition~\ref{cond:sr:implicit}
for the systems in Examples~\ref{ex:gauss} and \ref{ex:tripling} is to
verify \eqref{eq:sr-implicit:Ekl:2}. We assume the following.
Suppose that $T \colon [0,1] \to [0,1]$ is a piecewise $C^{1}$
expanding map and that there exists a partition
$\{C_j\}_{j=0}^{\infty}$ of $[0,1]$ (up to an $\leb$-null set) such
that $T|_{C_j} \colon C_j \to [0,1]$ is a bijection for each $j$.
Define the cylinder sets
\[
  C_{j_0,\ldots,j_n} = \bigcap_{i=0}^{n} T^{-i} C_{j_i}.
\]
Suppose further that $\mu$ is a Gibbs measure for a H{\"o}lder
continuous potential $\varphi$ (for existence see
\cite{bowen-2008-equilibrium} for the finite branch setting and
\cite{sarig-2003-existence} for the countable branch setting).
Then $\mu \ll \nu$ for a conformal measure $\nu$ with density $h$
satisfying $c^{-1} \leq h \leq c$ for some $c > 0$. We can assume that
the pressure of $\varphi$ is zero. Thus, there exists $D > 0$ such
that, for all $j_0,\ldots,j_{n-1}$ and $x \in C_{j_0,\ldots,j_{n-1}}$,
\[
  D^{-1} \leq \frac{\mu(C_{j_0,\ldots,j_{n-1}})}{\exp( \varphi_n(x) )}
  \leq D,
\]
where $\varphi_n = \sum_{k=0}^{n-1} \varphi \circ T^{k}$.

\begin{lemma}
\label{lem:sr:implicit:Ekl}
  In addition to the above assumptions, suppose that $|T'| \geq 3$.
  Then there exists $C > 0$ such that, for all sequences
  $(M_k) \subset (0,1]$ and all $k,l \geq 1$,
  \begin{equation}
  \label{eq:sr-implicit:Ekl}
    \mu(E_{k,l}) \leq CM_k.
  \end{equation}
\end{lemma}

\begin{remark}
  In fact, an analysis of the following proof of
  Lemma~\ref{lem:sr:implicit:Ekl} makes it clear that we can relax the
  conditions at the cost of a more complicated formulation. In
  particular, we can alter the hypotheses in the following two ways.
  \begin{enumerate}[label=(\roman*)]
    \item\label{item:rem:sr-implicit:Ekl:1}
      The conclusion holds for all $l \geq 1$ such that
      $|(T^{l})'| \geq 3$. Furthermore, one can relax the condition
      that $|(T^{l})'(x)| \geq 3$ to $(T^{l})'(x) \geq 3$ for all $x$
      such that $(T^{l})'(x) > 0$, with no condition on the magnitude
      of the derivative whenever it is negative. 
    \item\label{item:rem:sr-implicit:Ekl:2}
      Assume that there exist $C_1,C_2,\lsri > 0$ such that, for all
      $x \in X$, there exist $s_2(x) \leq s_1(x)$ with
      $s_2(x)/s_1(x) \geq \lsri$ and, for all $r > 0$,
      \[
        C_1r^{s_1(x)} \leq \mu(B(x,r)) \leq C_2r^{s_2(x)}.
      \]
      Then one can replace the condition $|(T^{l})'| \geq 3$ in
      Lemma~\ref{lem:sr:implicit:Ekl} with $|(T^{l})'| \geq 2$ and
      instead obtain that
      \[
        \mu(E_{k,l}) \leqs M_k^{\lsri}.
      \]
  \end{enumerate}
  The version in \ref{item:rem:sr-implicit:Ekl:1} allows us to conclude
  \eqref{eq:sr-implicit:Ekl} for the Gauss map and $\times n$ maps (see
  Examples~\ref{ex:gauss} and \ref{ex:tripling}).
  The version in \ref{item:rem:sr-implicit:Ekl:2} allows us to conclude
  \eqref{eq:sr-implicit:Ekl} for the doubling map for appropriate Gibbs
  measures (such as $\leb$).
\end{remark}

Before proving Lemma~\ref{lem:sr:implicit:Ekl}, we show that it can be
used to establish \eqref{eq:r-explicit:Ekl} for \emph{any} Gibbs
measure satisfying Condition~\ref{cond:frostman}. In particular,
for any finitely-branched expanding interval map,
\eqref{eq:r-explicit:Ekl} holds for any Gibbs measures corresponding to
a H{\"o}lder continuous potential.
To the author's knowledge, this is the first such result for
non-absolutely continuous Gibbs measures for expanding interval
systems.

\begin{corollary}
  In addition to the above assumptions on $([0,1], T, \mu)$,
  suppose $\mu$ satisfies Condition~\ref{cond:frostman}. Then there
  exists $c,\lsre > 0$ such that
  \[
    \mu \{x \in X : d(T^{l}x,x) < r\} \leq c r^{\lsre}
    \quad \text{for all } r > 0,\ l \geq 1.
  \]
\end{corollary}

\begin{proof}
  Assume that $\mu$ satisfies Condition~\ref{cond:frostman}.
  That is, there exist $C_0,\frost > 0$ such that, for any sequence
  $(M_k)$, we have $M_k \leq C_0 r_k^{L}(x)^{\frost}$ and
  $M_k \leq C_0 r_k^{R}(x)^{\frost}$
  for all $x \in [0,1]$ and $k \geq 1$.
  For $r > 0$ sufficiently small, let $M_k = C_0 r^{\frost}$ for all
  $k \geq 1$. Then, for all $k,l \geq 1$,
  \[
    \mu \{x : d(T^{l}x,x) < r \}
    = \mu \{x : d(T^{l}x,x) \in L_k(x) \cup R_k(x)\}
    \leq C M_k \leqs r^{\frost},
  \]
  where the implicit constant is universal. Letting $\lsre = \frost$,
  we conclude the proof.
\end{proof}

We now prove Lemma~\ref{lem:sr:implicit:Ekl}. We adapt the arguments
used in the proof of \cite[Lemma~6]{persson-2026-strong}.

\begin{proof}[Proof of Lemma~\ref{lem:sr:implicit:Ekl}]
  Fix $k,l \geq 1$ and $j_0,\ldots,j_{l-1}$, and let
  $I = I_{j_0,\ldots,j_{l-1}}
  \coleq E_{k,l} \cap C_{j_0,\ldots,j_{l-1}}$.
  Let $p$ denote the unique periodic point $p \in I$ of period $l$.
  By the continuity of $T|_{C_{j_0,\ldots,j_{l-1}}}$, we have
  $I = (a, b)$ for some $a < p < b$.
  We will first establish $\mu(T^{l}I) \leqs M_k$, where the
  implicit constant is independent of the sequence $(M_k)$.
  We consider two cases.

  \case[$(T^{l})' > 0$ on $C_{j_0,\ldots,j_{l-1}}$]
  Since $T$ is expanding with positive derivative,
  $T^{l}a < a < p < b < T^{l}b$ and $T^{l}I = (T^{l}a,T^{l}b)$.
  We focus on estimating $\mu( (T^{l}a,p] )$ since the corresponding
  argument for $\mu( [p,T^{l}b) )$ is similar.
  By the continuity of $r_k^{L}$, we have
  $d(T^{l}a,a) = r_k^{L}(a)$ and
  \[
    d(T^{l}a,p) = d(T^{l}a, a) + d(a,p) = r_k^{L}(a) + d(a,p).
  \]
  At the same time, since $(T^{l})'$ is continuous, there exists
  $a < \xi < p$ such that
  \[
    d(T^{l}a,p) = d(T^{l}a,T^{l}p) = (T^{l})'(\xi) d(a,p).
  \]
  As a consequence,
  $((T^{l})'(\xi) - 1)d(a,p) = r_k^{L}(a) \leq r_k^{L}(p) + d(a,p)$,
  where we have used $\lipconst{r_k^{L}} = 1$. Therefore
  \[
    ((T^{l})'(\xi) - 2)d(a,p) \leq r_k^{L}(p).
  \]
  Using $(T^{l})' \geq 3$, we have $d(a,p) \leq r_k^{L}(p)$
  and therefore (see Figure~\ref{fig:interval:1})
  \[
    (T^{l}a,p] = (T^{l}a, a] \cup (a,p] \subset L_k(a) \cup L_k(p)
  \]
  and $\mu((T^{l}a,p]) \leq 2 M_k$.
  \begin{figure}[ht]
    \centering
    \input{figures/srt-case-1.tex}
    \caption{The positions of $p,a$ and $T^{l}a$ when $(T^{l})' > 0$.}
    \label{fig:interval:1}
  \end{figure}

  \case[$(T^{l})' < 0$ on $C_{j_0,\ldots,j_{l-1}}$]
  In this case, $T^{l}I = (T^{l}b, T^{l}a)$, where
  $T^{l}a = a + r_k^{R}(a)$ and $T^{l}b = b - r_k^{L}(b)$.
  This implies (see Figure~\ref{fig:interval:2})
  \[
    T^{l}I = (T^{l}b, T^{l}a) \subset (T^{l}b, b) \cup (a, T^{l}a)
    = L_k(b) \cup R_k(a)
  \]
  and $\mu(T^{l}I) \leq 2 M_k$.
  \begin{figure}[ht]
    \centering
    \input{figures/srt-case-2.tex}
    \caption{The interval $T^{l}I$ when $(T^{l})' < 0$.}
    \label{fig:interval:2}
  \end{figure}

  We now prove that $\mu(E_{k,l}) \leqs M_k$.
  Using that $\nu$ is conformal and $\mu$ is a Gibbs measure, we have
  \[
    \nu(T^{l}I_{j_0,\ldots,j_{l-1}}) = \int_{I_{j_0,\ldots,j_{l-1}}}
    e^{-\varphi_l(x)} \diff\nu(x)
    \geqs \mu(C_{j_0},\ldots,j_{l-1})^{-1} \nu(I_{j_0,\ldots,j_{l-1}}).
  \]
  Consequently, using that $d\mu = h \diff\nu$ with
  $c^{-1} \leq h \leq c$, we have
  \[
  \begin{split}
    \mu(I_{j_0,\ldots,j_{l-1}})
    &\leqs \nu(I_{j_0,\ldots,j_{l-1}}) \\
    &\leqs \nu(T^{l}I_{j_0,\ldots,j_{l-1}}) \mu(C_{j_0,\ldots,j_{l-1}})
    \\
    &\leqs \mu(T^{l}I_{j_0,\ldots,j_{l-1}}) \mu(C_{j_0,\ldots,j_{l-1}})
    \\
    &\leqs M_k\mu(C_{j_0,\ldots,j_{l-1}}).
  \end{split}
  \]
  We conclude the proof by summing over $j_0,\ldots,j_{l-1}$.
\end{proof}

\subsection*{Acknowledgements}
I am deeply grateful to my supervisor, Tomas Persson, for the
insightful discussions and his careful review of numerous drafts, and
in particular for the idea of using the left and right intervals, $L_k$
and $R_k$, in Section~\ref{sec:srt:proof:implicit}.

\end{document}

%% file: figures/srt-double.tex
\begin{tikzpicture}
  \draw (-5.5,0) -- (5.5,0);
  \draw (0,0) node[below=4pt] {$x$};
  \ptmark{0};
  \draw (-2.5,0) node[below=4pt] {$y$};
  \ptmark{-2.5};
  \draw (3.5,0) node[below=4pt] {$T^{l}y$};
  \ptmark{3.5};
  \draw[|-|] (-2.5,-0.9) --node[below=0pt] {$r_k^{R}(y)$}
    (5.0,-0.9);
  \draw[|-|] (0,0.4) --node[above=0pt] {$r_k(x)$} (4.0,0.4);
  \draw[|-|] (-4.0,0.4) --node[above=0pt] {$r_k(x)$} (0,0.4);
\end{tikzpicture}

%% file: figures/srt-case-1.tex
 \begin{tikzpicture}
   \draw (-5.5,0) -- (5.5,0);
   \draw (0,0) node[below=4pt] {$a$};
   \ptmark{0};
   \draw (-3.0,0) node[below=4pt] {$T^{l}a$};
   \ptmark{-3.0};
   \draw (3.5,0) node[below=4pt] {$p$};
   \ptmark{3.5};
   \draw[|-|] (-0.5,0.8) --node[above=0pt] {$r_k^{L}(p)$} (3.5,0.8);
   \draw[|-|] (-3.0,0.4) --node[above=0pt] {$r_k^{L}(a)$} (0,0.4);
 \end{tikzpicture}

%% file: figures/srt-case-2.tex
\begin{tikzpicture}
  \draw (-5.5,0) -- (5.5,0);
  \draw (0,0) node[below=4pt] {$p$};
  \ptmark{0};
  \draw (-2.0,0) node[below=4pt] {$a$};
  \ptmark{-2.0};
  \draw (2.5,0) node[below=4pt] {$b$};
  \ptmark{2.5};
  \draw (-4.5,0) node[below=4pt] {$T^{l}b$};
  \ptmark{-4.5};
  \draw (4.0,0) node[below=4pt] {$T^{l}a$};
  \ptmark{4.0};
  \draw[|-|] (-4.5,-0.8) --node[below=0pt] {$r_k^{L}(b)$}
    (2.5,-0.8);
  \draw[|-|] (-2.0,0.4) --node[above=0pt] {$r_k^{R}(a)$}
    (4.0,0.4);
\end{tikzpicture}